\documentclass[10pt,oneside]{article}
\usepackage[T1]{fontenc}
\usepackage[english]{babel}

\usepackage{graphicx,accents}

\newcommand{\cev}[1]{\reflectbox{\ensuremath{\vec{\reflectbox{\ensuremath{#1}}}}}}

\usepackage[cal=cm]{mathalfa}

\usepackage[T1]{fontenc}
\usepackage{titlesec}

\usepackage{hhline}

\titleformat
{\chapter} 
[display] 
{\bfseries\Large\itshape} 
{Story No. \ \thechapter} 
{0.5ex} 
{
    \rule{\textwidth}{1pt}
    \vspace{1ex}
    \centering
} 
[
\vspace{-0.5ex}%
\rule{\textwidth}{0.3pt}
] 

\titleformat{\section}[wrap]
{\normalfont\bfseries}
{\thesection.}{0.5em}{}

\titlespacing{\section}{12pc}{1.5ex plus .1ex minus .2ex}{1pc}

\usepackage{tocloft} 

\usepackage{titlesec}
\titleformat{\subsection}[runin]
       {\normalfont\bfseries}
       {\thesubsection}
       {0.5em}
       {}
       [.]

\makeatletter
\@ifundefined{dddot}{}{}
\usepackage{mathdots}  
\makeatother

\makeatletter
\@ifundefined{ddddot}{}{}
\usepackage{yhmath}
\makeatother

\usepackage{geometry} 
\usepackage{xcolor} 
\usepackage{tikz} 
\usepackage[draft=false]{hyperref} 
\usepackage{mathtools}
\usepackage{amsmath,amsthm,amssymb}
\usepackage[all,cmtip]{xy}

\usepackage{fullpage}

\usepackage[capitalise]{cleveref}  
\newcommand{\clevertheorem}[3]{%
	\newtheorem{#1}[thm]{#2}
	\crefname{#1}{#2}{#3}
}
\numberwithin{equation}{section} 
\numberwithin{figure}{section} 

\theoremstyle{plain} 
\newtheorem{thm}{Theorem}[subsection]
\crefname{thm}{Theorem}{Theorems}
\newtheorem*{thm*}{Theorem}
\clevertheorem{proposition}{Proposition}{Propositions}
\clevertheorem{prop}{Proposition}{Propositions}
\newtheorem*{prop*}{Proposition}
\clevertheorem{theorem}{Theorem}{Theorems}
\clevertheorem{lemma}{Lemma}{Lemmas}
\clevertheorem{corollary}{Corollary}{Corollaries}
\clevertheorem{cor}{Corollary}{Corollaries}
\clevertheorem{conj}{Conjecture}{Conjectures}
\clevertheorem{questn}{Question}{Questions}

\theoremstyle{definition} 
\clevertheorem{definition}{Definition}{Definitions}
\clevertheorem{assumption}{Assumption}{Assumptions}
\clevertheorem{defn}{Definition}{Definitions}
\clevertheorem{notn}{Notation}{Notations}
\clevertheorem{conv}{Convention}{Conventions}

\clevertheorem{remark}{Remark}{Remarks}
\clevertheorem{rmk}{Remark}{Remarks}
\clevertheorem{warn}{Warning}{Warnings}
\clevertheorem{example}{Example}{Examples}
\clevertheorem{summ}{Summary}{Summaries}

\usepackage{amsmath,amsthm,amssymb,amsfonts,extarrows}
\usepackage{enumerate,amscd,amsxtra,MnSymbol}
\usepackage{mathrsfs}
\usepackage{color}
\usepackage[cmtip,all]{xy}
\usepackage{eucal}
\usepackage{bbold}
\usepackage{upgreek}
\usepackage{paralist}
\usepackage{tikz-cd}
\usepackage{dsfont}
\usepackage{bbm}
\usepackage[bbgreekl]{mathbbol}

\DeclareMathSymbol\bbDelta \mathord{bbold}{"01}
\DeclareMathSymbol\bDelta \mathord{bbold}{"01}

\newtheorem{remark*}{Remark}

\newtheorem{notation}[thm]{Notation}

\newcommand{\bD}{{\mathbb D}}

\renewcommand{\P}{{\mathbb P}}

\newcommand{\mA}{{\mathcal A}}
\newcommand{\mB}{{\mathcal B}}
\newcommand{\mC}{{\mathcal C}}
\newcommand{\mD}{{\mathcal D}}

\newcommand{\mJ}{{\mathcal J}}
\newcommand{\mK}{{\mathcal K}}
\newcommand{\mL}{{\mathcal L}}
\newcommand{\mM}{{\mathcal M}}
\newcommand{\mN}{{\mathcal N}}
\newcommand{\mO}{{\mathcal O}}
\newcommand{\mP}{{\mathcal P}}
\newcommand{\mQ}{{\mathcal Q}}
\newcommand{\mR}{{\mathcal R}}
\newcommand{\mS}{{\mathcal S}}

\newcommand{\mU}{{\mathcal U}}
\newcommand{\mV}{{\mathcal V}}
\newcommand{\mW}{{\mathcal W}}
\newcommand{\mX}{{\mathcal X}}
\newcommand{\mY}{{\mathcal Y}}

\newcommand{\A}{A}
\newcommand{\B}{B}
\newcommand{\C}{C}

\newcommand{\E}{{E}}
\newcommand{\F}{{F}}
\newcommand{\G}{{G}}

\renewcommand{\L}{{\mathrm L}}
\newcommand{\M}{{M}}
\newcommand{\N}{{\mathrm N}}
\renewcommand{\P}{{P}}
\newcommand{\Q}{{Q}}
\newcommand{\R}{{\mathrm R}}

\newcommand{\T}{{T}}

\newcommand{\W}{W}
\newcommand{\X}{X}
\newcommand{\Y}{Y}
\newcommand{\Z}{Z}

\newcommand{\bj}{{j}}
\newcommand{\bi}{{i}}
\newcommand{\m}{{m}}
\newcommand{\bk}{{k}}

\newcommand{\g}{{g}}
\newcommand{\n}{{n}}

\newcommand{\op}{\mathrm{op}}

\newcommand{\colim}{\mathrm{colim}}

\newcommand{\rev}{{\mathrm{rev}}}

\newcommand{\ot}{\otimes}

\newcommand{\co}{\mathrm{co}}

\newcommand{\univ}{\mathrm{univ}}
\newcommand{\strict}{\mathrm{strict}}

\newcommand{\Gpd}{\mathrm{Gpd}}

\newcommand{\id}{\mathrm{id}}
\newcommand{\Cat}{\mathrm{Cat}}

\newcommand{\Set}{\mathrm{Set}}

\newcommand{\Alg}{\mathrm{Alg}}

\newcommand{\Mon}{\mathrm{Mon}}
\newcommand{\Fun}{\mathrm{Fun}}

\newcommand{\Cocart}{{\mathrm{coCart}}}

\newcommand{\lax}{{\mathrm{lax}}}
\newcommand{\oplax}{{\mathrm{oplax}}}

\newcommand{\tu}{{\mathbb 1}}
\newcommand{\coop}{\mathrm{coop}}

\newcommand{\Ind}{{\mathrm{Ind}}}

\newcommand{\Map}{{\mathrm{Map}}}

\newcommand{\bZ}{{\mathbb{Z}}}

\newcommand{\Mor}{{\mathrm{Mor}}}

\newcommand{\PrL}{\mathrm{Pr^L}}
\newcommand{\PrR}{\mathrm{Pr^R}}

\newcommand{\cop}{\mathrm{cop}}
\newcommand{\cofib}{\mathrm{cofib}}

\newcommand{\fib}{\mathrm{fib}}

\newcommand{\red}{\mathrm{red}}
\newcommand{\map}{\mathrm{Map}}
\newcommand{\LMor}{\mathrm{LMor}}
\newcommand{\RMor}{\mathrm{RMor}}
\newcommand{\Seg}{\mathrm{Seg}}

\newcommand{\cube}{{\,\vline\negmedspace\square}}

\newcommand{\scat}{\mathcal{C}\mathit{at}}

\newcommand{\fcat}{\mathfrak{Cat}}

\newcommand{\vertrule}[1][1ex]{\rule{.4pt}{#1}}

\newcommand{\abd}{{\,\,\vertrule{}\!\!\bar{\Diamond}}}

\begin{document}



\title{\textsc{Oriented Category Theory}}

\author{David Gepner and Hadrian Heine}

\maketitle

\begin{abstract}

As categorical dimension increases, classical categorical concepts often become too rigid and must be replaced by appropriately lax analogues. A basic manifestation of this principle is that higher-categorical versions of familiar geometric constructions --- such as cylinders, cones, and suspensions --- are not functorial in the classical sense. To remedy this situation, we
extend the theory of higher categories to a framework of so-called oriented and antioriented categories, which provides a unified formalism for lax and oplax phenomena and reveals the geometric nature inherent in higher category theory.
We characterize (anti)oriented categories as deformations of higher categories in which the various compositions only commute up to coherent (anti)oriented interchange law, and provide a geometric presentation of (anti)oriented categories as sheaves on a suitable thickening of the simplex category.
We also construct an embedding of the category of $(\infty,\infty)$-categories into the category of (anti)oriented categories and characterize the image as those $(\infty,\infty)$-categories satisfying a strict interchange law.


\end{abstract}

\tableofcontents

\vspace{5mm}

\section{\mbox{Introduction}}

\subsection{Higher categorical mathematics}

For many questions in modern mathematics, it has become convenient to replace classical set-based mathematics with homotopical space-based mathematics.
Still more recently, and in part due to applications in mathematical physics and topological quantum field theory, the need has arisen for a categorically-based formalism in which higher categories themselves are the basic geometric objects, and can be manipulated via lax variants of standard homotopical constructions.
In categorical mathematics, unlike homotopical mathematics, paths need no longer be reversible, but are instead oriented from source to target.
Like in homotopical mathematics, it is necessary to have paths between paths, etc., forcing us to consider $\infty$-categories instead of just $1$-categories.\footnote{We work in a natively homotopical context throughout: a $0$-category is automatically understood to be an $(\infty,0)$-category, or homotopy type, and inductively an $n$-category is an $(\infty,n)$-category, etc. We refer to the non-homotopical versions as strict $n$-categories, etc.}.
Hence the basic geometric objects are comprised of not necessarily invertible cells of potentially arbitrary high dimension, which can be composed when the outgoing boundary of one agrees with the incoming boundary of another.

The following table depicts how various fundamental concepts extend from classical to higher mathematics.
The leftmost column lists certain fundamental structures in the strict, classical world, the middle column lists their generalization to the groupoidal, homotopical world,\footnote{We have opted for the most universal generalization, though others are possible. For instance, connective spectra or the derived category of $\bZ$ are other useful generalizations.} and the rightmost column lists their generalization to the higher categorical world.
$$\begin{tabular}{| c | c | c |}
\hline
Classical & Homotopical & Categorical\\
\hhline{|=|=|=|}	
Sets & Spaces & $\infty$-Categories \\
\hline
Groups & Loop Spaces & Endomorphism $\infty$-Categories \\
\hline
Abelian Groups & Spectra & Categorical Spectra \\
\hline
Set-enriched Categories & Space-enriched Categories & Gray-tensor-product-enriched Categories\\
\hline
\end{tabular}$$

The purpose of this paper is to study the category of $\infty$-categories from the geometric perspective of oriented spaces, which is a monoidal category under the Gray tensor product, and develop the theory of categories enriched in oriented spaces.
This geometric approach to higher category theory, generalizing Grothendieck's homotopy hypothesis and developed in work of \cite{campion2023graytensorproductinftyncategories}, \cite{gepner2026homotopy}, \cite{GepnerHeine2026}, is inherent in the Baez--Dolan cobordism and tangle hypotheses, and has been developed in works of a number of authors. See also the work of Ayala--Francis--Rozenblyum on the stratified homotopy hypothesis \cite{MR3941460}.
In particular, we establish several geometric formulas about oriented spaces which have significant categorical consequences.

In addition, we develop a theory of bioriented categories, which are categories that carry a compatible orientation and antiorientation. This theory is built on the theory of bienriched categories developed in \cite{heine2024bienriched}. Bienrichment is a variant of the usual notion of enrichement in nonsymmetric situations, where one wants to keep track of the enrichment in both the monoidal and the reverse monoidal structures simultaneously.

The fundamental example of a bioriented category is the category of $\infty$-categories enriched in the oplax and lax Gray tensor product.
The stable counterpart of this bioriented category is the bioriented category of categorical spectra \cite{heine2025categorification} \cite{heine2026stable}, \cite{Masuda}.
Its biorientation arises from the biorientation on (based) $\infty$-categories together with the biorientation of the reduced suspension functor and its adjoint, the endomorphisms functor.
The biorientation of categorical spectra is necessary in order to formulate its universal property in terms of bioriented excisive functors \cite[Proposition 3.57]{heine2025categorification}.

The various categorical hierarchies which appear in this paper are depicted in the following diagram of fully faithful inclusions:
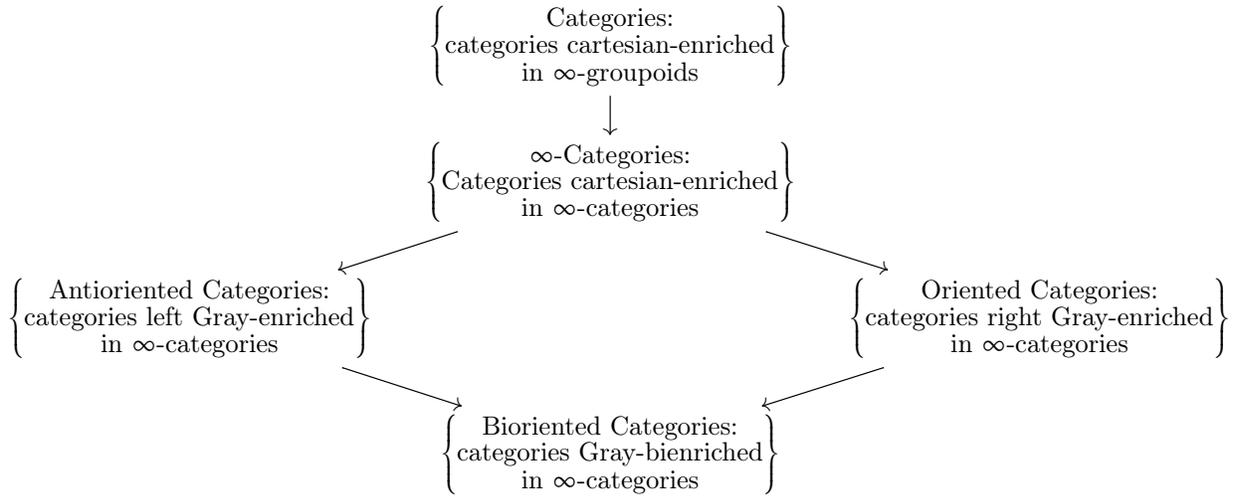
\[
\xymatrixrowsep{.2in}
\xymatrixcolsep{.2in}
\xymatrix{
& \left\{\overset{\text{{\normalsize{Categories:}}}}{\underset{\text{{\normalsize{in $\infty$-groupoids}}}}{\text{categories cartesian-enriched}}}\right\}\ar[d] &\\
& \left\{\overset{\text{{\normalsize{$\infty$-Categories:}}}}{\underset{\text{{\normalsize{in $\infty$-categories}}}}{\text{Categories cartesian-enriched}}}\right\}\ar[ld]\ar[rd] &\\
\left\{\overset{\text{{\normalsize{Antioriented Categories:}}}}{\underset{\text{{\normalsize{in $\infty$-categories}}}}{\text{categories left Gray-enriched}}}\right\}\ar[rd]
& & \left\{\overset{\text{{\normalsize{Oriented Categories:}}}}{\underset{\text{{\normalsize{in $\infty$-categories}}}}{\text{categories right Gray-enriched}}}\right\}\ar[ld] &\\
& \left\{\overset{\text{{\normalsize{Bioriented Categories:}}}}{\underset{\text{{\normalsize{in $\infty$-categories}}}}{\text{categories Gray-bienriched}}}\right\} &
}
\]
As in bimodule theory, bienriched categories can be viewed as categories enriched in the tensor product of the left and right enriching monoidal categories.

$$\begin{tabular}{| c | c | c |}
\hline
Categorical Structure & Notation & Morphism Objects\\
\hhline{|=|=|=|}
$\infty$-Categories & $\infty\scat$ & $\Fun(-,-)$ \\
\hline
Oriented Spaces & $\infty\fcat|_{\boxtimes}$ & $\Fun^{\lax}(-,-)$ \\
\hline
Antioriented Spaces & $_{\boxtimes}|\infty\fcat$ & $\Fun^{\oplax}(-,-)$ \\
\hline
Bioriented Spaces & $\infty\fcat$ & $(X,Y)\mapsto\Map_{\infty\Cat}(X\boxtimes -\boxtimes Y,-)$ \\
\hline
Oriented Categories & $\Cat\boxtimes$ & $\Fun\boxtimes(-,-)$\\
\hline
Antioriented Categories & $\boxtimes\Cat$ & $\boxtimes\Fun(-,-)$\\
\hline
Bioriented Categories & $\boxtimes\Cat\boxtimes$ & $\boxtimes\Fun\boxtimes(-,-)$\\
\hline
\end{tabular}$$

\subsection{Higher dimensional category theory}
In geometry, the cylinder on an $n$-dimensional manifold with boundary is an $n+1$-dimensional manifold with boundary.
The interaction of the cartesian product with geometric dimension crucially underlies many notions and constructions in geometry.

There is an analogous and equally fundamental notion of dimension in higher category theory.
However, naively interpreted, the cartesian product of an $m$-dimensional category and an $n$-dimensional category is {\em never} an $m+n$-dimensional category unless one of $m$ or $n$ is zero.
Hence, the basic geometry of higher categories only becomes apparent once we replace the cartesian product with a refinement which behaves additive with respect to categorical dimension. The correct replacement is the Gray tensor product, and was originally suggested by Gray \cite{GRAY197663} and further studied by Steiner \cite{Steiner2004OmegacategoriesAC} in the strict context.
There has been considerable work on the Gray tensor product in the context of bicategory theory \cite{MAEHARA2021107461}, \cite{gagna2021gray}, and it was recently constructed in full generality by Campion \cite{campion2023graytensorproductinftyncategories}.

However, it is not enough simply to use the Gray tensor product as an available operation.
Instead, it must be injected directly into the foundations of the theory itself.
The goal of this paper is study the deformation of higher category theory in which morphism objects compose not according the to cartesian product but directly via the Gray tensor product.
The Gray tensor product is not symmetric but rather antisymmetric: for every $\infty$-categories $X,Y$ there is a canonical equivalence 
\[
X\bar{\boxtimes} Y:=Y\boxtimes X\simeq (X^\co\boxtimes Y^\co)^\co
\]
Due to the antisymmetry of the Gray tensor product, the categories of our framework always arise in two flavors, namely left-handed and right-handed versions, and we must keep track of these orientations throughout.

Sets are evidently $0$-dimensional geometric objects, as are $\infty$-groupoids, since colimits of $n$-dimensional objects remain $n$-dimensional.
Nevertheless, higher dimensional categories can be built out of lower dimensional categories, provided one uses a version of the colimit which is compatible with categorical dimension.
In bicategory theory, this can be accomplished via the usual theory of lax and oplax colimits \cite{articles}, though in higher dimensions the situation is more subtle.\footnote{In a subsequent paper we develop the theory of {\em oriented} limits and colimits.
Surprisingly, it is not the same as the lax limit or colimit, the issue being that the tensor with $\mC$, or $\mC$-indexed coproduct, of an $\infty$-category $\mD$ is $\mC\times\mD$, whereas it must be the Gray tensor product $\mC\boxtimes\mD$ is order to be compatible with categorical dimension.}
We require only the rudiments of such a theory here, and will be able to get by with iterated oriented pushouts and pullbacks.\footnote{We warn the reader that oriented pullbacks and pushouts are indistinguishable from their lax variants in bicategory theory, as the difference is only visible in dimensions three and above.}
The most basic example of a $0$-dimensional category, or an $\infty$-groupoid, is the $0$-disk $\bD^0$, the terminal $\infty$-category, playing the role of the point.
Higher dimensional disks are obtained by suspending lower dimensional disks; specifically,
the $n$-dimensional disk $\bD^n$ is the $n$-fold suspension of $\bD^0$, an $n$-category with a unique ordered pair of $m$-cells in all dimensions $m<n$ and a single nondegenerate $n$-cell.
The $n$-disk $\bD^n$ is the most basic $n$-dimensional category, as it is free on a single $n$-cell, and any $n$-category can be built out of $\bD^n$ via iterated colimits, and an $n$-dimensional cell of an $\infty$-category is specified by a map from the $n$-disk.

In practice, it is essential to be able to perform the suspension operation.
Hence we must work at the level of $\infty$-categories,\footnote{We refer to $(\infty,n)$-categories simply as $n$-categories, and $(\infty,\infty)$-categories simply as $\infty$-categories.}
which we define as the limit
\[
\infty\Cat=\lim\{\cdots\to (n+1)\Cat\to n\Cat\to (n-1)\Cat\to\cdots\to 0\Cat\}
\]
taken along the core functors, which forget the top-dimensional morphisms
\footnote{One could also take the limit along the truncation functors, which invert the top-dimensional morphisms, but this results in the full subcategory of $\infty\Cat$ consisting of the Postnikov complete objects \cite{gepner2026homotopy}.
This full subcategory is not appropriate for our purposes since for instance the infinite cobordism $\infty$-category and its classifying space have equivalent Postnikov completions}.
As a bifunctor
\[
\boxtimes:\infty\Cat\times\infty\Cat\to\infty\Cat
\]
the Gray tensor product is biclosed.
For $\infty$-categories $X$, $Y$, $Z$ there are canonical equivalances of spaces
\[
\Map_{\infty\Cat}(Y,\Fun^{\lax}(X,Z))\simeq\Map_{\infty\Cat}(X\boxtimes Y,Z)\simeq\Map_{\infty\Cat}(X,\Fun^{\oplax}(Y,Z))
\]
which are natural, provided we view $\infty\Cat$ only as a category.
The associativity of the Gray tensor product makes $\infty\Cat$ enriched in the Gray tensor product, which we denote by $\infty\mathfrak{Cat}$.


\subsection{The failure of the cartesian product}
One of the main subtleties arising in higher category theory is that dimension shifting operations, such as the suspension, are typically no functors of the $\infty$-category of $\infty$-categories, but only of the underlying $1$-category of $\infty$-categories.
An endofunctor of the $1$-category of $\infty$-categories refines to a functor of the $\infty$-category of $\infty$-categories precisely when it is compatible with the cartesian enrichment of $\infty\Cat$ over itself, and operations involving the Gray tensor product are not usually compatible with the cartesian monoidal structure, but rather with the Gray monoidal structure.

Said differently, while a functor of $\infty$-categories corresponds to an object of the $\infty$-category $\Fun(\bD^1,\infty\Cat)$, a transformation of functors cannot be determined by a morphism of the $\infty$-category $\Fun(\bD^1,\infty\Cat)$, i.e. a {\em strictly} commuting square $\bD^1\times\bD^1\to\infty\Cat$.
Instead, it is an instance of a {\em laxly} or {\em oplaxly} commuting square $\bD^1\boxtimes\bD^1\to\infty\Cat$ (depending on which copy of $\bD^1$ corresponds to the original functors).
This is because a transformation is actually specified by a functor $\bD^2\to\infty\Cat$, but $\bD^2$ is only a quotient of $\bD^1\boxtimes\bD^1$ by either $(\partial\bD^1)\boxtimes\bD^1$ (in the oplax case) or $\bD^1\boxtimes(\partial\bD^1)$ (in the lax case), and not of $\bD^1\times\bD^1$.
Both quotients admit a common section, namely the inclusion of the $2$-cell $\bD^2\to\bD^1\boxtimes\bD^1$, and the strictly commuting square is obtained by gluing $\bD^1$ to $\bD^1\boxtimes\bD^1$ along the map $\bD^2\to\bD^1$ which collapses the $2$-cell:
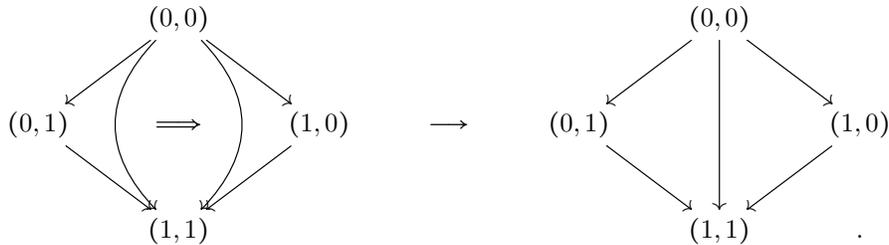
\[
\xymatrix{
& (0,0)\ar[rd]\ar[ld]\ar@/^2pc/[dd]\ar@/_2pc/[dd] & &  & & (0,0)\ar[rd]\ar[ld]\ar[dd] &\\
(0,1)\ar[rd] & \Longrightarrow & (1,0)\ar[ld] & \longrightarrow & (0,1)\ar[rd]& & (1,0)\ar[ld]\\
& (1,1) & & & & (1,1) &.
}
\]
From this perspective we see that the Gray tensor product can be viewed as an laxification of the cartesian product which is compatible with categorical dimension.
Note that this phenomenon is invisible in dimension one, showing that intuition from ordinary category theory often fails in higher dimensions.

A related issue can be seen in the suspension and morphism adjunction.
If $X$ is a space equipped with a pair of basepoints $s$ and $t$,
the space of paths from $s$ to $t$ is computed as the homotopy pullback
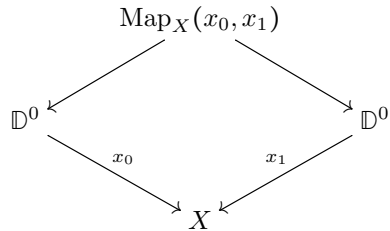
\[
\xymatrix{
& \Map_X(x_0,x_1)\ar[rd]\ar[ld] &\\
\bD^0\ar[rd]^{x_0} & & \bD^0\ar[ld]_{x_1}\\
& X &
}
\]
In order for something similar to work for categories, we must replace the homotopy pullback with the {\em oriented pullback}.
If $X$ is an $\infty$-category equipped with an ordered pair of objects $(s,t)$, then these is indeed an analogous formula for the $\infty$-category $\Mor_X(x_0,x_1)$ of morphisms in $X$ from $s$ to $t$, which we might depict as a square
\[
\xymatrix{
& \Mor_X(s,t)\ar[rd]\ar[ld] &\\
\bD^0\ar[rd]^{s} & \Longrightarrow & \bD^0\ar[ld]_{t}\\
& X &
}
\]
which only commutes up to a $2$-cell.
To make this precise, we show that the suspension and morphism object constructions form an adjunction of antioriented categories 
\[
S:\infty\fcat\to\infty\fcat_{\partial\bD^1/}.
\]
Thus for $\infty$-categories $X$, $Y$, and an ordered pair $(s,t)$ of $Y$, there is a natural equivalence
\[
\Fun^\oplax_{\partial\bD^1/}(S(X),Y) \simeq \Fun^\oplax(X,\Mor_Y(s,t)).
\]
Note that with respect to the cartesian enrichment, any equivalence
\[
\Fun_{\partial\bD^1/}(S(X),Y) \simeq \Fun(X,\Mor_Y(s,t)).
\]
would specialize, even in the case $X=\bD^0$, to an equivalence
\[
\Fun_{\partial\bD^1/}(\bD^1,(Y;s,t)) \simeq \Fun(*,\Mor_Y(s,t)) \simeq \Mor_Y(s,t),
\]
which is necessarily an $\infty$-groupoid since
the restriction functor $\Fun(\bD^1,Y) \to \Fun(\partial\bD^1,Y)\simeq Y \times Y$ is conservative and therefore its fibers are spaces.

Finally, consider the cylinder on an $\infty$-category $X$, the {\em oriented pushout} of the span $X\overset{\id}{\leftarrow} X \overset{\id}{\rightarrow} X$, computed as the Gray tensor product $ X \boxtimes \bD^1$.
This is only an {\em oriented functor}, and not a functor of $\infty$-categories.
Indeed, part of the data of a functor of $\infty$-categories is a functor $\Fun(X,Y)\to\Fun(FX,FY)$.
In particular, there is a natural functor
\[
X\to\Fun(Y,X\times Y)\to\Fun(FY,F(X\times Y)),
\]
which is adjoint to a natural functor $X\times FY\to F(X\times Y)$.
Taking $X=\bD^1$ and $Y=\bD^0$, then such a structure for $F(X)=X\boxtimes \bD^1$ would require a functor $\bD^1\times\bD^1\to\bD^1\boxtimes\bD^1$.
This functor must be an isomorphism on objects by naturality: consider the two inclusions $\bD^0\to\bD^1$, and observe that the functor $\bD^0\times F(\bD^0)\to F(\bD^0\times\bD^0)$ must be an isomorphism.
But there is no functor $\bD^1\times\bD^1\to\bD^1\boxtimes\bD^1$ which is an isomorphism on objects.

\subsection{Oriented category theory}
Oriented categories are categories enriched in oriented spaces, or equivalently, categories enriched in $\infty\Cat$ with respect to the lax Gray tensor product.
Dually, there are antioriented categories, which are categories enriched in $\infty\Cat$ with respect to the oplax Gray tensor product.
There are $2$-categories of oriented and antioriented categories, denoted
\[
{\Cat\boxtimes}\qquad\textrm{and}\qquad{\boxtimes\Cat}
\]
respectively, in which the objects are the (anti)oriented categories, the 1-morphisms are the (anti)oriented functors, and the $2$-morphisms are the enriched transformations of (anti)oriented functors.

When we view $\infty\Cat$ as an oriented category, the morphism objects are
\[
\R\Mor_{\infty \fcat}(S,T)=\Fun^{\lax}(S,T)
\]
because $\Fun^{\lax}(S,-)$ is adjoint to $S\boxtimes(-)$, which is compatible with the right action.
Dually, when we view $\infty\Cat$ as an antioriented category, the morphisms objects are
\[
\L\Mor_{\infty \fcat}(S,T)=\Fun^{\oplax}(S,T)
\]
because $\Fun^{\oplax}(S,-)$ is adjoint to $(-)\boxtimes S$, which is compatible with the left action.

A fundamental example of an oriented endofunctor of oriented spaces is the oriented cylinder, given by $X\mapsto \bD^1\boxtimes X$.
The oriented cylinder functor $X\mapsto \bD^1\boxtimes X$ is an oriented functor because
\[
\Fun^{\lax}(S,T)\to\Fun^{\lax}(\bD^1\boxtimes S,\bD^1\boxtimes T)
\]
is adjoint to the evaluation functor
\[
\bD^1\boxtimes S\boxtimes\Fun^{\lax}(S,T)\to \bD^1\boxtimes T
\]
which is compatible with composition: the diagram
\[
\xymatrix{
\Fun^{\lax}(S,T)\boxtimes\Fun^{\lax}(T,U)\ar[r]\ar[d] & \Fun^{\lax}(\bD^1\boxtimes S,\bD^1\boxtimes T)\boxtimes\Fun^{\lax}(\bD^1\boxtimes T,\bD^1\boxtimes U)\ar[d]\\
\Fun^{\lax}(S,U)\ar[r] & \Fun^{\lax}(\bD^1\boxtimes S,\bD^1\boxtimes U)
}
\]
is adjoint to commutative square
\[
\xymatrix{
\bD^1\boxtimes S\boxtimes\Fun^{\lax}(S,T)\boxtimes\Fun^{\lax}(T,U)\ar[r]\ar[d] & \bD^1\boxtimes S\boxtimes\Fun^{\lax}(S,U)\ar[d]\\
\bD^1\boxtimes T\boxtimes\Fun^{\lax}(T,U)\ar[r] & \bD^1\boxtimes U.
}
\]
Note that commutativity requires that we tensor on the left with the disk.
Tensoring on the right with the disk is an antioriented functor, and the morphism objects are computed using $\Fun^{\oplax}$.

Related to the cylinder is the suspension, which by \cite[Theorem 3.7.2.]{GepnerHeine2026} is computed as the pushout
\[
\xymatrix{
X\boxtimes\partial\bD^1\ar[r]\ar[d] & X\boxtimes\bD^1\ar[d]\\
\partial\bD^1\ar[r] & {S}(X).
}
\]
Similarly, the antisuspension is computed by either of the pushouts
\[
\xymatrix{
X\bar{\boxtimes}\partial\bD^1\ar[r]\ar[d] & X\bar{\boxtimes}\bD^1\ar[d]& & &\partial\bD^1\boxtimes X\ar[r]\ar[d] & \bD^1\boxtimes X\ar[d]\\
\partial\bD^1\ar[r] & \bar{S}(X) & & & \partial\bD^1\ar[r] & \bar{S}(X)
}.
\]
Using these descriptions, we show in \cref{alfa} that the suspension is an antioriented functor, and the antisuspension is an oriented functor.

The analogous operations for the join monoidal structure are the antijoin and the join operations.
We show that the antijoin and join operations are computed in terms of the Gray tensor product as the pushouts
    
\begin{equation}\label{joinform}
\xymatrix{
X\bar{\boxtimes}\partial\bD^1\bar{\boxtimes} Y\ar[d]\ar[r] & X\bar{\boxtimes}\bD^1\bar{\boxtimes} Y\ar[d] & & & X{\boxtimes}\bD^1{\boxtimes} Y \ar[d]\ar[r] & X\boxtimes\bD^1\boxtimes Y\ar[d]\\
X+Y\ar[r] &  X\bar{\star}Y & & & X+Y\ar[r] & X\star Y.
}
\end{equation}
As the oriented cylinder is an oriented functor and the
antioriented cylinder is an antioriented functor,
this has the consequence that for $\infty$-categories $X$ and $Y$,
the left join $ X \ast (-) $ and the right antijoin 
$ (-) \bar{\ast} Y $ are oriented functors and the 
right join $  (-)\ast Y $ and the left antijoin 
$ X \bar{\ast} (-) $
are antioriented functors.

The join operation is very useful, as it encodes the lax under and oplax over $\infty$-categories.
We show that, for every morphism $f:Y\to Z$, there exist slice $\infty$-categories $Z_{f/\!/^\lax}$ and $Z_{/\!/^\oplax f}$, such that functors $X\to Z_{f/\!/^\lax}$ and $X\to Z_{/\!/^\oplax f}$ correspond to functors $Y\star X\to Z$ and $X\star Y\to Z$, respectively, extending the given functor $f:Y\to Z$.
Moreover, this adjunction refines to an antioriented adjunction
\[
(-)\star Y : \infty\fcat \leftrightarrows \infty\fcat_{Y/} : Z_{(-)/\!/^\lax}
\]
and an oriented adjunction
\[
X\star (-) :\infty\fcat\leftrightarrows\infty\fcat_{X/} : Z_{/\!/^\oplax (-)}.
\]

Much of the computational power of homotopy theory comes from the notion of the fiber of a map, and the long exact sequence of a fibration.
In oriented and antioriented category theory these come as
oriented and antioriented left and right fibers.
The oriented left and right fibers of an oriented functor $Y\to X$ over an object $x\in X$ are the oplax under and over categories $Y \times_X X_{x //^\oplax}$ and $Y\times_X X_{//^{\oplax} x}$, respectively, and analogously for the antioriented versions.

The cone is an instance of the join.
Hence the left cone and right anticone are oriented functors, and the right cone and left anticone are antioriented functors.
There is also the bicone operation.
Specifically, the antioriented bicone is given by the formula
\[
\xymatrix{
& X\ar[ld]\ar[rd] &\\
\bD^0\bar{\star} X\ar[rd] & & X\star\bD^0\ar[ld]\\
& X^\abd &
}
\]
and is an antioriented functor.
Its dual version, the oriented bicone, is an oriented functor.

Our approach to oriented categories allows us to model the category of oriented categories without explicit reference to the Gray tensor product on $\infty\Cat$, provided we input the combinatorics of the oriented cubes using a combinatorial device such as Steiner theory.
This also allows us to characterize the oriented category of oriented spaces via a simple universal property: namely, as the oriented cocompletion of the point.
We will treat the general theory of oriented colimits in a subsequent paper.

\subsection{Bioriented category theory}
The Gray tensor product comes in both oriented and antioriented (also called oplax and lax) versions, suggesting that we should regard $\infty\fcat$ as bienriched over itself.
Bienriched categories are a refinement of enriched categories that allows for two-sided enrichment in two monoidal categories.
This is a crucial notion when working with enrichment in a not necessarily symmetric monoidal category, as we may enrich in the monoidal and reverse monoidal structure simultaneously.

In keeping with our notation $\boxtimes\Cat$ and $\Cat\boxtimes$ for the categories of (anti)oriented categories, we write $\boxtimes\Cat\boxtimes$ for the category of bioriented categories.
The theory of enrichment provides a commutative square
\[
\xymatrix{
& \infty\Cat\ar[ld]\ar[rd] &\\
\boxtimes\Cat\ar[rd] & & \Cat\boxtimes\ar[ld]\\
& \boxtimes\Cat\boxtimes &
}
\]
and we show that each of these functors is a fully faithful embedding.

In our case of interest, we wish to enrich in the Gray tensor product, without preferring one orientation over the other.
In the case of the category of $\infty$-categories, this amounts to specifying, for each ordered pair of $\infty$-categories $(S,T)$, morphism $\infty$-categories
\[
\Fun^{\lax}(S,T)\qquad\textrm{and}\qquad\Fun^{\oplax}(S,T)
\]
of oriented (lax) and antioriented (oplax) morphisms from $S$ to $T$.
The bienriched morphism object is then the object of $\infty\Cat\otimes\infty\Cat$ which represents the functor $(\infty\Cat\otimes\infty\Cat)^{\op}\to\mS$ given by the formula
$$ (X,Y)\mapsto\Map_{\infty\Cat}(X\boxtimes S\boxtimes Y,T).
$$
This is well defined since the canonical map $\infty\Cat\times\infty\Cat\to\infty\Cat\otimes\infty\Cat$ generates under colimits.
In other words, we can write the associated representable functor in either one of the forms
\[
\Map(S,\Fun^{\lax}(X,\Fun^{\oplax}(Y,T)))\simeq\Map(S,\Fun^{\oplax}(Y,\Fun^{\lax}(X,T)).
\]

In bienriched category theory, the monoidal product is a bienriched bifunctor.
To make this notion precise in our setting and discusss further important examples, we exhibit the cartesian product $ \mC\times \mC'$ of an antioriented category $\mC$ and an oriented category $\mC'$ as a
bioriented category, 
denoted $ \langle \mC,\mC'\rangle. $
Here the morphism object is \[
\Mor_{\langle \mC,\mC'\rangle}((S,S'),(T,T'))=\LMor_{\mC}(S,T)\otimes \RMor_{\mC'}(S',T'),
\]
where $A\otimes B$ denotes the image in $\mC\otimes\mC'$ of any object $(A,B)\in\mC\times\mC'$.
If $\mD$ is a bioriented category, a bioriented functor
$\langle \mC,\mC'\rangle\to\mD$
is equivalent to either the data of an oriented functor or antioriented functor
\[
\mC'\to\boxtimes\Fun(\mC,\mD) \qquad\qquad\mC\to {\Fun\boxtimes}(\mC',\mD)
\]
respectively.

The join formula (\ref{joinform}) implies that the join is an example of a bioriented functor.
Specifically, we show that the join is a bioriented functor and that the antijoin is an antioriented functor.
The bioriented functoriality of the join can be used to obtain a very useful distributivity type formula relating the join and the tensor.
Specifically, we show that there is a cocartesian square of the form
\[
\xymatrix{
X\boxtimes A\boxtimes Y\coprod X\boxtimes B\boxtimes Y\ar[r]\ar[d] & X\boxtimes A\coprod B\boxtimes Y\ar[d]\\
X\boxtimes(A\star B)\boxtimes Y\ar[r] & (X\boxtimes A)\star(B\boxtimes Y)
}
\]
for any $\infty$-categories $A,B,X,Y$.

Nevertheless, there are some nontrivial bioriented functors coming from the cylinder construction.
The suspension is a quotient of the right cylinder, and therefore an antioriented functor, whereas the antisuspension is a quotient of the left cylinder, and therefore an oriented functor.
However, one can construct a bioriented suspension functor by applying the $(-)^\cop$ involution to the bioriented category $\infty\fcat$, which does not effect the antiorientation but twists the orientation
\[
\RMor_{\infty\fcat}(S,T)^{\coop}=\Fun^{\lax}(S,T)^{\coop}=\Fun^{\lax}(S^{\coop},T^{\coop}).
\]

Since the $(-)^{\coop}$ twisting is an oriented functor
\[
(-)^{\coop}:\infty\fcat^{\cop}\to\infty\Cat,
\]
we can compose the antisuspension with this twisting to obtain an oriented functor
\[
\bar{S}((-)^{\coop}):\infty\fcat^{\cop}\to\infty\fcat.
\]
On the other hand, the suspension itself is an antioriented functor
\[
S:\infty\fcat^{\cop}\to\infty\fcat.
\]
In this case, we have an equivalence $S(-)\simeq\bar{S}((-)^{\coop})$, and we prove in \cref{alfas2} that this equivalence is compatible with both the oriented and the antioriented structures, endowing $S$ with a bioriented structure.

\subsection{Main results}

In the following we present the main results of our work.
We prove two main results on oriented and antioriented categories.
The first provides a geometric presentation of (anti)oriented categories, the second result embeds $\infty$-category theory into (anti)oriented category theory.

Oriented categories are categories left enriched in the lax Gray tensor product, or right enriched in the oplax Gray tensor product.
Dually, antioriented categories are categories left enriched in the oplax Gray tensor product, or right enriched in the lax Gray tensor product.
Throughout the paper, unless specified otherwise, enrichment will refer to left enrichment; we could have equally well chosen the other convention, in which case one systematically replaces the oplax Gray tensor product $\boxtimes$ with the lax Gray tensor product $\bar{\boxtimes}$.
Perhaps the most confusing thing in this paper is keeping track of all the many involutions which occur!

We give a geometric presentation of (anti)oriented categories, which does not make reference to any enriched category theory. To state this geometric presentation we introduce the following notation.
Let
\begin{align*}
\Theta(\cube,\times)\subset {_{(\infty\Cat,\times)}\Cat}\simeq \Cat_{(\infty\Cat,\times)}&\simeq\infty\Cat\\
\Theta(\cube,\boxtimes)\subset {_{(\infty\Cat,\boxtimes)}\Cat}\simeq \Cat_{(\infty\Cat,\bar{\boxtimes})}&\simeq{\boxtimes\Cat} \\
\Theta({\cube}, \bar{\boxtimes})\subset {_{(\infty\Cat,\bar{\boxtimes})}\Cat\simeq\Cat_{(\infty\Cat,\boxtimes)}} &\simeq\Cat\boxtimes\\
\Theta(\cube\otimes\cube,\boxtimes\otimes\bar{\boxtimes})\subset {_{(\infty\Cat,\boxtimes)}\Cat_{(\infty\Cat,\boxtimes)}}&\simeq\boxtimes\Cat\boxtimes
\end{align*}
be the full subcategories consisting of the finite wedges of $\infty$-categories of the form $S(\cube^n), S(\bar{\cube}^n)$ for all natural numbers $n$, respectively.
Note that for any oriented cubes $\cube^m$ and $\cube^n$, one can compute that
\begin{align*}
\Mor_{S(\cube^m)\vee S(\cube^n)}(0,2)&\simeq\cube^m\times\cube^n\\
\LMor_{S(\cube^m)\vee S(\cube^n)}(0,2)&\simeq\cube^m\boxtimes\cube^n\simeq\cube^{m+n}\\
\RMor_{S({\cube}^m)\vee S({\cube}^n)}(0,2)&\simeq{\cube}^m\,\bar{\boxtimes}\,{\cube}^n\simeq{\cube}^{n+m}\\
\mathrm{BMor}_{S(\cube^m\otimes\cube^{m'})\vee S(\cube^n\otimes\cube^{n'})}(0,2)&\simeq(\cube^m\boxtimes\cube^n)\otimes(\cube^{m'}\bar{\boxtimes}\cube^{n'})\simeq\cube^{m+n}\otimes{\cube}^{n'+m'}
\end{align*}
In order to distinguish these we keep track of the enriching monoidal structure in the notation.

We present $\infty$-categories as sheaves on
$ \Theta(\cube,\times)$, antioriented categories as sheaves on
$ \Theta(\cube,\boxtimes)$, and oriented categories as sheaves on
$ \Theta(\bar{\cube},\bar{\boxtimes})$. 
For any presheaf $F$ on $ \Theta(\cube,\boxtimes) $ and $(s,t)\in F(\ast)$, we define the oriented cubical space of morphisms $s \to t$ in $F$ by the presheaf $$ \LMor_{F}(s,t) :
\cube^{\op}\xrightarrow{F\circ S} \infty\Gpd_{/F(\ast)\times F(\ast)}\overset{\fib}{\to}\infty\Gpd,
$$
where $S: \cube \to \Theta(\cube,\boxtimes)$ is the suspension.
We can do the same for presheaves on $ \Theta(\bar{\cube},\bar{\boxtimes})$ and $ \Theta(\cube,\times)$ 
using the suspensions $S: \cube \to \Theta({\cube},\bar{\boxtimes})$
and $S : \boxtimes \to \Theta(\cube,\times), $ in order to obtain oriented cubical spaces $ \RMor_{F}(s,t) $ and $ \Mor_{F}(s,t)$.

We derive the following geometric presentation of (anti)oriented categories and of $\infty$-categories:

\begin{theorem}[\cref{geomorient}]\label{geomorient1}
     The induced nerve functors
    \[
    \N_{\Theta(\cube,\boxtimes)}:{\boxtimes\Cat}\to\Fun(\Theta(\cube, \boxtimes)^\op,\infty\Gpd)
    \]
     \[
    \N_{\Theta(\bar{\cube},\bar{\boxtimes})}:\Cat\boxtimes \to \Fun(\Theta(\bar{\cube}, \bar{\boxtimes})^\op,\infty\Gpd)
    \]
     \[
    \N_{\Theta(\cube,\times)}:\infty\Cat \to \Fun(\Theta(\cube, \times)^\op,\infty\Gpd)
    \]
    are fully faithful right adjoint embeddings.
    A presheaf $F:\Theta(\cube, \boxtimes)^\op\to\infty\Gpd$ belongs the image of the nerve functor if and only if it satisfies the following two conditions:
    \begin{enumerate}[\normalfont(1)]\setlength{\itemsep}{-2pt}
    \item the Segal condition, i.e. for every $k, n_1, ..., n_k \geq 0$ the following map is an equivalence:
\[
    F(S(\cube^{n_1})\lor\cdots\lor S(\cube^{n_k}))\to F(S(\cube^{n_1}))\times_{F(\ast)}\cdots\times_{F(\ast)}F(S(\cube^{n_k}))
    \]

    \item the local Segal condition: for any $(s,t)\in F(\ast)$
    the presheaf $\LMor_{F}(s,t) $ is local with respect to the following families of maps, where $\N: \infty\Cat \to \Fun(\cube^\op, \infty\Cat) $ is the oriented cubical nerve:
\begin{enumerate}[\normalfont(i)]\setlength{\itemsep}{-2pt}
\item The boundary decomposition
$\mathrm{coeq}(\underset{0 \leq i <j \leq n}{\coprod}4 \times \N(\cube^{n-2})\rightrightarrows \underset{0 \leq i \leq n}{\coprod} 2 \times \N(\cube^{n-1})) \to \N(\partial\cube^\n).$
\item The oriental decomposition $\N(\bD^\n) \underset{\partial\N(\bD^\n)}{\coprod} \N(\partial\cube^\n) \to \N(\cube^\n).$
\item The globular decomposition
$\N(\bD^{i_0}) \!\!\underset{\N(\bD^{j_1})}{\coprod} \!\!\N(\bD^{i_1}) \underset{\N(\bD^{j_2})}{\coprod} \cdots \underset{\N(\bD^{j_n})}{\coprod} \!\!\N(\bD^{i_n}) \to \N(\bD^{i_0} \underset{\bD^{j_1}}{\coprod} \bD^{i_1} \underset{\bD^{j_2}}{\coprod} \cdots \underset{\bD^{j_n}}{\coprod} \bD^{i_n}).$
\end{enumerate}
Here, $n$, $i_0,\ldots, i_n$, and $j_1,\ldots, j_n$ are natural numbers, and 
$ \bD^{j_\ell} \rightarrowtail \bD^{i_\ell},\bD^{j_\ell} \rightarrowtail \bD^{i_{\ell-1}}$ are monomorphisms.
    
The similar holds for a presheaf on $\Theta(\bar{\cube}, \bar{\boxtimes})$
and $\Theta(\cube, \times)$.
\end{enumerate}
    
\end{theorem}

The collection of small $\infty$-categories forms a large $\infty$-category $\infty\Cat$ and also forms a large bioriented category $\infty\fcat.$ 
We use \cref{geomorient1} to characterize the underlying oriented and antioriented categories of $\infty\fcat$ as well as the $\infty$-category $\infty\Cat$ by universal properties, where $\Map'_{\mP(\Theta(\cube, \boxtimes))}\subset\Map_{\mP(\Theta(\cube, \boxtimes))}$ denotes the subspace spanned by the maps of presheaves on 
$\Theta(\cube, \boxtimes)$ which preserve tensors:

\begin{corollary}[\cref{charas}]\label{chara}
\begin{enumerate}[\normalfont(1)]\setlength{\itemsep}{-2pt}
\item Let $X$ be a presheaf on $\Theta(\cube, \boxtimes) $ satisfying the Segal condition and that is tensored.
The map $$ \Map'_{\mP(\Theta(\cube, \boxtimes))}(\N_{\Theta(\cube, \boxtimes)}(_{\boxtimes \mid}\infty\fcat),X) \to \Map_{\mP(\Theta(\cube, \boxtimes))}(\N_{\Theta(\cube, \boxtimes)}(\ast),X) \simeq X(\ast),$$ induced by the antioriented functor $\ast \to \infty\fcat $ selecting the final $\infty$-category, is an equivalence.

\item Let $X$ be a presheaf on $\Theta(\bar{\cube}, \bar{\boxtimes}) $ satisfying the Segal condition and that is tensored.
The map $$ \Map'_{\mP(\Theta(\bar{\cube}, \bar{\boxtimes}))}(\N_{\Theta(\bar{\cube}, \bar{\boxtimes})}(\infty\fcat_{\mid \boxtimes}),X) \to \Map_{\mP(\Theta(\bar{\cube}, \bar{\boxtimes}))}(\N_{\Theta(\bar{\cube}, \bar{\boxtimes})}(\ast),X) \simeq X(\ast),$$ induced by the oriented functor $\ast \to \infty\fcat $ selecting the final $\infty$-category, is an equivalence, where the left hand side is the 
subspace spanned by the maps of presheaves on 
$\Theta(\bar{\cube}, \bar{\boxtimes})$ preserving tensors.

\item Let $X$ be a presheaf on $\Theta(\cube, \times) $ satisfying the Segal condition and that is tensored.
The map $$ \Map'_{\mP(\Theta(\cube, \times))}(\N_{\Theta(\cube, \times)}(\infty\Cat),X) \to \Map_{\mP(\Theta(\cube, \times))}(\N_{\Theta(\cube, \times)}(\ast),X) \simeq X(\ast),$$ induced by the functor $\ast \to \infty\Cat $ selecting the final $\infty$-category, is an equivalence, wheere the left hand side is the subspace spanned by the maps of presheaves on 
$\Theta(\cube, \times)$ preserving tensors.
\end{enumerate}
\end{corollary}

\cref{chara} is an analogue of the classical result in homotopy theory that the category of homotopy types is the completion of the point under small colimits. 
This corollary characterizes $\infty\fcat$ without explicit reference to the Gray tensor product.

The proof of \cref{geomorient} relies on two results about density.
The first result is about dense subcategories of $\infty\Cat$
and generalizes Grothendieck's philosophy of test categories to higher categorical dimensions. The second result explains how a dense subcategory of $\infty\Cat$ gives rise to a dense subcategory of 
$ {\boxtimes \Cat}, {\Cat \boxtimes}$, the categories of 
$ (\infty,\Cat, \boxtimes)$-enriched categories, 
$ (\infty,\Cat, \bar{\boxtimes})$-enriched categories, respectively.
We prove it in full generality in the context of enrichment in any monoidal category.

To state it the second result let $(\mV, \ot)$ be a monoidal category, which one can encode by a cartesian fibration $\mV_\ot \to \Delta$ to the opposite of the simplex category whose fiber over the one simplex is $\mV$ and over the zero simplex is contractible and which satisfies a Segal condition.
We prove a result about dense subcategories of the category $\Cat_{(\mV,\otimes)}$ of small categories enriched in a monoidal category $(\mV,\otimes)$, strengthening a result of \cite{MR3345192}.
We identify the category $\mV_\otimes$ with the $\Theta$-construction $\Theta(\mV,\otimes)=\mV\wreath\Delta^{\op}$ of \cite{barwick2018operator}, which we view as a full subcategory of $\Cat_{(\mV,\otimes)}$ via the functor
\[
S:\mV_\ot \to\Cat_{(\mV,\otimes)}, \ (A_1,\ldots,A_n) \mapsto S(A_1)\vee\cdots\vee S(A_n)
\]

\begin{theorem}[\cref{denseinherited}]
Let $\mW\subset\mV$ be a dense full subcategory which is closed under the monoidal product in $\mV$, and write $\mW_\otimes\subset\mV_\otimes$ for the resulting full subcategory.
Then $S:\mW_\otimes\to\Cat_{(\mV,\otimes)}$ is a dense full subcategory.
\end{theorem}

As a second main structural result we embed the theory of $\infty$-categories into the larger theory of oriented and antioriented categories. To state this embedding, we introduce the following notation.

The identity functor $\id:(\infty\Cat,\times)\to(\infty\Cat,\boxtimes)$ is lax monoidal, 
which implies that we may view any $\infty$-category as an oriented category and also as an antioriented category.
We therefore obtain right adjoint functors of $2$-categories
\begin{equation}\label{eqn:bicat1}
\infty\Cat \simeq {_{(\infty\Cat,\times)}\Cat} \to {_{(\infty\Cat,\boxtimes)}\Cat}={\boxtimes\Cat}, \end{equation}
\begin{equation}\label{eqn:bicat2}
\infty\Cat \simeq \Cat_{(\infty\Cat,\times)}\to\Cat_{(\infty\Cat,\boxtimes)}={\Cat\boxtimes}.
\end{equation}
An (anti)oriented space is an (anti)oriented category in the essential image of the respective functor.
We refer to the left adjoints 
\[
||-||_\boxtimes:\Cat\boxtimes\to\infty\Cat\qquad\textrm{and}\qquad _\boxtimes||-||:\boxtimes\Cat\to\infty\Cat
\]
of the functors \ref{eqn:bicat1} and \ref{eqn:bicat2} as the oriented and antioriented realization functors.
The oriented and antioriented realization functors restrict to functors $$ \Theta(\cube,\boxtimes)\to\Theta(\cube,\times), \ \Theta(\bar{\cube},\bar{\boxtimes})\to\Theta(\bar{\cube},\times)$$ which send a wedge of suspensions of (anti)oriented cubes to a wedge of suspensions of (anti)oriented cubes, now formed with respect to the cartesian enrichment.

\begin{theorem}[\cref{interchange}]\label{thm:rightadjoint}
\begin{enumerate}[\normalfont(1)]\setlength{\itemsep}{-2pt}
\item The functors
    $$
    \infty\Cat\to\Cat\boxtimes, $$$$  \infty\Cat\to {\boxtimes\Cat}
    $$
    of 2-categories are fully faithful.
     
    \item An oriented category $\mC$ is an oriented space if and only if it satisfies
     the strict interchange law: for all functors $\bD^m\to\RMor_\mC(A,B)$ and $\bD^n\to\RMor_\mC(B,C)$, the induced functor
     \[
     \bD^m\boxtimes\bD^n\to\RMor_\mC(A,C)
     \]
     factors through the quotient $\bD^m\boxtimes\bD^n\to\bD^m\times\bD^n$.
     The similar holds for antioriented categories.

     \item A presheaf $F:\Theta(\cube, \boxtimes)^\op\to\infty\Gpd$ 
     is an antioriented space if and only if
     it factors through $\Theta(\cube, \times)^\op.$
     
     A presheaf $F:\Theta(\bar{\cube}, \bar{\boxtimes})^\op\to\infty\Gpd$ 
     is an oriented space if and only if
     it factors through $\Theta(\bar{\cube}, \times)^\op.$
     \end{enumerate}
\end{theorem}

\cref{geomorient} provides localizations
$$ \Fun(\Theta(\cube, \boxtimes)^\op,\infty\Gpd) \rightleftarrows {\boxtimes\Cat}, $$
$$ \Fun(\Theta(\cube, \times)^\op,\infty\Gpd) \rightleftarrows {\infty\Cat}. $$

\cref{thm:rightadjoint} implies the following corollary:

\begin{corollary}[\cref{restind}]
Restriction and left Kan extension along the functor of $1$-categories $\Theta(\bar{\cube},\bar{\boxtimes})\to\Theta(\bar{\cube},\times)$ descends to a localization
\[
{\Cat\boxtimes}\leftrightarrows\infty\Cat.
\]
Moreover, the left adjoint agrees with the oriented realization functor $||-||_\boxtimes$.
\end{corollary}

\cref{thm:rightadjoint} has a number of interesting applications, including to the theory of oriented limits and colimits, the natural notion of (co)limit in oriented category theory.
For instance, it can be used to show that for any $\infty$-categories $\mC,\mD$ the oriented colimit of the constant oriented functor $\mD\to\infty\fcat$ with value $\mC$ is $\mC\boxtimes\mD$. Note that the corresponding statement for the cartesian product is considerably easier, but already uses the Grothendieck construction for $\infty$-categories.
In this paper, we only make use of oriented pullback and pushouts, and we develop the general theory of oriented (co)limits in a follow-up paper.

An important ingredient in the proof of \ref{thm:rightadjoint} is the following result, which exhibits the product as a quotient of the Gray tensor product:

\begin{theorem}[\cref{grayepi}]\label{grayepi2} 
For every family of $\infty$-categories $X_1,\ldots,X_n$, the functor
\[
X_1\boxtimes\cdots\boxtimes X_n\to X_1\times\cdots\times X_n
\]
is an epimorphism.
\end{theorem}

The latter theorem implies an important relationship between $\infty\scat$
and $\infty\fcat$.
Another consequence of the fact that the identity functor is lax monoidal is that, for any pair of $\infty$-categories $\mA$ and $\mB$, there are induced functors
\[
\Fun(\mA,\mB)\to\Fun^\lax(\mA,\mB)\qquad\textrm{and}\qquad\Fun(\mA,\mB)\to\Fun^\oplax(\mA,\mB)
\]
which collectively organize to form a bioriented functor $i:\infty\scat$$\to\infty\fcat$.
Using \cref{grayepi2} we deduce the following corollary:

\begin{corollary}\label{monomo}
The functor $i:\infty\scat$$\to\infty\fcat$ is a monomorphism of bioriented categories.
\end{corollary}

To prove \cref{grayepi2}, we compute the $\infty$-category of functors and lax transformations out of a suspension (\cref{laxnat}).
This computation also implies the following 
formula for the Gray tensor product of suspensions (\cref{tensorsuspensionformula}):
the tensor product $S(\mA)\boxtimes S(\mB)$ is computed as the suspension of $\mA\boxtimes\bD^1\boxtimes\mB$ glued to the boundary of $S(\mA)\times S(\mB)$, i.e. the pushout
\[
\xymatrix{
S(\mA\times\partial\bD^1\times\mB)\ar[r]\ar[d] & S(\mA\boxtimes\bD^1\boxtimes\mB)\ar[d]\\
\partial(S(\mA)\times S(\mB))\ar[r] & S(\mA)\boxtimes S(\mB).
}
\]
\begin{corollary}
    The vertical maps in the above pushout square are fully faithful embeddings.
    In particular, we obtain a formula for the morphism $\infty$-category between the first and last objects of the tensor:
\[
\Mor_{S(\mA)\boxtimes S(\mB)}((0,0),(1,1))=\mA\boxtimes\bD^1\boxtimes\mB\underset{\mA\boxtimes\partial\bD^1\boxtimes\mB}{\coprod}\mA\times\partial\bD^1\times\mB.
\]
\end{corollary}

\subsection{Relation to other work}

Our theory of oriented spaces studies $\infty$-categories as a category enriched in the Gray monoidal structure.
The basic theory of $\infty$-categories, the Gray tensor product, fibrations, and limits and colimits have been worked out recently in various models.
Notably, in the strict setting, Verity developed the model of complicial sets \cite{VERITY1}, and Ara-Guetta \cite{ara2025lax} study the lax slice construction and its functoriality.
In the homotopical context, Loubaton adapted the approach of Verity to weak $\infty$-categories \cite{loubaton2024complicialmodelinftyomegacategories} and has written extensively about foundations from this perspective \cite{loubaton2024categorical}.

We expect the theory of oriented categories to have applications both internally to higher category theory, as well as externally, to mathematical physics, representation theory and non-commutative algebra.
For instance Johnson-Freyd-Reutter\cite{johnson2025build} apply the Gray tensor product to
construct Hopf algebras in braided monoidal categories from a suitable retract in a 3-category, which specializes to the classical Tannakian reconstruction of a Hopf algebra from a monoidal category with duals and a fiber functor. 

The gradual adaptation of homotopy theory in the higher categorical context has led to the development of categorical versions of spectra, which make essential use of the Gray enrichment.
Versions of the theory of categorical spectra have been studied in  \cite{heine2025categorification}, \cite{heine2026stable}, \cite{Masuda}, \cite{stefanich2021higher}.

The Baez--Dolan cobordism hypothesis builds a bridge between topological quantum field theory and $n$-category theory.
Lurie's sketch proof \cite{Cob} of the cobordism hypothesis suggests that higher category theory is essential in the study of field theories and categorical symmetries.
Other work in this direction, and more generally applications of higher category theory to mathematical physics and representation theory, has been carried out in \cite{baez1995higher}, \cite{baez2011prehistory}, \cite{ferrer2024dagger}, \cite{gaiotto2019condensations}, \cite{liu2024braided}, and many others.

None of these developements would have been possible without the Gray tensor product, which plays a fundamental role throughout this paper.
The insight that there should be a refinement of the cartesian product which plays an analogous role in higher categorical settings goes back to the work of Gray \cite{GRAY197663}.
\footnote{The Gray tensor has been subsequently studied by many others, especially in the context of 2-category theory. See for instance \cite{ara.folkmodel}, \cite{Dimitri_Ara_2020}, \cite{bourke2023skew}, \cite{campion2022cubesdenseinftyinftycategories}, \cite{crans1999tensor}, \cite{gordon1995coherence}, \cite{MAEHARA2021107461}.}

\subsection{Notation and terminology}

We fix a hierarchy of set-theoretic universes whose objects we call small, large, very large, etc.
We call a space (equivalently, $\infty$-groupoid) $X$ small, large, etc. if for any choice of basepoint and natural number $n$ its homotopy sets $\pi_n X$ are small, large, etc.
We call an $\infty$-category small, large, etc. if its maximal subspace and all its mapping spaces are small, large, etc.

We refer to (not necessarily univalent) weak $(\infty,\n)$-categories for $0 \leq \n \leq \infty$ simply as $\n$-categories, and we refer to (not necessarily univalent) weak $(\n,\n)$-categories as $(\n,\n)$-categories.
In particular, we refer to (not necessarily univalent) $(\infty,1)$-categories as 1-categories, or simply categories.
We will sometimes want to work strictly, which can be viewed as a basechange along the colimit-preserving symmetric monoidal functor $\mS\to\Set$.
In this case, we will refer to strict $(\n,\n)$-categories simply as strict $\n$-categories.\footnote{Note that an $(n,n)$-category need not be a strict $(n,n)$-category if $n>2$.
For instance, the fundamental $\infty$-groupoid of the $2$-sphere $S^2\simeq \mathrm{B}^2\Omega^2 S^2$ is an $(\infty,0)$-category which is not strict, nor are its $n$-truncations for any $n>2$.}


\begin{notation}
We will make use of the following notation and terminology when discussing categories, in the sense of categories enriched in the monoidal category of $\infty$-groupoids under the cartesian product.
\begin{enumerate}[\normalfont(1)]\setlength{\itemsep}{-2pt}
\item We write $\mS$ for the category of spaces, by which we mean small $\infty$-groupoids, homotopy types, or anima, and $\Set$ for the category of small sets.
\item We write $\infty\Cat$ for the large category of small $\infty$-categories.

\item We write $\Delta$ for (a skeleton of) the category of finite, non-empty, partially ordered sets and order preserving maps, whose objects we denote by $[\n] = \{0 < ... < \n\}$ for $\n \geq 0$.\footnote{This should not be confused with the category $\bDelta$ of oriented simplices.}
\item We write $\Map_{\mC}(A,B)$ for the space of maps (equivalently, $1$-morphisms) from $A$ to $B$ in $\mC$, for any category $\mC$ containing an ordered pair of objects $(A,B)\in\mC$.

\item We write $\ast$ for the final object and
$\emptyset$ for the initial object in any category.

\item We call a fully faithful functor $\mC \to \mD$ an embedding.

\item We call a functor $\mC \to \mD$ an inclusion if it induces an embedding on maximal subspaces and on all mapping spaces. A functor is an inclusion if and only if for every category $\mB$ the induced map
$\Map_\Cat(\mB,\mC) \to \Map_\Cat(\mB,\mD)$ is an embedding.
Thus inclusions are exactly the monomorphisms in $\Cat$. 

\item For a diagram $X\to Z\leftarrow Y$ in a category $\mC$ we write $X\underset{Z}{\prod} Y$ or $X\underset{Z}{\times} Y$ for the pullback, and given a diagram $X\leftarrow W\to Y$ in a category $\mC$, we write $X\underset{W}{\coprod} Y$ or $X\underset{Z}{+} Y$ for the pushout.
\item If $\mC$ and $\mD$ are categories and $\mC\to\mD$ is a left adjoint functor with right adjoint $\mD\to\mC$, we often write $\mC\rightleftarrows\mD$ for this adjunction, where the left adjoint is understood to be the functor going from left to right.

\item 
We write $\mC_*$ or $\mC_{\ast/}$ for the category of pointed objects in a category $\mC$, i.e. the full subcategory of $\Fun([1],\mC)$ of arrows in $\mC$ whose source is a final object.

\item We write $\PrL$ and $\PrR$ for the subcategories of $\widehat{\Cat}$ spanned by the presentable categories and the left and right adjoint functors, respectively.
There is a canonical equivalence $\PrL \simeq (\PrR)^\op$
sending left to right adjoints.
\item We write $\otimes$ for the symmetric monoidal structure on $\PrL$ and $\PrR$.
More precisely, by \cite{lurie.higheralgebra}, $\PrL$ carries a closed symmetric monoidal structure such that the subcategory inclusion $\PrL \subset \widehat{\Cat}$ is a lax symmetric monoidal with respect to the the cartesian structure on $\widehat{\Cat}$.
\item We write $\infty\scat$ for the large $\infty$-category of small $\infty$-categories.\footnote{The morphism $\infty$-categories are formed via enrichment in the cartesian monoidal structure.}
\item We write $\Fun(\mD,\mC)$ for the $\infty$-category of functors from an $\infty$-category $\mD$ to an $\infty$-category $\mC$, the value at $\mC$ of the right adjoint to the functor $(-)\times\mC:\infty\Cat\to\infty\Cat$ for the cartesian product.

\item We write $\bD^1$ for the walking arrow, the category with two objects and a unique non-identity arrow.
\item We write $\partial\bD^1$ and $S^0$ for the maximal subspace in $\bD^1$, the set with two elements.

\item We write $\iota_{n}\mC$ for the $n$-category arising from an $\infty$-category $\mC$ by discarding all noninvertible morphisms above dimension $n.$
\item We write $\tau_{n}\mC$ for the $n$-category arising from an $\infty$-category $\mC$ by inverting all morphisms above dimension $n.$
\end{enumerate}
\end{notation}

\subsection*{Acknowledgements}
We thank Tim Campion, Rune Haugseng, Felix Loubaton, Naruki Masuda, Thomas Nikolaus, Markus Spitzweck, and Germ\'an Stefanich for interesting conversations related to the subject of this paper.
We thank the MPIM for their hospitality while much of this work was carried out.

\vspace{.25cm}

\section{Higher categories}\label{enr}
\vspace{.25cm}

\subsection{Enriched categories}

We first recall the notion of homotopy coherent enrichment, as defined and studied in, for instance, \cite{MR3345192}, \cite{heine2024higher}, \cite{heine2025equivalence}, \cite{HINICH2020107129}.


For every presentably monoidal category $\mV$ there is a presentable 2-category $${\mV\mathrm{-}\Cat}$$ of (not necessarily univalent) $\mV$-enriched categories and $\mV$-enriched functors and a forgetful functor $$ \iota:{\mV\mathrm{-}\Cat} \to \Cat$$ to the presentable 2-category $\Cat$ of (not necessarily univalent) categories,
which is an equivalence for $\mV= \mS$ the category of homotopy types \cite[Corollary 3.23.]{heine2024bienriched}. 

Let $${\mV\mathrm{-}\Cat}^\univ \subset {\mV\mathrm{-}\Cat}$$ be the reflexive full subcategory of univalent $\mV$-enriched categories.


 


	


	



	

	
	

\begin{notation}Let $\mV$ be a presentably monoidal category.
A $\mV$-enriched category $\mC$ has an underlying category $\iota(\mC)$, for every objects $X,Y \in \iota(\mC)$ a morphism object 
$$\Mor_\mC(X, Y) \in \mV$$
and for every objects $X,Y,Z \in \iota(\mC)$ a composition morphism in $\mV:$
$$\Mor_\mC(Y, Z) \ot \Mor_\mC(X, Y) \to \Mor_\mC(X, Z).$$

We write $\X \in \mC$ for $\X \in \iota(\mC)$ and usually notationally identify $\mC$ with $\iota(\mC).$





	
\end{notation}



The following is \cite[Example 2.134.]{heine2024bienriched}:

\begin{example}\label{linenr}
Let $\mV$ be a presentably monoidal category. Every presentably left $\mV$-tensored category is a $\mV$-enriched category.
Every $\mV$-linear functor between presentably left $\mV$-tensored categories is a $\mV$-enriched functor.
In particular, $\mV$, which is presentably left tensored over itself, is a $\mV$-enriched category.

\end{example}

\begin{notation}Let $\mV$ be a presentably monoidal category.
Let $B\tu_\mV \subset \mV$ be the full $\mV$-enriched subcategory spanned by the tensor unit. Then $\Mor_{B\tu_\mV}(\tu_\mV, \tu_\mV) \simeq \tu_\mV.$
    
\end{notation}

For every enriched category there is an opposite one:

\begin{notation}
There is an involution $$(-)^\circ: {\mV\mathrm{-}\Cat} \simeq {\mV^\rev\mathrm{-}\Cat}$$ forming the opposite enriched category.
For every $\mC \in {\mV\mathrm{-}\Cat}$ and $X,Y \in \mC$
there are canonical equivalences
$ \iota(\mC^\circ) \simeq \iota(\mC)^\op$
and $$ \Mor_{\mC^\circ}(X,Y) \simeq \Mor_{\mC}(Y,X).$$
\end{notation}

The following is \cite[Proposition 3.72.]{heine2024bienriched}:

\begin{proposition}

Let $\mV,\mW$ be presentably monoidal categories and 
$\phi:\mV \to \mW$ a lax monoidal functor.
There is an induced functor
$\phi_!: \mV \mathrm{-}\Cat \to \mW \mathrm{-}\Cat$
that transfers the enrichment.
This functor descends to a functor
$\phi_!: \mV \mathrm{-}\Cat^\univ \to \mW \mathrm{-}\Cat^\univ.$
For every $\mV$-enriched category $\mC$ and $X,Y \in \iota(\mC)$
there is an equivalence
$$ \Mor_{\phi_!(\mC)}(X,Y) \simeq \phi(\Mor_\mC(X,Y)).$$
    
\end{proposition}

There is a close relationship between enriched categories
and tensored and cotensored categories:

\begin{definition}Let $\mV$ be a presentably monoidal category, $\mC$ a $\mV$-enriched category and $X \in \mC, V \in \mV.$	 
\begin{enumerate}[\normalfont(1)]\setlength{\itemsep}{-2pt}
\item The tensor of $V$ and $X$ in $\mC$ is the object $V \ot X \in \mC $ such that there is a morphism
$V \to \Mor_\mC(X, V \ot X) $ in $\mV$ that induces for every $Y \in \mC$ an equivalence
$$ \Mor_\mC(V \ot X,Y) \to \Mor_\mV(V, \Mor_\mC(X,Y)). $$ 
\item The cotensor of $V$ and $X$ in $\mC$ is the object ${^V X} \in \mC $ that is the tensor of $V $ and $X$ in the opposite $\mV^\rev$-enriched category $\mC^\circ.$
\end{enumerate}	
\end{definition}











Since the category of enriched categories forms a 2-category, there is a natural intrinsic notion of adjunction between enriched categories:

\begin{definition}Let $\mV$ be a presentably monoidal category.
A $\mV$-enriched functor $\mC \to \mD$ admits a left (right) adjoint if
it admits a left (right) adjoint in the 2-category $\mV \mathrm{-}\Cat.$

\end{definition}

\begin{notation}Let $\mV$ be a presentably monoidal category and $\mM, \mN$ be $\mV$-enriched categories.

Let $$\mV\mathrm{-}\Fun(\mM,\mN)$$ be the category of $\mV$-enriched functors
$\mM \to \mN.$

Let $$\mV\mathrm{-}\Fun^\L(\mM,\mN) \subset {\mV\mathrm{-}\Fun(\mM,\mN)}$$ be the full subcategory of $\mV$-enriched functors
$\mM \to \mN$ that admit a $\mV$-enriched right adjoint.
    
\end{notation}

The following is \cite[Remark 2.55.]{heine2024bienriched}:

\begin{proposition}
Let $\mV$ be a presentably monoidal category.
\begin{enumerate}[\normalfont(1)]\setlength{\itemsep}{-2pt}
\item A $\mV$-enriched functor $\phi: \mC \to \mD$ admits a $\mV$-enriched right adjoint if and only if for every $\Y \in \mD$ the $\mV^\rev$-enriched functor
$\Mor_\mD(\phi(-),\X): \mC^\circ \to \mV$ is representable.
\item A $\mV$-enriched functor $\phi: \mC \to \mD$ admits a $\mV$-enriched left adjoint if and only if the opposite $\mV^\rev$-enriched functor $\phi^\circ: \mC^\circ \to \mD^\circ$ admits a $\mV^\rev$-enriched right adjoint. By (1) this holds if and only if for every $\Y \in \mD$ the $\mV$-enriched functor
$\Mor_\mD(\Y,\phi(-)): \mC \to \mV$ is representable.
\end{enumerate}
\end{proposition}

The following is \cite[Lemma 2.77.]{heine2024bienriched}:

\begin{proposition}\label{adj}
Let $\mV$ be a presentably monoidal category.
\begin{enumerate}[\normalfont(1)]\setlength{\itemsep}{-2pt}
\item A $\mV$-enriched functor $\mC \to \mD$ admits a right adjoint if and only if it preserves tensors and the underlying functor admits a right adjoint.
\item A $\mV$-enriched functor $\mC \to \mD$ admits a left adjoint if and only if it preserves cotensors and the underlying functor admits a left adjoint.
\end{enumerate}

\end{proposition}





Next we introduce the enriched category of enriched presheaves \cite{heine2025equivalence}.

\begin{notation}

Let $\mV, \mW$ be presentably monoidal categories and $\mM$ a $\mV$-enriched category and $\mO$ a $\mW$-enriched category.
Let $\langle \mM, \mN \rangle $ be the $\mV \ot \mW$-enriched category that is the transfer of enrichment of the $\mV \times \mW$-enriched category $\mM \times \mN$ along the universal
monoidal functor $\mV \times \mW \to \mV \ot \mW$ preserving small colimits componentwise.
    
\end{notation}

The next theorem follows from \cite[Proposition 4.11, Theorem 4.86]{heine2024bienriched}:

\begin{theorem}\label{psinho} Let $\mV, \mW$ be presentably monoidal categories and $\mN$ a univalent $(\mV, \mW)$-bienriched category. 
\begin{enumerate}[\normalfont(1)]\setlength{\itemsep}{-2pt}
\item Let $\mM$ be a small univalent $\mV$-enriched category. The
category ${\mV\mathrm{-}\Fun}(\mM, \mN)$ refines to an univalent $\mW$-enriched category characterized by an 
equivalence
$$ \mW\mathrm{-}\Fun(\mO,{\mV\mathrm{-}\Fun}(\mM, \mN)) \to {\mV \ot \mW\mathrm{-}\Fun}(\langle\mM,\mO\rangle,\mN)$$
natural in any univalent $\mW$-enriched category $\mO$.
\item Let $\mO$ be a small univalent $\mW$-enriched category. 
The category ${\mW\mathrm{-}\Fun}(\mO, \mN)$ refines to an univalent $\mV$-enriched category characterized by an equivalence
$$ \mV\mathrm{-}\Fun(\mM,{\mW\mathrm{-}\Fun}(\mO, \mN)) \to {\mV \ot \mW\mathrm{-}\Fun}(\langle\mM,\mO\rangle,\mN) $$
natural in any univalent $\mV$-enriched category $\mM$.

\item If $\mN$ is a presentably $\mV, \mW$-bitensored category, then ${\mV\mathrm{-}\Fun}(\mM, {\mN}) $ is a presentably right $\mW$-tensored category and ${\mW\mathrm{-}\Fun}(\mO, {\mN}) $ is a presentably left $\mV$-tensored category.

\end{enumerate}

\end{theorem}
    
\begin{definition}

Let $\mV$ be a presentably monoidal category and $\mC$ a small univalent $\mV$-enriched category.
The presentably left $\mV$-tensored category of $\mV$-enriched presheaves on $\mC$ is
$$ \mP_\mV(\mC):= \mV\mathrm{-}\Fun(\mC^\op,\mV). $$
\end{definition}

The next theorem, which follows from \cite[Theorem 3.41, Theorem 4.70]{heine2024bienriched}, is an enriched version of the universal property of the category of presheaves as the free cocompletion under small colimits \cite[Theorem 5.1.5.6]{lurie.HTT}:

\begin{theorem}
\label{Yonedaext}
Let $\mV$ be a presentably monoidal category, $\mC$ a small univalent $\mV$-enriched category and $\mD$ a presentably left $\mV$-tensored category.
\begin{enumerate}[\normalfont(1)]\setlength{\itemsep}{-2pt}
\item There is a $\mV$-enriched embedding $\iota_\mC : \mC \to \mP_\mV(\mC)$ that sends $X$ to $\L\Mor_\mC(-,X)$ and induces
for every presentably left $\mV$-tensored category $\mD$ an equivalence
$$ {\mV\mathrm{-}\Fun^\L}(\mP_\mV(\mC),\mD)\to \mV\mathrm{-}\Fun(\mC,\mD).$$
\item Let $F: \mC \to \mD$ be a $\mV$-enriched functor and
$\bar{F}: \mP_\mV(\mC) \to \mD $ the unique $\mV$-enriched left adjoint extension of $F$.
For every $\mV$-enriched functor $G: \mP_\mV(\mC) \to \mD $
the induced morphism $$\Map_{{\mV\mathrm{-}\Fun^\L}(\mP_\mV(\mC),\mD)}(\bar{F},G) \to \Map_{\mV\mathrm{-}\Fun(\mC,\mD)}(F,G \circ \iota_\mC) $$ is an equivalence.
\item Let $F: \mC \to \mD$ be a $\mV$-enriched functor and
$\bar{F}: \mP_\mV(\mC) \to \mD $ the unique $\mV$-enriched left adjoint extension of $F$.
The $\mV$-enriched right adjoint $\mD \to \mP_\mV(\mC)$ of $\bar{F}$ sends $Y$ to $\Mor_\mD(-,Y) \circ F. $
\end{enumerate}
\end{theorem}

The following is the enriched Yoneda-lemma proven in 
\cite[Corollary 4.44]{heine2024bienriched}:

\begin{lemma}
Let $\mV$ be a presentably monoidal category and $\mC$ a small univalent $\mV$-enriched $\infty$-category. For every object $X \in \mC$ and $F \in \mP_\mV(\mC)$ the induced morphism $$\Mor_{\mP_\mV(\mC)}(\Mor_\mC(-,X),F) \to F(X) $$ is an equivalence.

\end{lemma}

We will use the following terminology to describe compatible enrichments in two monoidal categories:

\begin{definition}\label{bienr}
Let $\mV$ and $\mW$ be presentably monoidal categories.	 
\begin{enumerate}[\normalfont(1)]\setlength{\itemsep}{-2pt}
\item A left $\mV$-enriched category is a $\mV$-enriched category.
\item A left $\mV$-enriched functor is a $\mV$-enriched functor
\item A right $\mV$-enriched category is a $\mV^\rev$-enriched category.
\item A right $\mV$-enriched functor is a $\mV^\rev$-enriched functor.
\item A $\mV,\mW$-bienriched category is a $\mV \ot \mW^\rev$-enriched category.
\item A $\mV,\mW$-enriched functor is a $\mV \ot \mW^\rev$-enriched functor.
\end{enumerate}
\end{definition}




We often refer to a $(\mV,\mV)$-bienriched category simply as $\mV$-bienriched category.

\begin{notation}
Let $\mV, \mW$ be presentably monoidal categories.
\begin{enumerate}[\normalfont(1)]\setlength{\itemsep}{-2pt}
\item Let $${_\mV \Cat}:= {\mV}\mathrm{-}\Cat$$
be the 2-category of left $\mV$-enriched categories and left $\mV$-enriched functors.
\item Let $${\Cat_\mW}:= {\mW^\rev}\mathrm{-}\Cat$$
be the 2-category of right $\mW$-enriched categories and right $\mW$-enriched functors.
\item Let $${_\mV \Cat_\mW}:= {\mV \ot \mW^\rev}\mathrm{-}\Cat$$
be the 2-category of $\mV,\mW$-bienriched categories and $\mV,\mW$-enriched functors.
\end{enumerate}
\end{notation}

\begin{notation}

We will write $\L\Mor, \R\Mor, \mathrm{B}\Mor$ for the morphism objects in a left, right and bienriched category, respectively.    
\end{notation}

\begin{notation}
Let $\mV, \mW$ be presentably monoidal categories and $\mM, \mN$ be $(\mV,\mW)$-bienriched categories.

Let $$_\mV\Fun_\mW(\mM,\mN):= {\mV \ot \mW^\rev \mathrm{-}\Fun(\mM,\mN)}$$ be the category of $\mV,\mW$-enriched functors $\mM \to \mN.$

\end{notation}

The following is \cite[Example 2.134.]{heine2024bienriched}:

\begin{example}
Let $\mV,\mW$ be presentably monoidal categories. Every presentably $\mV,\mW$-bitensored category is a $\mV,\mW$-enriched category.
Every presentably right $\mW$-tensored category is a right $\mW$-enriched category.
This follows also from \cref{linenr} identifying
presentably $\mV,\mW$-bitensored categories with presentably left $\mV \ot \mW^\rev$-tensored categories, and identifying
presentably right $\mW$-tensored categories with presentably left $\mW^\rev$-tensored categories.
In particular, $\mV$, which is presentably bitensored over itself, is a $\mV,\mV$-bienriched category.

Similarly, every $\mV,\mW$-linear functor between presentably $\mV,\mW$-bitensored categories is a $\mV,\mW$-enriched functor.

\end{example}

\begin{definition}Let $\mV$ and $\mW$ be presentably monoidal categories, let $\mC$ be a 
$\mV,\mW$-enriched category, and let $X \in \mC, V \in \mV, W \in \mW.$	 
\begin{enumerate}[\normalfont(1)]\setlength{\itemsep}{-2pt}
\item The left (co)tensor of $V$ and $X$ in $\mC$ is the (co)tensor of
$V \ot \tu_\mW \in \mV \ot \mW^\rev$ and $X$ in $\mC.$

\item The right (co)tensor of $W$ and $X$ in $\mC$ is the (co)tensor of
$\tu_\mV \ot W \in \mV \ot \mW^\rev$ and $X$ in $\mC.$



\end{enumerate}

\end{definition}

In the following we consider enriched slice categories.
To define bienriched slice categories we use the following lemma and corollary:

\begin{lemma}\label{tensorunit}
Let $\mV, \mW$ be presentably monoidal categories whose tensor unit is final.
The tensor unit of $\mV \ot \mW$ is final.
	
\end{lemma}

\begin{proof}

By \cite[Proposition 7.15]{FreeAlgebras} there are 
small regular cardinals $\kappa, \tau$ such that
the monoidal structures on $\mV,\mW$ restrict to the full subcategories $\mV^\kappa \subset \mV , \mW^\tau \subset \mW$ of $\kappa, \tau$-compact objects, respectively.
The latter monoidal embeddings induce monoidal equivalences
$\Ind_\kappa(\mV^\kappa) \simeq \mV, \Ind_\kappa(\mW^\kappa) \simeq \mW$.
The monoidal localizations $\mP(\mV^\kappa) \rightleftarrows \Ind_\kappa(\mV^\kappa) \simeq \mV,$
$\mP(\mW^\tau) \rightleftarrows \Ind_\tau(\mW^\kappa) \simeq \mW $
induce a monoidal localization 
$\mP(\mV^\kappa \times \mW^\tau) \simeq \mP(\mV^\kappa) \ot \mP(\mW^\tau) \rightleftarrows \mV \ot \mW . $
By assumption the tensor unit of $ \mV^\kappa \times \mW^\tau $ is final.
Since the Yoneda embedding is monoidal and preserves the final object, the tensor unit of $\mP(\mV^\kappa \times \mW^\tau) $ is final and so local for any localization.
So the tensor unit of $\mP(\mV^\kappa \times \mW^\tau) $ belongs to $\mV \ot \mW.$
Since there is a monoidal localization functor $\mP(\mV^\kappa \times \mW^\tau) \to \mV \ot \mW$,
the tensor unit of $\mV \ot \mW$ is equivalent to the tensor unit of $\mP(\mV^\kappa \times \mW^\tau) $, which is final in $\mP(\mV^\kappa \times \mW^\tau) $ and so in $\mV \ot \mW.$ 
\end{proof}

\begin{corollary}\label{finality}
	
Let $\mV, \mW$ be presentably monoidal categories whose tensor unit is final.
Then ${_\mV \Cat_\mW}$ admits a final object $*$.
	
\end{corollary}

\begin{corollary}Let $\mV, \mW$ be presentably monoidal categories whose tensor unit is final.
For every $(\mV,\mW)$-bienriched category $\mC$
the induced functor ${_\mV\Fun_\mW}(*,\mC) \to \mC$ is an equivalence.
    
\end{corollary}

\subsection{Enriched slices, wedges, and suspensions}

For the next notation we that the 2-category $\mV\mathrm{-}\Cat$ admits cotensors.

\begin{notation}
Let $\mV$ be a presentably monoidal category whose tensor unit is final, $\mC$ a $\mV$-enriched category and $X $ an object of $\mC.$
Let $$ \mC_{\X/}:= \{X\} \times_{\mC^{\{0\}}} \mC^{\bD^1} $$ be the pullback of the $\mV$-enriched functor $\mC^{\bD^1} \to \mC^{\{0\}}$ evaluating at the source along the $\mV$-enriched functor $* \to \mC$ classifying $X$. 
By definition there is a $\mV$-enriched functor $\mC_{\X/}\to \mC$.
    
\end{notation}

\begin{remark}\label{init}
Let $\mV$ be a presentably monoidal category whose tensor unit is final, $\mC$ a $\mV$-enriched category and $\X \to \Y, \X \to \Z$ morphisms in $\mC.$
The induced morphism $$\tu \to \Mor_{\mC_{\X/}}(\X,\X) \to \Mor_{\mC_{\X/}}(\X,\Z)$$ is an equivalence and the resulting commutative square
$$\begin{xy}
\xymatrix{
\Mor_{\mC_{\X/}}(\Y,\Z) \ar[d]^{} \ar[r]
& \Mor_{\mC}(\Y,\Z) \ar[d] \ar[d]^{}
\\ 
\tu \simeq \Mor_{\mC_{\X/}}(\X,\Z) \ar[r] & \Mor_{\mC}(\X,\Z)
}
\end{xy}$$
is a pullback square.
    
\end{remark}

\begin{lemma}

Let $\mV$ be a presentably monoidal category whose tensor unit is final, $\mC$ a $\mV$-enriched category and $X $ an object of $\mC.$
If $\mC$ admits small weakly contractible conical colimits, then $\mC_{\X/}$ admits small colimits.

\end{lemma}

\begin{proof}

If the $\mV$-enriched category $\mC$ admits small conical weakly contractible colimits, the underlying category of $\mC_{\X/}$ admits small weakly contractible colimits.
Moreover the description of morphism objects of \cref{init} implies that the functor $$\Mor_{\mC_{\X/}}(-,\Z): (\mC_{\X/})^\circ \to \mV \otimes \mW$$ preserves weakly contractible limits.
Hence the $\mV$-enriched category $\mC_{\X/}$ admits
small conical weakly contractible colimits.
By \cref{init} the $\mV$-enriched category $\mC_{\X/}$ admits a conical initial object. This implies the result.
\end{proof}

\begin{remark}
Let $\mV$ be a presentably monoidal category whose tensor unit is final, $\mC, \mD$ be $\mV$-enriched categories, $F: \mC \to \mD$ a $\mV$-enriched functor 
and $Y \to X $ a morphism in $\mC.$
There is an induced $\mV$-enriched functor 
$\mC_{\X/} \to \mD_{F(\Y)/}$ that fits into a commutative square 
$$\begin{xy}
\xymatrix{
\mC_{\X/} \ar[d]^{} \ar[r]
&  \mD_{F(\Y)/} \ar[d] \ar[d]^{}
\\ 
\mC
\ar[r]^F & \mD
}
\end{xy}$$	
of $\mV$-enriched categories.
    
\end{remark}

\begin{lemma}\label{bien}
Let $\mV$ be a presentably monoidal category whose tensor unit is final, $\mC$ a $\mV$-enriched category and $X \to Y$ a morphism in $\mC.$
If $\mC$ admits conical pushouts, the $\mV$-enriched functor $\mC_{\Y/} \to \mC_{\X/}$ admits a $\mV$-enriched left adjoint.
	
\end{lemma}

\begin{proof}

Let $\X \to \W, \Y \to \T $ a morphism in $\mC.$
By assumption there is a conical pushout $ \Y \coprod_\X \W$.
The canonical morphism $\Y \to \Y \coprod_\X \W$ induces the morphism
$ \Mor_{\mC_{\Y/}}(\Y \coprod_\X \W,\T) \to  \Mor_{\mC_{\X/}}(\W,\T),$
which factors as equivalences $$ \Mor_{\mC_{\Y/}}(\Y \coprod_\X \W,\T) \simeq \tu \times_{\Mor_{\mC}(\Y,\T) } \Mor_{\mC}(\Y \coprod_\X \W,\T) \simeq $$$$ \tu \times_{\Mor_{\mC}(\Y,\T)} \Mor_{\mC}(\Y,\T) \times_{\Mor_{\mC}(\X,\T)} \times_{\Mor_{\mC}(\W,\T)} \simeq \Mor_{\mC_{\X/}}(\W,\T).$$
\end{proof}


The following is \cite[Corollary 2.2.2.]{GepnerHeine2026}:

\begin{corollary}\label{susp}

Let $\mV$ be a presentably monoidal category and $\A \in \mV$.
There is a $\mV$-enriched category $S(\A) $, which we call the suspension of $A$, satisfying the following properties:
\begin{enumerate}[\normalfont(1)]\setlength{\itemsep}{-2pt}
\item The space of objects of $S(\A) $ is the set $\{0,1\}.$

\item For every $0 \leq \ell \leq 1$ the unit $\tu \to \Mor_{S(\A)}(\ell,\ell)$ is an equivalence.

\item The morphism object $\Mor_{S(\A)}(1,0)$ is initial.

\item There is an equivalence $ \Mor_{S(\A)}(0,1) \simeq \A$. 

\item For every $\mV$-enriched category $\mC $ and objects $\X,\Y$ of $\mC$
the induced map is an equivalence $$ \Map_{\mV\mathrm{-}\Cat_{B(\tu) \coprod B(\tu)/}}(S(\A), (\mC; \X,\Y)) \to \Map_\mV(\A, \Mor_\mC(\X,\Y)).$$
\end{enumerate}
    
\end{corollary}

\begin{definition}
Let $\mV$ be a presentably monoidal category, $\mC, \mD$ be $\mV$-enriched categories and $X,Y \in \mC, Y,Z \in \mD$
objects. The bipointed wedge $(\mC; X,Y) \vee (\mD; Y,Z) $
is the pushout $ (\mC \coprod_{\{Y\}} \mD; X,Z).$
    
\end{definition}

\begin{corollary}

Let $\mV$ be a presentably monoidal category, $n \geq 1$ and $\A_1, ..., \A_\n \in \mV$.
\begin{enumerate}[\normalfont(1)]\setlength{\itemsep}{-2pt}
\item The space of objects of the $\mV$-enriched category $S(A_1) \vee ...\vee S(A_n)$ is the set $\{0,...,\n\}.$
\item For every $0 \leq \ell \leq \n$ the unit $\tu \to \Mor_{S(A_1) \vee ...\vee S(A_n)}(\ell,\ell)$ is an equivalence.
\item For every $0 \leq \bk < \ell \leq \n$ the morphism object $\Mor_{S(A_1) \vee ...\vee S(A_n)}(\ell,\bk)$ is initial.
\item For every $0 \leq \ell < \n$ there is an equivalence $ \Mor_{S(A_1) \vee ...\vee S(A_n)}(\ell,\ell+1) \simeq \A_{\ell+1}$. 
\item For every $0 \leq \bk < \m \leq \n$ the following induced morphism is an equivalence $$\bigotimes_{\bk \leq \ell < \m}  \A_{\ell+1} \simeq \bigotimes_{\bk \leq \ell < \m} \Mor_{S(A_1) \vee ...\vee S(A_n)}(\ell,\ell+1) \to \Mor_{S(A_1) \vee ...\vee S(A_n)}(\bk,\m).$$
\end{enumerate}
\end{corollary}

\subsection{Reduced enrichment}

\begin{notation}
For every category $\mC$ let $\mC_* \subset \Fun(\bD^1,\mC)$ be the full subcategory of morphisms $\A \to \B$ in $\mC$ such that $\A$ is a final object. 
\end{notation}

\begin{remark}
If $\mC$ has a final object, $\mC_*$ is the fiber over the final object of evaluation at the source $\Fun(\bD^1,\mC) \to \mC$.	
Moreover if $\mC$ has a final object and cofibers, the embedding $\mC_* \subset \Fun(\bD^1,\mC)$ admits a left adjoint that sends $\phi: \A \to \B$ to the cofiber $\cofib(\phi)$ of $\phi$ since the morphism $\B \to \cofib(\phi)$
induces for every morphism $\alpha: * \to \Z$ in $\mC$ an equivalence	
$ \mC_*(\cofib(\phi),\Z) \to \Fun(\bD^1,\mC)(\phi, \alpha).$

\end{remark}

\begin{definition}
Let $\mV$ be a monoidal category compatible with finite colimits. The pushout product 
monoidal structure on $ \Fun(\bD^1,\mV)$ is the Day-convolution,
where $\bD^1$ carries the cartesian structure.
\end{definition}
\begin{remark}\label{saso}

By definition of the Day-convolution the pushout product of morphisms $f_1: \A_1 \to \B_1, ..., f_\n:\A_\n \to \B_\n$ in $\mV$ for $\n \geq 0$ is the left Kan extension of the functor
$ (\bD^1)^{\times \n} \xrightarrow{f_1 \times ...\times f_\n} \mV^{\times\n} \xrightarrow{\ot} \mV$
along the functor $ (\bD^1)^{\times \n} \to \bD^1$ taking the minimum.
This left Kan extension is the induced morphism $$\colim_{(\bD^1)^{\times \n} \setminus \{(1,...,1)\}}(\ot \circ (f_1 \times ...\times f_\n)) \to \B_1 \ot ... \ot \B_\n.$$

So the pushout product of two morphisms $f: \A \to \B, \g: \A' \to \B'$ is the morphism $$\B \ot \A' \coprod_{\A \ot \A'} \A \ot \B' \to \B \ot \B'.$$

The tensor unit for the pushout product is the morphism $\emptyset \to \tu$.
So the pushout product monoidal structure is compatible with the same sort of colimits $\mV$ is compatible with.

\end{remark}

The pushout product gives rise to the smash product:

\begin{lemma}\label{smash}
Let $\mV$ be a monoidal category compatible with finite colimits that admits a final object. The localization $$ \Fun(\bD^1,\mV) \to \mV_*$$ 
is compatible with the pushout product monoidal structure.
Consequently, there is a unique monoidal structure on $\mV_*$ compatible with the same sort of colimits $\mV$ is compatible with, such that the functor $$\cofib: \Fun(\bD^1,\mV) \to \mV_*$$ forming the cofiber is monoidal, where the left hand side carries the pushout product monoidal structure. 
\end{lemma}

\begin{proof}
We need to see that the pushout product of two local equivalences 
is a local equivalence. For this it is enough to see that
for every objects $f: \A \to \B, \g: \X \to \Y$ in $\Fun(\bD^1,\mV)$ and any local equivalence $\sigma: g \to h$ in $\Fun(\bD^1,\mV)$ to a local object $h : * \to \Z$ in $\Fun(\bD^1,\mV)$
the induced map from the pushout product of $f, g $ to the one of $f, h$
and from the pushout product of $g, f $ to the one of $h, f$ induces an equivalence on cofibers.
We prove the case of the first map, the case of the second map is similar.
The first map identifies with the bottom right hand square in the following diagram:
$$\begin{xy}
\xymatrix{
\A \ot \X \ar[d] \ar[r] & \B \ot \X \ar[d] \ar[r]
& \B \ot * \ar[d]
\\ 	
\A \ot \Y \ar[r]& \A \ot \Y \coprod_{\A \ot \X} \B \ot \X\ar[d] \ar[r]
& \A \ot \Z \coprod_{\A \ot *} \B \ot * \ar[d]
\\ 
& \B \ot \Y \ar[r] & \B \ot \Z.
}
\end{xy}$$ 
We need to see that the induced map on cofibers of the vertical morphisms of the bottom right hand square is an equivalence. For this it is enough to see that the bottom right hand square
is a pushout square. 
By the pasting law this holds if the upper and outer right hand squares are pushout squares. The outer right hand square is a pushout square since the tensor product preserves pushouts component-wise and the canonical map from the cofiber of $\g: \X \to \Y$ to $\Z$ is an equivalence as $\sigma: g \to h$ induces an equivalence on cofibers by the assumption that $\sigma$ is a local equivalence.
The upper right hand square is a pushout square if the upper outer square is a pushout square since the upper left hand square is a pushout square.
The upper outer square is the outer square in the diagram
$$\begin{xy}
\xymatrix{
\A \ot \X \ar[d]\ar[r] & \A \ot * \ar[d] \ar[r]
& \B \ot * \ar[d]
\\ 	
\A \ot \Y \ar[r] & \A \ot \Z \ar[r]
&\A \ot \Z \coprod_{\A \ot *} \B \ot *. 
}
\end{xy}$$ 
The right hand square is a pushout square by definition. The left hand square is a pushout square since the tensor product preserves pushouts component-wise
and $\sigma$ is a local equivalence.
\end{proof}

\begin{definition}\label{smasho}
Let $\mV$ be a monoidal category compatible with finite colimits that has a final object. 
The smash product monoidal structure on $\mV_*$ is the monoidal structure of \cref{smash}.
\end{definition}

\begin{remark}\label{hiso}
\begin{enumerate}[\normalfont(1)]\setlength{\itemsep}{-2pt}
\item For every morphisms $f_1: * \to \Z_1,...,f_\n: * \to \Z_\n$ in $ \mV$ for $\n \geq 0$ the smash product $\Z_1 \wedge ... \wedge \Z_\n$ is the cofiber of the pushout product of $ f_1,..., f_\n$ and the tensor unit for the smash product is   
$\tu \coprod *,$ the cofiber of the tensor unit $\emptyset \to \tu$ for the pushout product.
\item If the tensor unit of $\mV$ is final, the pushout product of any morphisms $ f_1,..., f_\n$ is the morphism 
$$ \colim_{\T \subsetneq \{1,...,\n\}}(  \ot(\Z_\bj \mid \bj \in \T)  ) \to \Z_1 \ot ... \ot \Z_\n.$$
In this case the smash product $\Z \wedge \Z'$ of morphisms $* \to \Z, *\to \Z'$ is the cofiber of the morphism $$\Z \vee \Z'= \Z \coprod_* \Z' \to \Z \ot \Z'.$$
\end{enumerate}	
\end{remark}

\begin{notation}
Let $\mC$ be a category that admits finite coproducts.
The forgetful functor $\mC_* \to \mC$ admits a left adjoint $(-)_+: \mC \to \mC_*$
that sends $\X$ to $\X \coprod \bD^0.$	
\end{notation}

\begin{lemma}

Let $\mV$ be a monoidal category compatible with finite colimits that admits a final object.
The forgetful functor $\mV_* \to \mV$ is lax monoidal and the left adjoint $(-)_+:\mV \to \mV_*$
is monoidal.		

\end{lemma}

\begin{proof}
By properties of Day-convolution \cite[Remark 2.24 2]{heine2017topological} the functor $\Fun(\bD^1,\mV) \to \mV$ evaluating at the target is lax monoidal and the left adjoint, which sends $\X$ to $\emptyset \to \X$, is monoidal.
By \cref{smasho} the lax monoidal embedding $\mV_* \subset \Fun(\bD^1,\mV)$ admits a monoidal left adjoint.
Thus the forgetful functor $\mV_* \to \mV$, which factors as $\mV_* \subset \Fun(\bD^1,\mV) \to \mV $, is lax monoidal and the left adjoint $(-)_+:\mV \to \mV_*$, which factors as $\mV \to  \Fun(\bD^1,\mV) \to \mV_*$, is monoidal.	
\end{proof}

\begin{lemma}\label{leyy} Let $\mC$ be a monoidal category compatible with finite colimits that admits a final object, $\phi: \A \to \B$ a morphism in $\mC$ and $\Y \in \mC.$	
There are canonical equivalences
$$\cofib(\phi) \wedge \Y_+ \simeq (\B \ot \Y) / (\A \ot \Y) , \Y_+ \wedge \cofib(\phi) \simeq (\Y \ot \B) / (\Y \ot \A).$$

\end{lemma}

\begin{proof}We prove the first case. The second case is dual.	
By \cref{smasho} the object $\cofib(\phi) \wedge \Y_+$ is the cofiber of the morphism
$$(\A\ot \Y)  \coprod (\B\ot*) \simeq (\A\ot \Y)  \coprod (\A \ot *) \coprod_{(\A \ot *)} (\B \ot *) \simeq \A\ot \Y_+ \coprod_{(\A \ot *)} (\B \ot *)  \to \B \ot \Y_+ \simeq (\B \ot \Y) \coprod (\B\ot*) $$ in $\mC,$
which is the cofiber of the morphism $\A \ot \Y \to \B \ot \Y.$
\end{proof}

\begin{lemma}\label{siro}
Let $\mC$ be a monoidal category compatible with finite colimits that admits a final object and $\alpha:* \to \A,\beta: *\to \B$ morphisms of $\mC$ such that $\A,\B$ and $*,\B$ admit a morphism object $\Mor_\mC(\A,\B), \Mor_\mC(*,\B)$ in $\mC.$
Then the pullback
$$\begin{xy}
\xymatrix{
\Mor_{\mC_*}(\A,\B) \ar[d] \ar[r]
& \Mor_\mC(\A,\B) \ar[d]^{\alpha^*}
\\ 
\ast \ar[r]^\beta & \Mor_\mC(*,\B).
}
\end{xy}$$
pointed by the constant functor $\kappa: \A \to * \xrightarrow{\beta}\B$ is a morphism object of $\alpha, \beta$ for the smash product of $\mC_*$.

\end{lemma}



\begin{notation}
Let $\mV$ be a presentably monoidal category.
Let $$ \mV\mathrm{-}\Cat^\red \subset \mV\mathrm{-}\Cat$$ be the subcategory of $\mV$-enriched categories that admit a zero object and $\mV$-enriched functors preserving the zero object.

\end{notation}



\begin{proposition}\label{reduct}
Let $\mV$ be a presentably monoidal category that admit a zero object.
A $\mV$-enriched category admits a zero object if it admits an initial or final object. 

\end{proposition}

\begin{proof}
By duality it suffices to see the case of final object.
Embedding into a presentably left $\mV$-tensored category via the $\mV$-enriched Yoneda embedding, which preserves the final object, it suffices to see that any presentably left $\mV$-tensored category
admits a zero object if it admits a final object. The unique morphism
$ \tu \to 0$ in $\mV$ induces a morphism $* \simeq \tu \ot * \to 0 \ot * \simeq \emptyset$
that is necessarily an inverse of the unique morphism $\emptyset \to *.$
\end{proof}







\begin{proposition}\label{red}Let $\mV$ be a presentably monoidal category.
The functor $\nu: \mV_*\mathrm{-}\Cat \to \mV\mathrm{-}\Cat $ restricting along the monoidal functor $(-)\coprod * : \mV \to \mV_*$
restricts to an equivalence $ \mV_*\mathrm{-}\Cat^\red \to \mV\mathrm{-}\Cat^\red.$

\end{proposition}

\begin{proof}
Since the functor $\nu: {\mV_{\tu/}\mathrm{-}\Cat} \to \mV\mathrm{-}\Cat $ is conservative, it suffices to show that $\nu$ admits a fully faithful left adjoint $\phi.$
For every $\mV$-enriched category $\mC$ that admits a zero object let $\mP_\mV(\mC)' \subset \mP_\mV(\mC)$
be the full subcategory of $\mV$-enriched presheaves preserving the final object.
An object of $ \mP_\mV(\mC)$ belongs to $ \mP_\mV(\mC)' $ if and only if it is local with respect to the morphism $\emptyset \to *$.  
So $\mP_\mV(\mC)' \subset \mP_\mV(\mC)$ is an accessible $\mV$-enriched localization and $\mP_\mV(\mC)'$ is a presentably left $\mV$-tensored category. 

By the enriched Yoneda-lemma \cite[Corollary 5.23.]{heine2024bienriched} the $\mV$-enriched category $\mP_\mV(\mC)'$ admits a zero object.
The category $\mP_\mV(\mC)' \ot \mS_* \simeq \mP_\mV(\mC)'_* \simeq \mP_\mV(\mC)'$ is presentably left tensored
over $\mV \ot \mS_* \simeq \mV_*$.
The $\mV$-enriched Yoneda embedding $\mC \subset \mP_\mV(\mC)$ lands in $\mP_\mV(\mC)'$ 
and the resulting $\mV$-enriched embedding $\mC \subset \mP_\mV(\mC)'$ preserves the zero object.
Moreover the $\mV$-enriched localization $\mP_\mV(\mC) \rightleftarrows \mP_\mV(\mC)'$
induces a $\mV_*$-enriched localization $$\mP_\mV(\mC)_* \simeq \mP_\mV(\mC) \ot \mS_* \rightleftarrows \mP_\mV(\mC)' \ot \mS_* \simeq \mP_\mV(\mC)'$$ 
whose local objects are local with respect to the morphism $\emptyset \to *$. 

Let $\phi(\mC)$ be the essential image of the induced embedding $\mC \subset \mP_\mV(\mC)'.$
Then $\phi(\mC)$ is a $\mV_*$-enriched category whose underlying $\mV$-enriched category is $\mC,$ and that therefore admits a zero object.
We complete the proof by showing that for every $\mV_*$-enriched category $\mD$ that admits a zero object the induced functor
$$ {_{\mV_*}\Fun_*}(\phi(\mC),\mD) \to {_\mV\Fun_*(\mC,\mD)}$$ is an equivalence.
The latter functor is the pullback of the functor
$$\Psi: {_{\mV_*}\Fun_*}(\phi(\mC),\mP_{\mV_*}(\mD)) \to {_\mV\Fun_*(\mC,\mP_{\mV_*}(\mD))}.$$
The $\mV_*$-enriched localization functor $\mP_\mV(\mC)_* \to \mP_\mV(\mC)'$ 
induces an embedding $$ {_{\mV_*}\Fun}^\L(\mP_\mV(\mC)',\mP_{\mV_*}(\mD)) \hookrightarrow {_{\mV_*}\Fun}^\L(\mP_\mV(\mC)_*,\mP_{\mV_*}(\mD)) \simeq {_\mV\Fun(\mC,\mP_{\mV_*}(\mD))}$$
whose essential image is $ {_\mV\Fun_*(\mC,\mP_{\mV_*}(\mD))}. $ 
The latter functor factors as 
$$ {_{\mV_*}\Fun}^\L(\mP_\mV(\mC)',\mP_{\mV_*}(\mD)) \xrightarrow{\alpha} {_{\mV_*}\Fun_*}(\phi(\mC),\mP_{\mV_*}(\mD))\xrightarrow{\Psi} {_\mV\Fun_*(\mC,\mP_{\mV_*}(\mD))}.$$
So it remains to see that $\alpha$ is an equivalence.
For that it suffices to show that the left adjoint $\mV_*$-enriched functor $\theta: \mP_{\mV_*}(\phi(\mC))' \to \mP_{\mV}(\mC)'$ induced by the $\mV_*$-enriched embedding $\phi(\mC) \subset \mP_{\mV}(\mC)'$ is an equivalence.
This holds because $\mP_{\mV}(\mC)'$ is generated by the representables under small colimits and tensors, which are preserved by $\theta,$ and for every
$\X \in \mC$ the $\mV_*$-enriched functor $\Mor_{\mP_{\mV_*}(\phi(\mC))'}(\X,-): \mP_{\mV_*}(\phi(\mC))' \to \mV_*$ and the $\mV$-enriched functor $\Mor_{\mP_{\mV}(\mC)'}(\X,-): \mP_{\mV}(\mC)' \to \mV$ preserve small colimits and tensors. The latter holds since by \cite[Theorem 3.50.]{heine2024bienriched} the $\mV_*$-enriched functor $\Mor_{\mP_{\mV_*}(\phi(\mC))}(\X,-): \mP_{\mV_*}(\phi(\mC)) \to \mV_*$ and the $\mV$-enriched functor $\Mor_{\mP_{\mV}(\mC)}(\X,-): \mP_{\mV}(\mC) \to \mV$ preserve small colimits and tensors and so preserve the local equivalence $\emptyset \to *.$
\end{proof}


\subsection{Density}

\begin{notation}
Let $\mV^\ot \to \Delta^\op$ be a monoidal category.
We write $\mV_\ot \to \Delta$ for the cartesian fibration classifying the functor $\Delta^\op \to \Cat$ classified by the cocartesian fibration $\mV^\ot \to \Delta^\op$.
    
\end{notation}

\begin{notation}

Let $\mV$ be a small monoidal category.
There is a left adjoint functor 
$$ \mP(\mV_\ot) \to {\mP(\mV)}\mathrm{-}\Cat$$
that sends a representable presheaf $(X_1,...,X_n) \in \mV^{\times n} \simeq (\mV_\ot)_{[n]}$ for $n \geq 1$ to
$$S(X_1,\ldots,X_n):=S(X_1) \vee ... \vee S(X_n)$$
and sends the unique object in the fiber over $[0]$ to $B\tu_\mV.$
\end{notation}

\begin{definition}

Let $\mV$ be a small monoidal category.
A presheaf on $\mV_\ot $ satisfies the Segal condition 
if for every $n \geq 2$ and $X_1,...,X_n \in \mV$
the following canonical map is an equivalence:
$$ F(S(X_1) \vee\cdots\vee S(X_n)) \to F(S(X_1)) \times_{F(B\tu_\mV)}\times\cdots \times_{F(B\tu_\mV)} F(S(X_n)).$$
    
\end{definition}

\begin{notation}

Let $\mV$ be a small monoidal category.
Let 
$ \mP_\Seg(\mV_\ot) \subset \mP(\mV_\ot)$
be the full subcategory of presheaves satisfying the Segal condition.
\end{notation}

\begin{lemma}\label{Loc}

Let $\mV$ be a small monoidal category.
Every representable presheaf on $\mV_\ot$ satisfies the Segal condition.
    
\end{lemma}

\begin{proof}

Let $m$ and $n$ be natural numbers and choose $(Y_1,..., Y_m) \in (\mV_\ot)_{[m]}$ and
$(X_1,...,X_n) \in (\mV_\ot)_{[n]}$.
For every $0 \leq \bi \leq n$ let $\tau_\bi: \{\bi-1 < \bi \} \subset [n]$ be the canonical embedding.
We need to see that the map $$\Map_{\mV_\ot}((X_1,...,X_n),(Y_1,...,Y_m)) \to $$$$ \Map_{\mV_\ot}(X_1,(Y_1,...,Y_m)) \times_{\Map_{\Delta}([0],[m])}\times\cdots \times_{\Map_{\Delta}([0],[m])} \Map_{\mV_\ot}(X_n,(Y_1,...,Y_m)) $$
is an equivalence.
The latter is a map over $$ \Map_{\Delta}([n],[m]) \simeq \Map_{\Delta}([1], [m]) \times_{\Map_{\Delta}([0],[m])}\times\cdots \times_{\Map_{\Delta}([0],[m])} \Map_{\Delta}([1], [m])$$ and induces on the fiber over any $\phi$
the canonical equivalence. 
$$\Map_{(\mV^\ot)_{[n]}^\vee}((X_1,...,X_n),\phi^*(Y_1,...,Y_m)) \to $$$$
\Map_{(\mV^\ot))_{[1]}^\vee}(X_1,\tau_1^*(\phi^*(Y_1,...,Y_m))) \times\cdots\times \Map_{(\mV^\ot)_{[1]}^\vee}(X_n,\tau_n^*(\phi^*(Y_1,...,Y_m)))$$
induced by the equivalence $ (\tau^*_1, ..., \tau^*_n):(\mV_{[n]}^\ot)^\vee \to \mV^{\times n}.$
\end{proof}

\begin{notation}\label{Thetasgen}
By \cref{Loc} the Yoneda embedding of $\mV_\ot$
lands in $\mP_\Seg(\mV_\ot).$
Let $S$ be the embedding $\mV_\ot \subset \mP_\Seg(\mV_\ot) \to {\mP(\mV)}\mathrm{-}\Cat$.
Let $\Theta(\mV)=\mV\wreath\Delta^{\op}$ be the essential image of $S$, the full subcategory 
spanned by the wedges of suspensions of objects of $\mV$ and $B\tu_\mV$.

\end{notation}

\begin{definition}Let $\mV$ be a small monoidal category, $F$ a presheaf on $\Theta(\mV)$ and $A,B \in F(B\tu_\mV). $ 
The presheaf of morphisms $A$ to $B$ in $F$ is the following presheaf on $\mV$, where $S: \mV \to \mV\mathrm{-}\Cat_{B\tu_\mV \coprod B\tu_\mV/}$ is the suspension:
$$\Mor_F(A,B):= (F \circ S) \times_{F(B\tu_\mV) \times F(B\tu_\mV)} \{(A,B)\}.$$
    
\end{definition}

\begin{theorem}\label{denseinherited} Let $\mW$ be a presentably monoidal category
and $\mV$ a small dense full monoidal subcategory of $\mW$.
The full subcategory $\Theta(\mV)$ of $\mV\mathrm{-}\Cat$ spanned by the wedges of suspensions of objects of $\mV$ and $B\tu_\mV$ is dense in $\mW\mathrm{-}\Cat.$
A presheaf on $\Theta(\mV) $ belongs to the essential image of the $\Theta(\mV)$-nerve
if and only if it satisfies the Segal condition and 
all morphism presheaves belong to the essential image of the $\mW$-nerve.

\end{theorem}

\begin{proof}

We first assume that $\mW=\mP(\mV).$
By \cref{Loc} the Yoneda embedding $\mV_\ot \to \mP(\mV_\ot)$ lands in the full subcategory $\mP_\Seg(\mV_\ot).$
The embedding ${\mP(\mV)}\mathrm{-}\Cat \simeq \mP_\Seg(\mV_\ot) \subset \mP(\mV_\ot)$
factors as $$ {\mP(\mV)}\mathrm{-}\Cat \subset \mP({\mP(\mV)}\mathrm{-}\Cat) \xrightarrow{S^*} \mP(\mV_\ot) $$
and so as 
$$ {\mP(\mV)}\mathrm{-}\Cat \subset \mP({\mP(\mV)}\mathrm{-}\Cat) \to \mP(\Theta(\mV)) \simeq \mP(\mV_\ot). $$

Moreover \cref{premod} implies that for every left $\mV$-enriched category $\mC$ containing two objects
$A,B$ there is a canonical equivalence of presheaves on $\mV:$
$$ \Mor_\mC(A,B) \simeq \Map_{{\mP(\mV)}\mathrm{-}\Cat}(-, \mC) \times_{\Map_{{\mP(\mV)}\mathrm{-}\Cat}(B\tu_\mV, \mC) \times \Map_{{\mP(\mV)}\mathrm{-}\Cat}(B\tu_\mV, \mC)} \{(A,B)\}. $$

This implies the general case.
If $\mV$ is a small dense full monoidal subcategory of $\mW$,
the monoidal embedding $\mV \subset \mW$
induces a left adjoint monoidal functor
$\mP(\mV) \to \mW$ whose lax monoidal right adjoint is fully faithful.
The fully faithful lax monoidal right adjoint 
gives rise to an embedding 
$\mW\mathrm{-}\Cat \to {\mP(\mV)}\mathrm{-}\Cat $ whose essential image are the categories left enriched in $\mP(\mV)$ whose morphism objects belong to the essential image of the $\mW$-nerve.
The composition $ \mW\mathrm{-}\Cat \to {\mP(\mV)}\mathrm{-}\Cat \to \mP(\Theta(\mV)) $ 
of embeddings factors as the restricted Yoneda embedding of
$_\mW \Cat$ since $\Theta(\mV) \subset \mV\mathrm{-}\Cat \subset \mW\mathrm{-}\Cat.$
\end{proof}

\begin{notation}
Let $\mL,\mR \subset \Fun(\bD^1, \Cat)$ be the full subcategories of left and right fibrations, respectively.
    
\end{notation}

\begin{notation}

Let $\mC \to \mD$ be a cocartesian fibration.
By \cite[Corollary 3.29]{heine2026local} the functor $(-)\times_{\mD}\mC:\Cat_{/\mD}\to\Cat_{/\mD}$
admits a right adjoint $\Fun^\mD(\mC,-): \Cat_{/\mD}\to\Cat_{/\mD}$.
    
\end{notation}

\begin{lemma}\label{lemum} Let $\mC \to \mD$ be a cocartesian fibration classifying a functor $\mD \to \Cat.$
There is a canonical equivalence of categories over $\mD:$
$$ \Fun^\mD(\mC,\mD \times \mS) \simeq \mD \times_{\Cat} \mL.$$
\end{lemma}

\begin{proof}
Let $\mA \to \mD$ be a functor.
Functors $\mA \to \Fun^\mD(\mC,\mD \times \mS)$
correspond to functors $\mA \times_\mD \mC \to \mS.$
Functors $\mA \to \mD \times_{\Cat} \mL$
are classified by maps of cocartesian fibrations over $\mA$ to $\mA \times_\mD \mC$ that are fiberwise left fibrations.
By \cite[Proposition 3.70.]{heine2023monadicity}
the latter correspond to left fibrations over $\mA \times_\mD \mC$ and are classified by functors $\mA \times_\mD \mC \to \mS.$
\end{proof}

\begin{lemma}
Evaluation at the target $\mL \to \Cat$ is a cocartesian and cartesian fibration. 
The cocartesian fibration $\mL \to \Cat$ classifies the functor
$ \mP: \Cat \to \mS$, and the cartesian fibration $\mL \to \Cat$ classifies the functor
$ \Fun(-,\mS): \Cat^\op \to \mS$.
\end{lemma}

\begin{proof}

Evaluation at the target $\mL \to \Cat$ is a cocartesian and cartesian fibration by \cite[Lemma 3.4]{heine2026local}.

By \cref{lemum} there is a canonical equivalence 
$\mL \simeq \Fun^\Cat(\mU,\Cat \times \mS) $ over $\Cat, $
where $\mU \to \Cat$ is the cocartesian fibration classifying the identity. 
By \cite[Proposition 7.3]{articles} the cartesian fibration $$\Fun^\Cat(\mU,\Cat \times \mS) \to \Cat$$ classifies the functor
$ \Fun(-,\mS): \Cat^\op \to \mS$. Thus the cocartesian fibration $\mL \to \Cat$ classifies the functor $ \mP: \Cat \to \mS. $
\end{proof}

\begin{notation}Let $\mC \to \mD$ be a cartesian fibration
classifing a functor $\mD^\op \to \Cat.$
Let $$ \mP^{\mD^\op}(\mC):= \mD^\op \times_{\Cat} \mR \to \mD^\op. $$

\end{notation}

\begin{corollary}\label{leim} Let $\mC \to \mD$ be a cartesian fibration.
There is a canonical equivalence over $\mD^\op:$
$$ \mP^{\mD^\op}(\mC) \simeq \Fun^{\mD^\op}(\mC^\op,\mD^\op\times \mS) .$$ 

\end{corollary}

\begin{proof}

There is a canonical equivalence
$$ \Fun^{\mD^\op}(\mC^\op,\mD^\op \times \mS) \simeq \mD^\op \times_{\Cat} \mL,$$
where the functor $\mD^\op \to \Cat$ in the pullback classified by $\mC^\op \to \mD^\op$ factors as the functor $\mD^\op \to \Cat$ classified by $\mC \to \mD$ followed by the equivalence
$(-)^\op: \Cat \simeq \Cat.$
The latter induces an equivalence $\mR \simeq \mL$ that covers the former.
We obtain a canonical equivalence
$ \Fun^{\mD^\op}(\mC^\op,\mD^\op \times \mS) \simeq \mD^\op \times_{\Cat} \mR.$
\end{proof}

\begin{corollary}\label{Clpres}
Let $\phi: \mC \to \mD$ be a cartesian fibration.
The pullback of target fibration
$\Fun(\bD^1, \mP(\mC)) \to \mP(\mC) $ along the functor $\phi^*: \mP(\mD) \to \mP(\mC)$ is a cartesian fibration
that classifies the functor 
$$ \Fun_{\mD^\op}((-)^\op, \mP^{\mD^\op}(\mC)) : \mP(\mD)^\op \simeq (\mR_{\mD})^\op \to \Cat.$$

\end{corollary}

\begin{proof}

The pullback of target fibration
$\Fun(\bD^1, \mR_\mC) \to \mR_\mC$ along the functor $(-) \times_{\mD} \mC: \mR_{/\mD} \to \mR_\mC$ is 
the pullback of the cartesian fibration
$\mR \to \Cat $ along the functor $(-) \times_{\mD} \mC: \mR_{/\mD} \to \Cat$, which by \cref{leim} classifies the functor 
$$ \Fun_{\mD^\op}((-)^\op, \mP^{\mD^\op}(\mC)) : (\mR_{/\mD})^\op \to \Cat.$$
We apply the the Grothendieck construction $\mR_\mC \simeq \mP(\mC)$.
\end{proof}

Let $\mV$ be a small monoidal category.
For the following proposition recall the suspension functor
$S: \mV \to \mV\mathrm{-}\Cat$, which receives two natural transformations from the constant functor with value $B\tu_\mV.$
Let $\mV_\ot \to \Delta$ be the cartesian fibration classifying the monoidal category $\mV.$

The following proposition is an extension of \cite[Theorem 4.5.3]{MR3345192}, and the proof is a modification of the proof of the latter.

\begin{proposition}\label{premod} Let $\mV$ be a small monoidal category.
There is a canonical equivalence
$$ \mP_\Seg(\mV_\ot) \simeq {\mP(\mV)}\mathrm{-}\Cat$$
that sends a representable presheaf $(X_1,...,X_n) \in \mV^{\times n} \simeq (\mV_\ot)_{[n]}$ for $n \geq 0$ to
$S(X_1) \vee ... \vee S(X_n).$
For every $\mV$-enriched category $\mC$ containing two objects
$A,B$ there is a canonical equivalence of presheaves on $\mV:$
$$ \Mor_\mC(A,B) \simeq \Map_{{\mP(\mV)}\mathrm{-}\Cat}(S(-), \mC) \times_{\Map_{{\mP(\mV)}\mathrm{-}\Cat}(B\tu_\mV, \mC) \times \Map_{{\mP(\mV)}\mathrm{-}\Cat}(B\tu_\mV, \mC)} \{(A,B)\}. $$
    
\end{proposition}

\begin{proof}

Let $\mO \to \Delta^\op$
be the pullback of evaluation at the target $\mR \to \Cat $ along
$\mV: \Delta^\op \to \Cat.$

We apply \cref{Clpres} to the cartesian fibration
$\phi: \mV_\ot \to \Delta$. We deduce that the pullback $\mQ \to \mS $ of the cartesian fibration 
$\Fun(\bD^1,\mL_{(\mV_\ot)^\op}) \to \mL_{(\mV_\ot)^\op}$ evaluating at the target
along the functor $$\mS \to \mL_{(\mV_\ot)^\op}, \ X \mapsto (\mV_\ot)^\op \times_{\Delta^\op} \Delta_X^{\op} \to (\mV_\ot)^\op $$ is the cartesian fibration over $\mS$
classifying the functor $\mS^\op \to \Cat$ sending $X$ to 
$\Fun_{\Delta^\op}(\Delta_X^{\op}, \mO). $

An object of $\mQ$ is a small space $X$ and a map of left fibrations $\alpha: \mX \to (\mV_\ot)^\op \times_{\Delta^\op} \Delta_X^{\op} $ over $(\mV_\ot)^\op$.
There is a canonical equivalence 
$$\kappa: \mQ \simeq \mL_{(\mV_\ot)^\op} \times_{\Cocart_{\bDelta^\op}} \Fun(\bD^1,\Cocart_{\bDelta^\op}) \times_{\Cocart_{\bDelta^\op}} \mS, $$
where the functor $\mS \to \Cocart_{\bDelta^\op}$ in the pullback sends
a small space $X$ to $\Delta_X^{\op} \to \Delta^\op.$
The equivalence $\kappa$ sends a small space $X$ and a map of left fibrations $\alpha: \mX \to (\mV_\ot)^\op \times_{\Delta^\op} \Delta_X^{\op} $ over $(\mV_\ot)^\op$ to the corresponding map
$\alpha': \mX \to \Delta_X^{\op} $ of cocartesian fibrations over $\Delta^\op$.

The functor $ \Cocart_{\bDelta^\op} \to \Cat$ taking the fiber over $[0]$
is left adjoint to the functor sending a category $X$ to $\Delta_X^{\op} \to \Delta^\op.$
Consequently, there is a canonical equivalence 
$\Fun(\bD^1,\Cocart_{\bDelta^\op}) \times_{\Cocart_{\bDelta^\op}} \Cat \simeq \Cocart_{\bDelta^\op} \times_{\Cat} \Fun(\bD^1,\Cat), $
that sends a small space category $X$ and a map
$\beta: \mX \to \Delta_X^{\op} $ of cocartesian fibrations over $\Delta^\op$
to the corresponding functor $\beta_{[0]}: \mX_{[0]} \to (\Delta_X^{\op})_{[0]} \simeq X $.
Pulling back to $\mS \subset \Cat$ we obtain an equivalence
$\Fun(\bD^1,\Cocart_{\bDelta^\op}) \times_{\Cocart_{\bDelta^\op}} \mS \simeq \Cocart_{\bDelta^\op} \times_{\Cat} \Fun(\bD^1,\Cat) \times_{\Cat} \mS .$
The latter gives rise to an equivalence
$$ \mL_{(\mV_\ot)^\op} \times_{\Cocart_{\bDelta^\op}} \Fun(\bD^1,\Cocart_{\bDelta^\op}) \times_{\Cocart_{\bDelta^\op}} \mS
\simeq \mL_{(\mV_\ot)^\op} \times_{\Cocart_{\bDelta^\op}} \Cocart_{\bDelta^\op} \times_{\Cat} \Fun(\bD^1,\Cat) \times_{\Cat} \mS $$
$$ \simeq \mL_{(\mV_\ot)^\op} \times_{\Cat} \Fun(\bD^1,\Cat) \times_{\Cat} \mS\simeq \mL_{(\mV_\ot)^\op} \times_\mS (\mS \times_{\Cat} \Fun(\bD^1,\Cat) \times_{\Cat} \mS) \simeq \mL_{(\mV_\ot)^\op} \times_\mS \Fun(\bD^1,\mS).$$
So we obtain an equivalence
$$ \mQ \simeq \mL_{(\mV_\ot)^\op} \times_\mS \Fun(\bD^1,\mS)$$
that sends a map of left fibrations $\alpha: \mX \to (\mV_\ot)^\op \times_{\Delta^\op} \Delta_X^{\op} $ over $(\mV_\ot)^\op$ to the corresponding map
$$\alpha_{[0]}: \mX_{[0]} \to (\Delta_X^{\op})_{[0]} \simeq X .$$

Let $\mQ' \subset \mQ$ be the full subcategory of maps of left fibrations $\alpha: \mX \to (\mV_\ot)^\op \times_{\Delta^\op} \Delta_X^{\op} $ over $(\mV_\ot)^\op$ for some small space $X$
such that the induced map
$\alpha_{[0]}: \mX_{[0]} \to X $ is an equivalence.
Then the latter equivalence restricts to an equivalence 
$\mQ' \simeq \mL_{(\mV_\ot)^\op}$
that sends a map of left fibrations $\alpha: \mX \to (\mV_\ot)^\op \times_{\Delta^\op} \Delta_X^{\op} $ over $(\mV_\ot)^\op$ to $\mX $.
Using the Grothendieck construction $ \mL_{(\mV_\ot)^\op} \simeq \mP(\mV_\ot)$ we obtain an equivalence
$\mQ' \simeq \mP(\mV_\ot).$
Let $\mQ'' \subset \mQ'$ be the full subcategory corresponding 
under the equivalence $\mQ' \simeq \mP(\mV_\ot)$
to $\mP_\Seg(\mV_\ot).$

The restriction $\mQ' \subset \mQ \to \mS $ is a cartesian fibration
that classifies the functor $\mS^\op \to \Cat$ sending $X$ to 
the full subcategory of $\Fun_{\Delta^\op}(\Delta_X^{\op}, \mO) $
of functors $ \Delta_X^{\op} \to \mO$ sending objects in the fiber over $[0]$, which is $X$, to the contractible space in $\mO_{[0]} \simeq \mS.$
The restriction $ \mQ'' \subset \mQ \to \mS$ is a cartesian fibration classifying the functor sending $X$ to the full subcategory  
$\Fun'_{\Delta^\op}(\Delta_X^{\op}, \mO) \subset \Fun_{\Delta^\op}(\Delta_X^{\op}, \mO) $
of functors $ \Delta_X^{\op} \to \mO$ 
whose projection to $\mR$ is a monoid object,
which implies that it sends objects of $X$ to the
contractible space. 

The universal monoid $\mR^\times \to \mR$ of the cartesian structure 
induces a functor
$\Delta^\op \times_{\Cat^\times} \mR^\times \to \Delta^\op \times_{\Cat} \mR=\mO $, which by \cite[Proposition 5.1]{FreeAlgebras} induces an equivalence
$$\Alg_{\Delta_X^{\op}/\Delta^\op}(\Delta^\op \times_{\Cat^\times} \mR^\times) \simeq \bD^0 \times_{\Alg_{\Delta_X^{\op}/\Delta^\op}(\Cat^\times)} \Alg_{\Delta_X^{\op} /\Delta^\op}(\mR^\times)
\simeq $$$$ \bD^0 \times_{\Mon_{\Delta_X^{\op}}(\Cat)} \Mon_{\Delta_X^{\op}}(\mR) \simeq \Fun'_{\Delta^\op}(\Delta_X^{\op}, \mO).$$

Consequently, the cartesian fibration $ \mQ'' \to \mS$ classifies the functor sending $X$ to  
$\Alg_{\Delta_X^{\op}/\Delta^\op}(\Delta^\op \times_{\Cat^\times} \mR^\times). $

By \cite[Corollary 6.10]{FreeAlgebras} and \cite[Proposition 6.18]{FreeAlgebras}, the functor $\Delta^\op \times_{\Cat^\times} \mR^\times \to \Delta^\op$
is a monoidal category, which identifies with the Day-convolution monoidal structure on $\mP(\mV)$.
Thus the cartesian fibration $ \mQ'' \to \mS$ classifies the functor sending $X$ to 
$\Alg_{\Delta_X^{\op}/\Delta^\op}(\mP(\mV)^\ot) $, and so identifies with
${\mP(\mV)}\mathrm{-}\Cat. $
\end{proof}



\subsection{$\infty$-categories}

\begin{definition}
For every $\n \geq 0$ we inductively define the presentable cartesian closed category $\n\Cat$ of small (not necessarily univalent) $\n$-categories by setting
$$(\n+1)\Cat:= {{\n\Cat}\mathrm{-}\Cat} $$
starting with $$ 0\Cat :=\mS.$$
\end{definition}

\begin{notation}
For every $\n \geq 0$ we inductively define colocalizations
$$\n\Cat \rightleftarrows (\n+1)\Cat: \iota_\n,$$
where both adjoints  preserve finite products and filtered colimits.
Let
$$0\Cat= \mS \rightleftarrows 1\Cat = {\mS\mathrm{-}\Cat} : \iota_0 $$ be the canonical colocalization whose right adjoint assigns the space of objects. Let $$(\n+1)\Cat= {{\n\Cat}\mathrm{-}\Cat} \rightleftarrows (\n+2)\Cat= {{(\n+1)\Cat}\mathrm{-}\Cat}:\iota_{\n+1}:= (\iota_\n)_! $$
be the induced adjunction.	
	
\end{notation}

\begin{definition}The presentable category $\infty\Cat$ of small (non-univalent) $\infty$-categories is the limit
$$\infty\Cat:= \lim(\cdots\xrightarrow{\iota_{\n}} \n \Cat \xrightarrow{\iota_{\n-1}}\cdots \xrightarrow{\iota_0} 0 \Cat) $$
of presentable categories and right adjoint functors.

\end{definition}

The next proposition follows from the fact that the functor ${(-)\mathrm{-}\Cat}$ preserves small limits:

\begin{proposition}\label{fix}
	
There is a canonical equivalence
$$ \infty\Cat \simeq {\infty\Cat}\mathrm{-}\Cat. $$
	
\end{proposition}

\begin{notation}
Let $\partial\bD^1$ denote the two element set $\{0,1\}$.
\end{notation}

\begin{notation}
Let $\Mor: \infty\Cat_{\partial\bD^1/} \to \infty\Cat$
be the canonical functor $$\infty\Cat_{\partial\bD^1/} \simeq \mS_{\partial\bD^1/}\times_\mS {\infty\Cat}\mathrm{-}\Cat \to \infty\Cat $$
sending $(\mC,\X,\Y)$ to $\Mor_\mC(\X,\Y).$
\end{notation}

\begin{remark}\label{homfil}
The functor $\Mor: \infty\Cat_{\partial\bD^1/} \to \infty\Cat$
preserves small filtered colimits and limits.
 	
\end{remark}

\begin{definition}\label{ruik} Let $\n \geq 0.$
We inductively define involutions $$(-)^\op_\n, (-)^\co_\n: \n\Cat \to \n\Cat$$ by setting
$(-)^\co_0, (-)^\op_0:  0\Cat \to 0\Cat$ are the identities, $$(-)^\op_{\n+1}: {(\n+1)}\Cat \xrightarrow{(-)^\circ} {(\n+1)}\Cat \xrightarrow{((-)^\co_\n)_!}{(\n+1)}\Cat, $$ $$(-)^\co_{\n+1}:=((-)^\op_\n)_!: {(\n+1)}\Cat \xrightarrow{ }{(\n+1)}\Cat.$$
There are commutative squares:
$$\begin{xy}
\xymatrix{
{(\n+1)}\Cat \ar[d]^{\iota_\n}  \ar[r]^{(-)_{\n+1}^\op} & {(\n+1)}\Cat  \ar[d]^{\iota_\n}
\\ 
{\n}\Cat \ar[r]^{(-)_{\n}^\op} & {\n}\Cat
}
\end{xy}
\qquad
\begin{xy}
\xymatrix{
{(\n+1)}\Cat \ar[d]^{\iota_\n}  \ar[r]^{(-)_{\n+1}^\co} & {(\n+1)}\Cat  \ar[d]^{\iota_\n}
\\ 
{\n}\Cat \ar[r]^{(-)_{\n}^\co} & {\n}\Cat
}
\end{xy}
$$
and so induced involutions on the limit $$(-)^\op, (-)^\co: \infty\Cat \to \infty\Cat. $$

\end{definition}

\begin{remark}\label{oppo} By \cref{ruik} there are commutative squares, where $\sigma$ permutes the distinguished objects:

$$\begin{xy}
\xymatrix{
\infty\Cat_{\partial\bD^1/} \ar[d]^{\Mor}  \ar[r]^{(-)^\co} & \infty\Cat_{\partial\bD^1/}  \ar[d]^{\Mor}
\\ 
\infty\Cat \ar[r]^{(-)^\op} & \infty\Cat
}
\end{xy}\qquad
\begin{xy}
\xymatrix{
\infty\Cat_{\partial\bD^1/} \ar[d]^{\Mor}  \ar[r]^{\sigma \circ (-)^\op} & \infty\Cat_{\partial\bD^1/}  \ar[d]^{\Mor}
\\ 
\infty\Cat \ar[r]^{(-)^\co} & \infty\Cat.
}
\end{xy}
$$

\end{remark}

The following follows from \cref{susp}:

\begin{proposition}\label{suspi}

The functor $\Mor: \infty\Cat_{\partial\bD^1/} \to \infty\Cat$ admits a left adjoint
$S: \infty\Cat \to \infty\Cat_{\partial\bD^1/}$ such that for every $\infty$-category $\mC$ the $\infty$-category $S(\mC),$ called categorical suspension of $\mC$, has two objects 0,1, and morphism $\infty$-categories:
$$\Mor_{S(\mC)}(1,0)\simeq \emptyset, \ \Mor_{S(\mC)}(0,1)\simeq \mC, \ \Mor_{S(\mC)}(0,0)\simeq \Mor_{S(\mC)}(1,1) \simeq *.$$
    
\end{proposition}

\begin{definition}\label{susa}
For every $\infty$-category $\mC$ the antisuspension of $\mC$ is $\overline{S}(\mC):=S(\mC^\co)^\co.$

\end{definition}

\begin{corollary}\label{sw} Let $\mC $ be an $\infty$-category.
\begin{enumerate}[\normalfont(1)]\setlength{\itemsep}{-2pt}
\item There is a canonical equivalence $$S(\mC^\op) \simeq S(\mC)^\co,$$
which is the identity on the underlying space.
\item There is a canonical equivalence $$ S(\mC^\co) \simeq S(\mC)^\op,$$
which switches the two elements of the underlying space.
\item There is a canonical equivalence $$ S(\mC^{\co\op}) \simeq S(\mC)^{\co\op},$$
which switches the two elements of the underlying space.
\end{enumerate}

\end{corollary}

\begin{proof}
The first and second equivalence follow by adjointness from \cref{oppo}.
The third equivalence follows from the first and second one.
\end{proof}

Next we define strict and univalent $\infty$-categories:

\begin{definition}Let $ n \geq 1.$ By induction on $n$ we define $\n$-univalent $\infty$-categories.
An $\infty$-category is 1-univalent if it is local with respect to the unique functor $\{0 \simeq 1 \} \to *$, where $\{0\simeq 1\}$ denotes the category obtained from $\bD^1$ by inverting the unique map from $0$ to $1$.
An $\infty$-category is $n$-univalent if it is 1-univalent and all
morphism $\infty$-categories are $n-1$-univalent.
An $\infty$-category is univalent if it is $n$-univalent for every $n \geq 1.$
\end{definition}

\begin{definition}Let $ n \geq 0.$ By induction on $n$ we define $n$-strict $\infty$-categories.
An $\infty$-category is 0-strict if its underlying space is a set.
An $\infty$-category is $n$-strict if it is 0-strict and all morphism $\infty$-categories are $n-1$-strict.
An $\infty$-category is strict if it is $n$-strict for every $n \geq 0.$
\end{definition}

\begin{definition}
Let $ \infty\Cat^{\univ}, \infty\Cat^{\strict}\subset \infty\Cat$ 
be the respective full subcategories of univalent and strict $\infty$-categories.
\end{definition}






\begin{notation}
For every $0 \leq \n \leq \m$ the left adjoint embeddings $\n\Cat \leftrightarrows \m\Cat$ preserve small limits and thus induce a left adjoint embedding $\n\Cat \leftrightarrows \infty\Cat: \iota_\n$ that preserves small limits
and so admits a left adjoint $\tau_\n: \infty\Cat \to \n\Cat$ by presentability.


\end{notation}

\begin{remark}
The $\iota_n$ collectively filter every $\infty$-category $X$ as a colimit $X\simeq\colim_{n}\iota_n(X)$.
\end{remark}

The next is \cite[Lemma 2.4.16.]{GepnerHeine2026}:

\begin{lemma}\label{carclo}
The presentable category $\infty\Cat$ is cartesian closed.
\end{lemma}

\begin{notation}

For any $\infty$-category $\mC$ let 
$\Fun(\mC,-): \infty\Cat \to \infty\Cat$ be the right adjoint of the functor $(-) \times \mC:\infty\Cat \to \infty\Cat.$
    
\end{notation}

\begin{remark}
Since $\infty\Cat$ is cartesian closed, it refines to an $\infty$-category $\infty\scat$ such that for every two objects $X$ and $Y$:
\[
\Mor_{\infty\scat}(X,Y)=\Fun(X,Y).
\]
\end{remark}

	
	

	

	
		

	



The next is \cite[Proposition 2.4.21.]{GepnerHeine2026}:

\begin{proposition}\label{completion}
The embedding $$ \infty\Cat^{\mathrm{univ}}\subset \infty\Cat$$ admits a left adjoint, the univalent completion.
\end{proposition}

\begin{definition}Let $\n \geq 0$.
The $\n$-disk is the $n$-fold suspension $\bD^\n:= S^{\n}(*)$ of the terminal $\infty$-category $\ast$.
\end{definition}


\begin{definition}Let $\n \geq 0$.
The boundary of the $\n$-disk is the $n$-fold suspension $\partial\bD^\n:= S^{\n}(\emptyset)$ of the initial $\infty$-category $\emptyset$.
	
\end{definition}

\begin{remark}Let $\n \geq 0.$
The functor $\emptyset \subset *$ induces an inclusion $\partial\bD^\n \subset \bD^\n$.
	
\end{remark}

\begin{example}
Then $\partial\bD^0=\ast, \partial\bD^1=S(\emptyset)=*\coprod*$ is the set with two elements.
Viewed as a subject of $\bD^1$, these elements acquire a natural ordering.
\end{example}

\begin{definition}
The bipointed wedge of an ordered pair $(\mC,\mD)$ of bipointed $\infty$-categories  $\infty\Cat_{/\partial\bD^1}$ and $\mC'\in\infty\Cat_{/\partial\bD^1}$ is the bipointed $\infty$-category obtained by taking the pushout along the common basepoint.
\end{definition}

\begin{definition}\label{Theta}
	
Let $\Theta' \subset \infty\Cat_{\partial\bD^1/}$ be the full subcategory generated by $\bD^0$ under suspensions and bipointed wedges and $\Theta \subset \infty\Cat$ the essential image of $\Theta'$ under the forgetful functor.	
	
\end{definition}

\begin{remark}

An $\infty$-category belongs to $\Theta$ if and only if it is of the form $$ \bD^{i_0} \coprod_{\bD^{j_1}} \bD^{i_1} \coprod_{\bD^{j_2}} \cdots \coprod_{\bD^{j_n}} \bD^{i_n} $$
for a sequence of natural numbers $n, i_0,\ldots, i_n, j_1,\ldots, j_n$ and monomorphisms 
$ \bD^{j_\ell} \rightarrowtail \bD^{i_\ell},\bD^{j_\ell} \rightarrowtail \bD^{i_{\ell-1}}$, $ 1 \leq \ell \leq n$.

\end{remark}

The next is \cite[Theorem 2.5.8.]{GepnerHeine2026}:

\begin{theorem}\label{theta}

The restricted Yoneda embeddings \begin{equation}\label{jopol}
\infty\Cat \to \Fun(\Theta^\op,\mS)\qquad\mathrm{and}\qquad \infty\Cat^\strict \to \Fun(\Theta^\op,\Set)\end{equation} are fully faithful and admit left adjoints that preserves finite products.
For every $ 0 \leq n \leq \infty$ the full subcategories $n\Cat \subset \infty\Cat, n\Cat^\strict \subset \infty\Cat^\strict$ are generated under small colimits by the disks of dimension smaller $n+1.$

\end{theorem}

By \cite[Definition 2.6.25.]{GepnerHeine2026} for every $n \geq 0$ there is an oriented $n$-cube $\cube^n$
whose 1-truncation is $(\bD^1)^{\times n}.$

\begin{notation}\label{sqGray}

Let $\cube \subset \infty\Cat$ be the full subcategory of oriented cubes.
\end{notation}

The following is \cite[Theorem 2.8.12.]{GepnerHeine2026}:

\begin{theorem}

The full subcategory $\cube \subset \infty\Cat$ is dense.
    
\end{theorem}

By \cite[Definition 2.6.32.]{GepnerHeine2026} for every $n \geq 0$ there is an oriented $n$-simplex $\bDelta^n$
whose 1-truncation is the totally ordered set $[n]= \{0 < ...< n \}.$

\begin{notation}

Let $\bDelta \subset\infty\Cat$ be the full subcategory spanned by the oriented simplices.
  
Let $\bDelta^+ \subset\infty\Cat$ be the full subcategory spanned by the oriented simplices and the initial $\infty$-category.
  
\end{notation}

The following is \cite[Theorem 2.7.8.]{GepnerHeine2026}:

\begin{theorem}

The full subcategory $\bDelta\subset\infty\Cat$ is dense.
    
\end{theorem}

\newpage

\vspace{.25cm}
\section{Oriented categories}
\vspace{.25cm}

\subsection{The Gray tensor product}



By \cite[Definition 2.6.23.]{GepnerHeine2026} there is a canonical monoidal structure on $\cube$ whose tensor unit is
$\bD^0$ and such that $\cube^n \boxtimes \cube^m= \cube^{n+m}$ for every $n,m \geq 0.$


We have the following result of Campion \cite{campion2022cubesdenseinftyinftycategories}, which is also proven in \cite[Corollary 3.4.2.]{GepnerHeine2026}:

\begin{corollary}\label{locmon2}
Let $\cube \subset \infty\Cat$ denote the full subcategory of oriented cubes, equipped with the Gray monoidal structure.
\begin{enumerate}[\normalfont(1)]\setlength{\itemsep}{-2pt}
\item The convolution monoidal structure descends along 
the localization $$\mP(\cube) \rightleftarrows \infty\Cat.$$

\item The convolution monoidal structure descends along 
the localization $$ \mP_\Set(\cube) \rightleftarrows \infty\Cat^{\mathrm{strict}}.$$

\end{enumerate}

\end{corollary}

\begin{remark}\label{stricto}
	
The monoidal localization $\pi_0 : \mS \leftrightarrows \Set$ gives rise to a monoidal localization
$$ \Fun(\cube^\op, \mS) \leftrightarrows \Fun(\cube^\op,\Set) $$ that descends to a monoidal localization
$\infty\Cat \leftrightarrows \infty\Cat^\strict$ for the Gray monoidal structures.
In particular, the embedding $\infty\Cat^\strict \subset \infty\Cat$ is lax monoidal for the Gray monoidal structures.
	
\end{remark}

\begin{notation}
	
Since the Gray tensor product defines a presentably monoidal structure on $\infty\Cat$, it is closed: for every $\infty$-category $\mC$ the functor $$ \mC \boxtimes (-): \infty\Cat \to \infty\Cat$$ admits a right adjoint $\Fun^\lax(\mC,-)$, and the functor $$ (-) \boxtimes \mC : \infty\Cat \to \infty\Cat$$ admits a right adjoint $\Fun^\oplax(\mC,-)$.
\end{notation}

\begin{definition}Let $\F,\G:\mC \to \mD $ be functors of $\infty$-categories.
\begin{enumerate}[\normalfont(1)]\setlength{\itemsep}{-2pt}
\item A lax natural transformation $\F \to \G$ is a morphism in $\Fun^\lax(\mC,\mD)$.

\item An oplax natural transformation $\F \to \G$ is a morphism in $\Fun^\oplax(\mC,\mD)$.
\end{enumerate}
\end{definition}

\begin{remark}
	
Let $\mC,\mD$ be $\infty$-categories. The canonical functors $\mC \to *, \mD \to *$
give rise to functors $\mC \boxtimes \mD \to \mC \boxtimes * \simeq \mC, \mC \boxtimes \mD \to * \boxtimes \mD \simeq \mD$ and so to a functor $\mC \boxtimes \mD \to \mC \times \mD.$
By adjointess the latter functor induces functors
$ \Fun(\mC,\mD)\to \Fun^\lax(\mC,\mD)$ and $\Fun(\mC,\mD)\to \Fun^\oplax(\mC,\mD).$
\end{remark}

\begin{remark}\label{lao}\label{grayspace}
If $\mC$ is an $\n$-category and $\mD$ an $\m$-category for $\n,\m \geq 0$,
then $\mC \boxtimes \mD$ is an $\n+\m$-category.
This holds since $\n\Cat$ is closed under small colimits in $\infty\Cat$ and $\n\Cat$ is generated under small colimits by the oriented $\ell$-cubes for $1 \leq \ell \leq \n$.
For $n=m=0$ one finds that the Gray-monoidal structure restricts to the 
full subcategory of spaces $\mS \subset \infty\Cat$.
The restricted Gray-monoidal structure on $\mS$ is the cartesian structure since $\mS$ is generated in $\infty\Cat$ under small colimits by the final category, the tensor unit for the Gray tensor product and the cartesian product.
Moreover the left adjoint $\tau_0: \infty\Cat \to \mS$ of the monoidal embedding $\mS \subset \infty\Cat$ is monoidal since
$\infty\Cat$ is generated by the oriented cubes under small colimits and
the image of any oriented cube under $\tau_0$ is contractible.

\end{remark}

The following is \cite[Proposition 3.4.8.]{GepnerHeine2026}: 
\begin{proposition}\label{dua}
	
There are canonical monoidal involutions
$$(-)^\op, (-)^\co: (\infty\Cat, \boxtimes) \simeq (\infty\Cat, \boxtimes)^\rev $$
refining the involutions of \cref{ruik} and restricting to 
monoidal involutions
$$(-)^\op, (-)^\op: (\infty\Cat^{\mathrm{strict}}, \boxtimes) \simeq (\infty\Cat^{\mathrm{strict}}, \boxtimes)^\rev.$$


\end{proposition}


\begin{corollary}\label{grayhoms}
Let $\mC,\mD \in \infty\Cat$.
There are canonical equivalences $$ \Fun^\oplax(\mC,\mD)^\op \simeq \Fun^\lax(\mC^\op,\mD^\op), $$$$\Fun^\oplax(\mC,\mD)^\co \simeq \Fun^\lax(\mC^\co,\mD^\co).$$	
\end{corollary}






By \cite[Theorem 3.7.2.]{GepnerHeine2026} there is the following suspension formula:

\begin{theorem}\label{thas2}
For every $\infty$-category $X$ there is a pushout square in $\infty\Cat:$
\[
\xymatrix{
X \coprod X \ar[r]\ar[d] & X \boxtimes \bD^1 \ar[d]\\
\bD^0 \coprod \bD^0 \ar[r] & S(X).
}
\]
\end{theorem}


Next we define join and slice.

\begin{notation} For every small category $\mC$ that admits an initial object let $ \mP_{\mathrm{red}}(\mC) \subset \mP(\mC) $  
be the full subcategory of presheaves on $\mC$
that are reduced, i.e. send the initial object to the final space.
\end{notation}

By \cite[Definition 2.6.30.]{GepnerHeine2026} the category $\bDelta^+$ carries a monoidal structure, the join,
whose tensor unit is the empty $\infty$-category and such that $$\bDelta^n \star \bDelta^m = \bDelta^{n+m+1}$$ for every $n,m \geq -1.$
The following is \cite[Corollary 3.5.8.]{GepnerHeine2026}:

\begin{theorem}\label{locmon4}

The join monoidal structure descends along the localization $\mP_{\mathrm{red}}(\bDelta^+) \rightleftarrows \infty\Cat. $

\end{theorem} 


We also define the antijoin:

\begin{definition}
Let $X, Y \in \infty\Cat.$ The antijoin of $X, Y$ is
$$X \bar{\star} Y := (X^{\co} \star Y^{\co})^{\co} .$$

\end{definition}

\begin{definition}Let $X$ be a small $\infty$-category.
\begin{enumerate}[\normalfont(1)]\setlength{\itemsep}{-2pt}
\item The oplax slice, or oplax over $\infty$-category, functor is the right adjoint $$ \infty\Cat_{X/ } \to \infty\Cat, \qquad (F:X \to Y) \mapsto Y_{//^\oplax F}$$ of the functor $ (-) \star X: \infty\Cat \to \infty\Cat_{X/ }.$

\item The lax coslice, or lax under $\infty$-category, functor is the right adjoint $$ \infty\Cat_{X/ } \to \infty\Cat,\qquad (F:X \to Y) \mapsto Y_{F//^\lax }$$ of the functor $ X \star (-): \infty\Cat \to \infty\Cat_{\X / }.$

\item The lax slice, or lax over $\infty$-category, functor is the right adjoint $$ \infty\Cat_{\X / } \to \infty\Cat,\qquad (F:X \to Y) \mapsto Y_{//^\lax F} := (Y^\co_{//^\oplax F^\co})^\co $$
of the functor $ (-) \bar{\star} X: \infty\Cat \to \infty\Cat_{X/ }.$

\item The oplax coslice, or oplax under $\infty$-category, functor is the right adjoint $$ \infty\Cat_{X/ } \to \infty\Cat,\qquad (F:X \to Y) \mapsto Y_{F//^\oplax}:=(Y^\co_{F^\co//^\lax })^\co $$
of the functor $ X \bar{\star} (-): \infty\Cat \to \infty\Cat_{\X / }.$

\end{enumerate}

\end{definition}

\begin{notation}Let $F: X \to Y $ be a functor.
If unspecified, $Y_{//F}$ will refer to the oplax slice, and $Y_{F//}$ will refer to the oplax coslice.
\end{notation}

\begin{lemma}Let $F: X \to Y $ be a functor.
There is a canonical equivalence of $\infty$-categories
$$ (Y_{//^\oplax F})^\op \simeq (Y^\op)_{F^\op//^\lax }.$$
\end{lemma}

\begin{proof}

There is a canonical equivalence 
$$ (X \star (-))^\op \simeq (-)^\op \star {X^\op}.$$

Hence there is the following canonical equivalence natural in $Z \in \infty\Cat:$
$$ \Map_{\infty\Cat}(Z, (Y_{//^\oplax F})^\op) \simeq  \Map_{\infty\Cat}(Z^\op, Y_{//^\oplax F}) \simeq \Map_{\infty\Cat_{X/}}(Z^\op \ast X, F)$$
$$ \simeq \Map_{\infty\Cat_{X^\op/}}(X^\op \ast Z, F^\op) \simeq  \Map_{\infty\Cat}(Z, (Y^\op)_{F^\op //^\lax }).$$

\end{proof}


\subsection{Enrichment in the Gray tensor product}
In the following we apply the theory of bienriched $\infty$-categories, as outlined in \cref{enr}, to the Gray monoidal structure on $\infty\Cat$.
See \cite{heine2024bienriched} for a detailed discussion of bienriched category theory.


\begin{definition}
\begin{enumerate}[\normalfont(1)]\setlength{\itemsep}{-2pt}
\item An oriented category is a category right enriched in $(\infty\Cat,\boxtimes).$

\item An oriented functor is a functor right enriched in $(\infty\Cat,\boxtimes).$

\item An antioriented category is a category left enriched in $(\infty\Cat,\boxtimes).$

\item An antioriented functor is a functor left enriched in $(\infty\Cat,\boxtimes).$

\item A bioriented category is a category bienriched in  $((\infty\Cat,\boxtimes), (\infty\Cat,\boxtimes))$.

\item A bioriented functor is a functor bienriched in  $((\infty\Cat,\boxtimes), (\infty\Cat,\boxtimes))$.

\end{enumerate}

\end{definition}




\begin{notation}
\begin{enumerate}[\normalfont(1)]\setlength{\itemsep}{-2pt}
\item Let $\mC$ be an oriented category and $\X,\Y \in \mC.$	
By the structure of a right $(\infty\Cat,\boxtimes)$-enriched category there is a right morphism $\infty$-category $\R\Mor_\mC(\X,\Y)$.

\item Let $\mC$ be an antioriented category and $\X,\Y \in \mC.$	
By the structure of a left $(\infty\Cat,\boxtimes)$-enriched category there
is a left morphism $\infty$-category $\L\Mor_\mC(\X,\Y)$.

\item Let $\mC$ be a bioriented category and $\X,\Y \in \mC.$	
By the structure of a $(\infty\Cat,\boxtimes)$-bienriched category there
is a morphism $\infty$-category $\Mor_\mC(\X,\Y) \in \infty\Cat \otimes \infty\Cat.$

\end{enumerate}
	
\end{notation}

\begin{remark}
If $\mC$ is an oriented category, then $\mC$ has an underlying category $(\iota_0)_!\mC$, where $\iota_0:\infty\Cat\to\mS$ denotes the core functor; equivalently, the underlying category of $\mC$ is given by restricting the morphism $\infty$-categories to their spaces of objects.
However, $\mC$ does not have an underlying $\infty$-category, because $\mC$ does note satisfy the interchange law; instead, $\mC$ satisfies an oriented version of the interchange law.
Given $1$-cells $f,f':X\to Y$ and $g,g':Y\to Z$ and $2$-cells $\alpha:f\Rightarrow f'$ and $\beta:g\Rightarrow g'$, viewed as $1$-cells $\bD^1\to\RMor_{\mC}(X,Y)$ and $\bD^1\to\RMor_{\mC}(Y,Z)$, we obtain map $\bD^1\boxtimes\bD^1\to\RMor_{\mC}(X,Y)\boxtimes\RMor_{\mC}(Y,Z)\to\RMor_{\mC}(X,Z)$, which we can picture as a lax commuting square
\[
\xymatrix{
g\circ f\ar[r]^{\id_g\circ\alpha}\ar[d]_{\beta\circ\id_f} & g\circ f'\ar[d]^{\beta\circ \id_{f'}}\\
g'\circ f\ar[r]_{\id_{g'}\circ\alpha}\ar@{=>}[ru] & g'\circ f'.
}
\]
The failure of this square (and its higher dimensional analogues, coming from mapping higher dimensional cells into the morphism $\infty$-categories, as in \cref{orientedinterchange}) to commute is precisely the obstruction to $\mC$ being an $\infty$-category (viewed as an oriented category).
See \cref{interchange} for a precise statement.
\end{remark}

\begin{remark}
If $\mC$ is an antioriented category, then given $1$-cells $f,f':X\to Y$ and $g,g':Y\to Z$ and $2$-cells $\alpha:f\Rightarrow f'$ and $\beta:g\Rightarrow g'$, viewed as $1$-cells $\bD^1\to\LMor_{\mC}(X,Y)$ and $\bD^1\to\LMor_{\mC}(Y,Z)$, we obtain a lax commuting square $\bD^1\,\bar{\boxtimes}\,\bD^1\to\LMor_{\mC}(X,Y)\,\bar{\boxtimes}\,\LMor_{\mC}(Y,Z)\to\LMor_{\mC}(X,Z)$, which we can picture as a lax commuting square
\[
\xymatrix{
g\circ f\ar[r]^{\id_g\circ\alpha}\ar[d]_{\beta\circ\id_f} & g\circ f'\ar[d]^{\beta\circ \id_{f'}}\ar@{=>}[ld]\\
g'\circ f\ar[r]_{\id_{g'}\circ\alpha} & g'\circ f'.
}
\]
\end{remark}

The previous remarks motivate the following definition:

\begin{definition}\label{orientedinterchange}
\begin{enumerate}[\normalfont(1)]\setlength{\itemsep}{-2pt}
\item
Let $\mC$ be an oriented category.
For any triple of objects $X,Y,Z$ of $\mC$, an $n$-cell 
$\alpha:\bD^n\to\RMor_{\mC}(Y,Z)$ and $m$-cell
$\beta:\bD^m\to\RMor_{\mC}(X,Y)$, the {\em oriented interchange} associated to $\alpha$ and $\beta$ is the map
\[
\bD^n\boxtimes\bD^m\xrightarrow{\alpha \boxtimes \beta}\RMor_\mC(Y,Z)\boxtimes\RMor_\mC(X,Y)\to\RMor_\mC(X,Z)
\]
obtained by applying the composition to the oriented Gray tensor product of $\alpha$ and $\beta$.
\item Let $\mC$ be an antioriented category.
For any triple of objects $X,Y,Z$ of $\mC$, an $n$-cell $\alpha:\bD^n\to\LMor_{\mC}(Y,Z)$ and $m$-cell $\beta:\bD^m\to\LMor_{\mC}(X,Y)$, the {\em antioriented interchange} associated to $\alpha$ and $\beta$ is the map
\[
\bD^n\,\boxtimes\,\bD^m\xrightarrow{\alpha \boxtimes \beta}\LMor_\mC(Y,Z)\,\boxtimes\,\LMor_\mC(X,Y)\to\LMor_\mC(X,Z)
\]
obtained by applying the composition to the antioriented Gray tensor product of $\alpha$ and $\beta$.
\end{enumerate}
    
\end{definition}

\begin{example}
The Gray monoidal structure on $\infty\Cat$ is closed and so endows $\infty\Cat$ as bienriched in $(\infty\Cat,\boxtimes).$ This way we see
$\infty\Cat$ as a large bioriented category, which we denote by $ \infty\fcat.$
Every full subcategory of $\infty\Cat$ inherits the structure of a bioriented category.

\end{example}	

\begin{definition}
We refer to morphims in $\boxtimes\Cat$ as {\em antioriented functors}, to morphims in $\Cat\boxtimes$ as {\em oriented functors}, and to morphims in $\boxtimes\Cat\boxtimes$ as {\em bioriented functors}.
\end{definition}

\begin{notation}

Let $$\Cat\boxtimes, \boxtimes\Cat, \boxtimes\Cat\boxtimes$$
be the respective categories of oriented categories, antioriented categories and bioriented categories.
\end{notation}







\begin{remark}
An oriented, antioriented, or bioriented category is {\em presentable} if the respective left, right, or bienriched category is presentable in the sense of enriched $\infty$-category theory, which is likewise a presentable category endowed with a closed left action, closed right action, or closed biaction of the enriching monoidal category or categories, respectively.
\end{remark}
		
		


\begin{notation}\emph{}
\begin{enumerate}[\normalfont(1)]\setlength{\itemsep}{-2pt}
\item Let $\mC,\mD \in \boxtimes\Cat$. Let $${\boxtimes\Fun}(\mC,\mD)$$ be the category of antioriented functors $\mC \to \mD.$	
		
\item Let $\mC,\mD \in \Cat\boxtimes$. Let $${\Fun\boxtimes}(\mC,\mD)$$ be the category of oriented functors $\mC \to \mD.$	
		
\item Let $\mC,\mD \in \boxtimes\Cat\boxtimes $. Let $${\boxtimes\Fun\boxtimes}(\mC,\mD)$$ be the category of bioriented functors $\mC \to \mD.$		

\end{enumerate}
\end{notation}

\begin{remark}
A final object of an oriented, antioriented, or bioriented category is an object $* $ such that the left morphism object, right morphism object, or morphism object of any object to $*$ is the final object.
Dually, we define an initial object. 
A zero object in an oriented, antioriented, or bioriented category is an initial and final object.
\end{remark}

\begin{definition}
Let $\mK\subset\Cat$ be a full subcategory.
We say that an oriented, antioriented, or bioriented category admits $\mK$-indexed colimits if it admits $\mK$-indexed conical colimits.
    
\end{definition}

\begin{remark}
We refer to adjunctions of oriented, antioriented, and bioriented categories as oriented, antioriented, or bioriented adjunctions.
\end{remark}

\begin{remark}\label{adj}\label{adj2}
\begin{enumerate}[\normalfont(1)]\setlength{\itemsep}{-2pt}
\item An antioriented (oriented) functor $\mC \to \mD$ admits a right adjoint if and only if it preserves left (right) tensors and the underlying functor admits a right adjoint.

\item A bioriented functor $\mC \to \mD$ admits a right adjoint if and only if it preserves left and right tensors and the underlying functor admits a right adjoint.

\item An antioriented (oriented) functor $\mC \to \mD$ admits a left adjoint if and only if it preserves left (right) cotensors and the underlying functor admits a left adjoint.
	
\item A bioriented functor $\mC \to \mD$ admits a left adjoint if and only if it preserves left and right cotensors and the underlying functor admits a left adjoint.

\end{enumerate}

\end{remark}

\begin{remark}\label{Grayspaces} By \cref{grayspace} there is a monoidal localization
$\tau_0: \infty\Cat \rightleftarrows \mS$ that induces the following bioriented localization by \cref{adj}:
$$  \tau_0: \infty\Cat \rightleftarrows \tau_0^*(\mS). $$ 
	
\end{remark}

\subsection{Oriented opposites and conjugates}

Next we define the appropriate notions of opposite oriented category.

\begin{definition}
    
Let
\begin{align*}
&(-)^\circ: {\Cat\boxtimes} \simeq {\boxtimes\Cat}\\
&(-)^\circ: {\boxtimes\Cat} \simeq {\Cat\boxtimes}\\
&(-)^\circ:  {\boxtimes\Cat\boxtimes} \simeq {\boxtimes\Cat\boxtimes}
\end{align*}
be the opposite enriched category involutions.	
	
\end{definition}

\begin{definition}
We define the following involutions, which reverse the even dimensional cells.
Note that we are also forced to reverse the orientation.
The equivalences of \cref{dua} give rise to the equivalences
\begin{align*}
&(-)^\co:= (-)^\op_!: {\boxtimes\Cat} \simeq {\Cat\boxtimes}\\
&(-)^\co:= (-)^\op_!: {\Cat\boxtimes} \simeq {\boxtimes\Cat}\\
&(-)^\co:= ((-)^\op, (-)^\op)_!: {\boxtimes\Cat\boxtimes} \simeq {\boxtimes\Cat\boxtimes}.
\end{align*}
\end{definition}

\begin{definition}
We define the following involutions, which reverse the odd dimensional cells.
Note that we are also forced to keep the orientation.
\begin{align*}
&(-)^\op:= (-)^\circ\circ (-)^\co_!: {\boxtimes\Cat} \simeq {\boxtimes\Cat}\\
&(-)^\op:= (-)^\circ\circ (-)^\co_!: {\Cat\boxtimes} \simeq {\Cat\boxtimes}\\
&(-)^\op :=(-)^\circ \circ ((-)^\co, (-)^\co)_!: {\boxtimes\Cat\boxtimes} \simeq {\boxtimes\Cat\boxtimes}
\end{align*}
\end{definition}

\begin{definition}
Combining the latter two types of involutions gives rise to the following sort of involution, which reverses all cells:
\begin{align*}
&{(-)^{\co\op}}:= {(-)^{\co}} \circ {(-)^{\op}} \simeq {(-)^{\op}} \circ{(-)^{\co}}:{\boxtimes\Cat} \simeq {\Cat\boxtimes}\\
&{(-)^{\co\op}}:= {(-)^{\co}} \circ {(-)^{\op}} \simeq {(-)^{\op}} \circ {(-)^{\co}}:{\Cat\boxtimes} \simeq {\boxtimes\Cat}\\
&{(-)^{\co\op}}:= {(-)^{\co}} \circ {(-)^{\op}} \simeq {(-)^{\op}} \circ {(-)^{\co}}: {\boxtimes\Cat\boxtimes} \simeq {\boxtimes\Cat\boxtimes},
\end{align*}
where all equivalence are involutions.
\end{definition}

Combining all three types of involutions gives rise to the following sort of involution, which reverses all positively dimensional cells.

\begin{definition}
Moreover we set
\begin{align*}
&{(-)^{\cop}}:=  {(-)^{\co\op}} \circ {(-)^{\circ}} \simeq {(-)^{\circ}} \circ {(-)^{\co\op}} \simeq (-)^{\co\op}_!: {\boxtimes\Cat} \simeq {\boxtimes\Cat}\\
&{(-)^{\cop}}:=  {(-)^{\co\op}} \circ {(-)^{\circ}} \simeq {(-)^{\circ}} \circ {(-)^{\co\op}} \simeq (-)^{\co\op}_!: {\Cat\boxtimes} \simeq {\Cat\boxtimes}\\
&{^{\cop}(-)}:= ((-)^{\co\op}, \id)_!: {\boxtimes\Cat\boxtimes} \simeq {\boxtimes\Cat\boxtimes},\\
&(-)^\cop:= (\id, (-)^{\co\op})_!: {\boxtimes\Cat\boxtimes} \simeq {\boxtimes\Cat\boxtimes},\\
&{^{\cop}(-)}^\cop:=(-)^\cop \circ  {^{\cop}(-)} \simeq {^{\cop}(-)}\circ (-)^\cop \simeq ((-)^{\co\op}, (-)^{\co\op})_! \simeq {(-)^{\co\op}} \circ {(-)^{\circ}} \\
&\simeq {(-)^{\circ}} \circ {(-)^{\co\op}}: {\boxtimes\Cat\boxtimes} \simeq {\boxtimes\Cat\boxtimes}.
\end{align*}
\end{definition}

The equivalences of \cref{dua} gives rise to the following equivalences of bioriented categories:
\begin{corollary}\label{duae}There are canonical equivalences of
bioriented categories:
\begin{align*}  
&(-)^\op: \infty\Cat^\co \simeq \infty\Cat\\
&(-)^\co: \infty\Cat^\op \simeq \infty\Cat^\circ\\
&(-)^{\co\op} : {^\cop\infty\Cat^{\cop}} \simeq \infty\Cat
\end{align*}
\end{corollary}	




    


    





    


\subsection{Reduced oriented categories}

\begin{notation}\emph{}
\begin{enumerate}[\normalfont(1)]\setlength{\itemsep}{-2pt}
\item Let $\Cat\wedge $
be the category of weakly reduced oriented categories.	
\item Let $\wedge\Cat$
be the category of weakly reduced antioriented categories.
\item Let $\wedge\Cat\wedge $
be the category of weakly reduced bioriented categories.
\end{enumerate}
\end{notation}








\begin{example}

The Gray smash monoidal structure on $\infty\Cat_*$ is closed and endows $\infty\Cat_*$ as bienriched in $((\infty\Cat_*,\wedge), (\infty\Cat_*,\wedge)).$ This way we see
$\infty\Cat_*$ as a large reduced Gray-category, denoted by $ \infty\fcat_*.$

\end{example}	

\begin{notation}\emph{}
\begin{enumerate}[\normalfont(1)]\setlength{\itemsep}{-2pt}
\item Let $\mC,\mD \in \wedge\Cat$. Let $${\wedge\Fun}(\mC,\mD)$$ be the category of weakly reduced antioriented functors $\mC \to \mD.$	
\item Let $\mC,\mD \in \Cat\wedge$. Let $${\Fun\wedge}(\mC,\mD)$$ be the category of weakly reduced oriented functors $\mC \to \mD.$	
\item Let $\mC,\mD \in \wedge\Cat\wedge $. Let $${\wedge\Fun\wedge}(\mC,\mD)$$ be the category of weakly reduced bioriented functors $\mC \to \mD.$			
\end{enumerate}
\end{notation}





\begin{remark}
    
A final object of a weakly reduced antioriented, oriented, or bioriented category is an object $* $ such that the left morphism object, right morphism object, or morphism object of any object to $*$ is the final object.
Dually, we define an initial object.
A zero object of a weakly reduced antioriented, oriented, or bioriented category is an initial and final object.
\end{remark}

\begin{definition}
An antioriented, oriented, or bioriented category is reduced if it admits a zero object.
\end{definition}

The next proposition is \cref{red}:
\begin{proposition}\label{reduu}
The forgetful functors $${\wedge\Cat} \to \boxtimes\Cat, \ {\Cat\wedge} \to \Cat\boxtimes, \ {\wedge\Cat \wedge} \to {\boxtimes\Cat \boxtimes}$$
restrict to equivalences between the full subcategory of weakly reduced oriented, antioriented, bioriented categories 
that admit an initial or final object and the subcategory of reduced oriented, antioriented, bioriented categories and reduced oriented, antioriented, bioriented functors.
\end{proposition}

In view of \cref{reduu} we identify reduced oriented, antioriented, bioriented categories with weakly reduced oriented, antioriented, bioriented categories that admit an initial or final object, and use the term reduced oriented, antioriented, bioriented categories for such.
Similarly, we identify reduced oriented, antioriented, bioriented functors between reduced oriented, antioriented, bioriented categories with weakly reduced oriented, antioriented, bioriented functors, respectively.

\begin{remark}

We refer to adjunctions of (weakly) reduced oriented, antioriented, or bioriented categories as (weakly) reduced oriented, antioriented, or bioriented adjunctions.
\end{remark}

\subsection{Strict oriented categories}

\begin{definition}
We have the following notions of strict oriented categories.
\begin{enumerate}[\normalfont(1)]\setlength{\itemsep}{-2pt}
\item A strict oriented Gray-category is a category right enriched in $(\infty\Cat^\strict,\boxtimes).$

\item A strict antioriented Gray-category is a category left enriched in $(\infty\Cat^\strict,\boxtimes).$
		
\item A strict bioriented category is a category bienriched in $(\infty\Cat^\strict,\boxtimes), (\infty\Cat^\strict,\boxtimes).$

\item A strict oriented functor is a functor right enriched in $(\infty\Cat^\strict,\boxtimes).$

\item A strict antioriented functor is a functor left enriched in $(\infty\Cat^\strict,\boxtimes).$
			
\item A strict bioriented functor is a functor bienriched in $(\infty\Cat^\strict,\boxtimes), (\infty\Cat^\strict,\boxtimes).$	
\end{enumerate}
	
\end{definition}

	

\begin{notation}
Let
$$
{\Cat\boxtimes}^\strict, \ 
{\boxtimes\Cat}^\strict, \
{\boxtimes\Cat\boxtimes}^\strict
$$
be the respective categories of strict oriented categories and oriented functors, strict antioriented categories and antioriented functors, strict bioriented categories and bioriented functors.
	
\end{notation}

\begin{example}
The Gray monoidal structure on $\infty\Cat^\strict$ is closed and so exhibits $\infty\Cat^\strict$ as 
bienriched in $(\infty\Cat^\strict,\boxtimes), (\infty\Cat^\strict,\boxtimes),$ that is a strict bioriented category.
	
\end{example}

\begin{remark}
By \cref{stricto} there is a monoidal localization $\infty\Cat \rightleftarrows \infty\Cat^\strict$
whose right adjoint is the embedding.
The latter monoidal adjunction gives rise to localizations
$$ {\boxtimes\Cat} \rightleftarrows {\boxtimes\Cat}^\strict,\qquad {\Cat\boxtimes} \rightleftarrows {\Cat\boxtimes}^\strict,\qquad {\boxtimes\Cat\boxtimes} \rightleftarrows {\boxtimes\Cat\boxtimes}^\strict. $$
	
\end{remark}

For the proof of \cref{alfa} we use the following lemma:

\begin{lemma}\label{strict}
Let $\mV,\mW $ be monoidal (1,1)-categories and $\mC,\mD$ categories bienriched in $(\mV,\mW)$ (necessarily (1,1)-categories).
Let $\F: \mC \to \mD$ be a left $\mV$-enriched functor and $\F': \mC \to \mD$ a right $\mW$-enriched functor such that $\F=\F'$ as functors $\mC \to \mD.$
Then $\F=\F':\mC \to \mD$ uniquely refines to a $(\mV,\mW)$-enriched functor whose underlying left $\mV$-enriched functor is $\F$ and whose underlying right $\mW$-enriched functor is $\F'$
if and only if for every $V \in \mV, X \in \mC, W \in \mW$ the following canonical square commutes:
\begin{equation}\label{bicomu}\begin{xy}
\xymatrix{
V \ot \F(X) \ot W   \ar[d]^{} \ar[r]^{}
& \F(V\ot X) \ot W \ar[d]^{}
\\ 
V \ot \F(X \ot W)
\ar[r]^{} & \F(V\ot X\ot W).
}
\end{xy}\end{equation}
	
\end{lemma}

\begin{proof}
	
By \cite[Proposition 3.31]{HEINE2023108941} the structure of a left $\mV$-enriched functor $\mC \to \mD$ on a functor $\F: \mC \to \mD$ is a map of $L\M$-operads $\phi: \mO^\ot \to \mU^\ot$ inducing $\F$ on the fiber over the color $\mathfrak{m} $ and the identity of $\mV$ on the fiber over $\mathfrak{a}$.
Since $\mV$ is a (1,1)-category and source and target of $\phi$ exhibit $\mC,\mD$ as left $\mV$-enriched, $\mO^\ot, \mU^\ot$ are (1,1)-categories.
Thus $\phi$ is a map of discrete $L\M$-operads that exhibits $\F$ as left $\mV$-enriched.
The latter is the datum of a natural isomorphism $V \ot \F(X) \to \F(V\ot X) $ for $X \in \mC, V \in \mV$ such that for any $V,V' \in \mV$ the isomorphism $V' \ot V \ot \F(X) \to \F(V' \ot V \ot X)$ factors as \begin{equation}\label{hit}
V' \ot V \ot \F(X) \to V' \ot \F(V\ot X) \to \F(V'\ot V \ot X).\end{equation}

Dually, the structure of a right $\mW$-enriched functor $\mC \to \mD$ on a functor $\F': \mC \to \mD$ is a map of $R\M$-operads $\phi: \mO^\ot \to \mU^\ot$ inducing $\F'$ on the fiber over $\mathfrak{m} $ and the identity of $\mW$ on the fiber over $\mathfrak{b}$.
Since $\mW$ is a (1,1)-category and source and target of $\phi$ exhibit $\mC,\mD$ as right $\mW$-enriched, $\mO^\ot, \mU^\ot$ are (1,1)-categories.
Hence $\phi$ is a map of discrete $R\M$-operads that exhibits $\F'$ as right $\mW$-enriched.
The latter is the datum of a natural isomorphism $\F(X) \ot W \to \F(X\ot W)$ for $X \in \mC, W \in \mW$ such that for any $W,W' \in \mW$ the isomorphism $\F(X)\ot W \ot W' \to \F(X \ot W \ot W') $ factors as 
\begin{equation}\label{hil}
\F(X)\ot W \ot W' \to\F(X \ot W) \ot W' \to \F(X \ot W \ot W'). \end{equation} 

By \cite[Proposition 3.42]{HEINE2023108941} the structure of a $(\mV,\mW)$-enriched functor $\mC \to \mD$ on a functor $\F=\F': \mC \to \mD$ is a map of $B\M$-operads $\phi: \mO^\ot \to \mU^\ot$ inducing $\F=\F' $ on the fiber over $\mathfrak{m} $, the identity of $\mV$ on the fiber over $\mathfrak{a}$ and the identity of $\mW$ on the fiber over $\mathfrak{b}.$
Since $\mV,\mW$ are (1,1)-categories and source and target of $\phi$ exhibit $\mC,\mD$ as bienriched in $(\mV,\mW)$, also $\mO^\ot, \mU^\ot$ are (1,1)-categories.
Hence $\phi$ is a map of discrete $B\M$-operads that exhibits $\F=\F'$ as $(\mV,\mW)$-enriched.
The latter is the datum of natural isomorphisms $V \ot \F(X) \to \F(V\ot X), \F(X) \ot W \to \F(X\ot W)$ for $X \in \mC, V \in \mV, W \in \mW$ such that 
$V' \ot V \ot \F(X) \to \F(V'\ot V \ot X)$ factors as \ref{hit} and $\F(X)\ot W \ot W' \to \F(X \ot W \ot W') $ factors as \ref{hil} for any $V,V' \in \mV, W,W' \in \mW$ and such that square \ref{bicomu} commutes.
\end{proof}

\begin{proposition}\label{alfa}
\begin{enumerate}[\normalfont(1)]\setlength{\itemsep}{-2pt}
\item The suspension functor $S: \infty\Cat^\strict \to \infty\Cat^\strict_{\partial\bD^1/}$ refines to a strict antioriented functor. 
For every strict $\infty$-categories $\mA, \mB$ the canonical functor
$$ \alpha: \mA \boxtimes S(\mB) \to S(\mA\boxtimes \mB)$$ is the canonical functor
$$ \mA \boxtimes (\mB \boxtimes \bD^1 \coprod_{\mB \boxtimes \partial \bD^1} \partial \bD^1)
\cong \mA \boxtimes \mB \boxtimes \bD^1 \coprod_{\mA \boxtimes \mB \boxtimes \partial \bD^1} \mA \boxtimes \partial \bD^1 \to \mA \boxtimes \mB \boxtimes \bD^1 \coprod_{\mA \boxtimes \mB \boxtimes \partial \bD^1} \partial \bD^1$$ induced by the functor $\mA \to *.$
\item The antisuspension functor $\bar{S}: \infty\Cat^\strict \to \infty\Cat^\strict_{\partial\bD^1/}$ refines to a strict oriented functor.
For every strict $\infty$-categories $\mB, \mC$ the canonical functor
$$ \beta: \bar{S}(\mB)\boxtimes \mC \to \bar{S}(\mB\boxtimes \mC)$$ is the canonical functor
$$ (\bD^1\boxtimes \mB \coprod_{\partial \bD^1\boxtimes \mB} \partial \bD^1) \boxtimes \mC
\cong \bD^1 \boxtimes \mB \boxtimes \mC \coprod_{\partial \bD^1\boxtimes \mB \boxtimes \mC} \partial \bD^1 \boxtimes \mC \to \bD^1 \boxtimes \mB \boxtimes \mC \coprod_{\partial \bD^1 \boxtimes \mB \boxtimes \mC} \partial \bD^1$$ induced by the functor $\mC \to *.$
\item There is a canonical isomorphism of functors $S \cong \bar{S} \circ (-)^{\co\op}: \infty\Cat^\strict \to \infty\Cat^\strict_{\partial\bD^1/}$ and the resulting strict antioriented and oriented functor $$S \cong \bar{S} \circ (-)^{\co\op}: (\infty\Cat^\strict)^\cop \to \infty\Cat^\strict_{\partial\bD^1/}$$ refines to a strict bioriented functor. 

\item There is a canonical isomorphism of strict bioriented functors $$S \circ (-)^{\co\op} \cong (-)^{\co\op}  \circ S: {^\cop\infty\Cat^\strict} \to \infty\Cat^\strict_{\partial\bD^1/}.$$

\end{enumerate}		
\end{proposition}

\begin{proof}
(1): The functor $S: \infty\Cat^\strict \to \infty\Cat^\strict_{\partial\bD^1/}$ is the pushout in the category of strict antioriented functors $  \infty\Cat^\strict \to \infty\Cat^\strict$
of the strict antioriented functors $\id \coprod \id \cong (-)\boxtimes \partial \bD^1 \to (-)\boxtimes \bD^1 $
and $\id \coprod \id \to * \coprod *.$

This gives the description of the map in (1). We apply \cref{bien}.
(2): The functor $\bar{S}$ is the pushout in the category of strict oriented functors $ \infty\Cat^\strict \to \infty\Cat^\strict$ of the strict oriented functors $\id \coprod \id \cong \partial \bD^1 \boxtimes (-) \to \bD^1 \boxtimes (-)$
and $\id \coprod \id \to * \coprod *.$ This gives the description of the map in (2). We apply \cref{bien}.

(3): By \cref{sw} (1) there is an isomorphism of functors \begin{equation}\label{hjh} S_{\mid \cube} \cong \bar{S}_{\mid \cube} \circ (-)_{\mid \cube}^{\co\op}: \cube \to \infty\Cat^\strict_{\partial\bD^1/}\end{equation}
whose component at any oriented cube $\cube^n$ for $n \geq 0$ is the isomorphism of Steiner $\infty$-categories corresponding to the isomorphism of Steiner complexes 
$$ S(\lambda(\cube^n)) \cong \bar{S}((\lambda(\cube^n))^{\co\op})$$ of \cite[Remark A.3.3., A.2]{GepnerHeine2026} that is the identity.
By cubical density the isomorphism \ref{hjh} uniquely extends to an isomorphism of functors $S \cong \bar{S} \circ (-)^{\co\op}: \infty\Cat^\strict \to \infty\Cat^\strict_{\partial\bD^1/}$.
By (1) the functor $S: \infty\Cat^\strict \to \infty\Cat^\strict_{\partial\bD^1/}$ is canonically a strict antioriented functor.

By (2) the functor $\bar{S}: \infty\Cat^\strict \to \infty\Cat^\strict_{\partial\bD^1/}$ is canonically a strict oriented functor.
Composing the latter with the strict oriented functor $ (-)^{\co\op}: (\infty\Cat^\strict)^\cop \to \infty\Cat^\strict$ 
we obtain a strict oriented functor $\bar{S} \circ (-)^{\co\op}: (\infty\Cat^\strict)^\cop \to \infty\Cat^\strict_{\partial\bD^1/}.$ 
The structure map $$\beta': S(\mB) \boxtimes \mC \cong \bar{S}(\mB^{\co\op}) \boxtimes \mC \xrightarrow{\beta} \bar{S}((\mB \boxtimes \mC^{\co\op})^{\co\op})\cong S(\mB \boxtimes \mC^{\co\op}) $$
for strict $\infty$-categories $\mB,\mC$ factors as 
$$\bar{S}(\mB^{\co\op}) \boxtimes \mC \to \bar{S}(\mB^{\co\op} \boxtimes \mC) \cong \bar{S}((\mB \boxtimes \mC^{\co\op})^{\co\op}) .$$

To see that the strict left and oriented functor 
$S \cong \bar{S} \circ (-)^{\co\op}: \infty\Cat^\strict \to \infty\Cat^\strict_{\partial\bD^1/}$
is a strict Gray-functor by \cref{strict} it suffices to show that for every $\mA,\mB,\mC \in \infty\Cat^\strict$ the following square commutes:
$$\begin{xy}
\xymatrix{
\mA \boxtimes S(\mB) \boxtimes \mC  \ar[d]^{\alpha \boxtimes\mC} \ar[r]^{\mA \boxtimes \beta'}
& \mA \boxtimes S(\mB\boxtimes\mC^{\co\op}) \ar[d]^{\alpha}
\\ 
S(\mA \boxtimes \mB) \boxtimes \mC
\ar[r]^{\beta'} & S(\mA \boxtimes \mB \boxtimes \mC^{\co\op}).
}
\end{xy}$$		
	
By cubical density we can assume that $\mA,\mB,\mC$ are oriented cubes $\cube^\bk, \cube^\ell, \cube^\m$ for $\bk, \ell, \m \geq 0.$
	
By (1) the functor $S: \infty\Cat^\strict \to \infty\Cat^\strict$ is a strict antioriented functor.
Moreover by (1) 
the canonical functor
$\alpha:  \cube^\bk \boxtimes S(\cube^\ell) \to S(\cube^\bk \boxtimes \cube^\ell)$ is the image of the map of augmented directed complexes:
\begin{equation}\label{ohast}
\lambda(\cube^\bk) \ot S(\lambda(\cube^\ell))\cong (\lambda(\cube^\bk) \ot \lambda(\cube^\ell))[1]\oplus \lambda(\cube^\bk) \oplus \lambda(\cube^\bk) \to S(\lambda(\cube^\bk) \ot \lambda(\cube^\ell))\cong (\lambda(\cube^\bk) \ot \lambda(\cube^\ell))[1]\oplus \bZ \oplus \bZ \end{equation} induced by the augmentation $ \lambda(\cube^\bk) \to \bZ$. 

By (2) the functor $\bar{S}: \infty\Cat^\strict \to \infty\Cat^\strict$ is a strict oriented functor.
Moreover the canonical functor
$ \beta: \bar{S}(\cube^\ell)\boxtimes \cube^\m \to \bar{S}(\cube^\ell \boxtimes \cube^\m)$ is the image of the map of augmented directed complexes:
$$ \bar{S}(\lambda(\cube^\ell)) \ot \lambda(\cube^m) \cong (\lambda(\cube^\ell) \ot \lambda(\cube^m))[1]\oplus \lambda(\cube^m) \oplus \lambda(\cube^m) \to $$$$ \bar{S}(\lambda(\cube^\ell) \ot \lambda(\cube^m))\cong (\lambda(\cube^\ell) \ot \lambda(\cube^m))[1]\oplus \bZ \oplus \bZ$$ induced by the augmentation $ \lambda(\cube^m) \to \bZ$. 
Thus the canonical functor 
$$\beta':  S(\cube^\ell)\boxtimes \cube^\m \cong \bar{S}((\cube^\ell)^\mathrm{coop})\boxtimes \cube^\m \to \bar{S}((\cube^\ell)^\mathrm{coop}\boxtimes \cube^\m) \cong \bar{S}((\cube^\ell\boxtimes (\cube^\m)^\mathrm{coop})^\mathrm{coop}) \cong S( \cube^\ell\boxtimes (\cube^\m)^\mathrm{coop} )$$ is also the image of the following map of augmented directed complexes:
$$ S(\lambda(\cube^\ell)) \ot \lambda(\cube^m) \cong (\lambda(\cube^\ell) \ot \lambda(\cube^m))[1]\oplus \lambda(\cube^m) \oplus \lambda(\cube^m) \to $$$$ S(\lambda(\cube^\ell) \ot (\lambda(\cube^m))^\mathrm{coop})\cong (\lambda(\cube^\ell) \ot \lambda(\cube^m))[1]\oplus \bZ \oplus \bZ$$ induced by the augmentation $ \lambda(\cube^m) \to \bZ$.
Hence the latter square is the image of the following square of augmented directed complexes:
\begin{equation}\label{rrr}
\begin{xy}
\xymatrix{
\lambda(\cube^\bk) \ot S(\lambda(\cube^\ell)) \ot \lambda(\cube^m)  \ar[d]^{} \ar[r]^{}
& \lambda(\cube^\bk) \ot S(\lambda(\cube^\ell)\ot(\lambda(\cube^m))^{\co\op}) \ar[d]^{}
\\ 
S(\lambda(\cube^\bk) \ot \lambda(\cube^\ell)) \ot \lambda(\cube^m)
\ar[r]^{} & S(\lambda(\cube^\bk) \ot \lambda(\cube^\ell) \ot (\lambda(\cube^m))^{\co\op}),
}
\end{xy}\end{equation} which identifies with the following commutative square induced by the augmentations, where $T:=(\lambda(\cube^\bk)\ot\lambda(\cube^\ell) \ot \lambda(\cube^m))[1]:$
$$\begin{xy}
\xymatrix{
T \oplus \lambda(\cube^\bk)\ot\lambda(\cube^m) \oplus \lambda(\cube^\bk)\ot\lambda(\cube^m) \ar[d]^{} \ar[r]
&T \oplus \lambda(\cube^\bk) \oplus \lambda(\cube^\bk) \ar[d] \ar[d]^{}
\\ 
T \oplus \lambda(\cube^m) \oplus \lambda(\cube^m)
\ar[r] &T \oplus \bZ \oplus \bZ.
}
\end{xy}$$	
	
(4): By \cref{sw} (3) there is a canonical isomorphism $S \circ (-)^{\co\op} \simeq (-)^{\co\op}  \circ S$ of functors $ {\infty\Cat^\strict} \to \infty\Cat^\strict_{\partial\bD^1/}.$
By (3) both functors are bioriented functors $ {^\cop\infty\Cat^\strict} \to \infty\Cat^\strict_{\partial\bD^1/}$
and we need to see that the canonical isomorphism $S \circ (-)^{\co\op} \simeq (-)^{\co\op}  \circ S$
identifies the structures of antioriented functor and oriented functor.
We prove this for the case of antioriented functors. The other case is similar.
Let $\mA, \mB$ be strict $\infty$-categories. We need to verify that both canonical functors
$$ \mA \boxtimes S(\mB^{\co\op}) \to S(\mA \boxtimes \mB^{\co\op}) \cong S((\mA^{\co\op} \boxtimes \mB)^{\co\op}), $$
$$ \mA \boxtimes S(\mB)^{\co\op} \cong (\mA^{\co\op} \boxtimes S(\mB))^{\co\op} \to 
S(\mA^{\co\op} \boxtimes \mB)^{\co\op}$$
are equivalent.
By cubical density we can assume that $\mA=\cube^\bk,\mB=\cube^\ell$ are oriented cubes.
In this case both functors are the image of the similarly defined maps of augmented directed complexes for the complexes $ \lambda(\cube^\bk), \lambda(\cube^\ell).$
Both maps identify with map \ref{ohast}.
\end{proof}


\subsection{Oriented functoriality of suspension and morphism objects}
The suspension functor
\[
S: \infty\Cat \to \infty\Cat_{\partial\bD^1/}
\]
is a functor of categories but not $\infty$-categories.
However, because the suspension can be defined in terms of the Gray tensor product, which is oriented, and not antioriented (by convention), it turns out that the suspension functor is antioriented, and the antisuspension functor is oriented.

\begin{proposition}\label{alf}	
\begin{enumerate}[\normalfont(1)]\setlength{\itemsep}{-2pt}
\item The suspension functor refines to an antioriented functor
\[
S: \infty\fcat \to \infty\fcat_{\partial\bD^1/}.
\]
For every $\infty$-categories $\mA, \mB$ the canonical functor
$$ \alpha: \mA \boxtimes S(\mB) \to S(\mA\boxtimes \mB)$$ is the canonical functor
$$ \mA \boxtimes (\mB \boxtimes \bD^1 \coprod_{\mB \boxtimes \partial \bD^1} \partial \bD^1)
\simeq \mA \boxtimes \mB \boxtimes \bD^1 \coprod_{\mA \boxtimes \mB \boxtimes \partial \bD^1} \mA \boxtimes \partial \bD^1 \to \mA \boxtimes \mB \boxtimes \bD^1 \coprod_{\mA \boxtimes \mB \boxtimes \partial \bD^1} \partial \bD^1$$ induced by the functor $\mA \to *.$
\item The antisuspension functor refines to an oriented functor
\[
\bar{S}: \infty\fcat \to \infty\fcat_{\partial\bD^1/}.
\]
For every $\infty$-categories $\mA, \mB$ the canonical functor
$$ \alpha: \bar{S}(\mA) \boxtimes \mB \to \bar{S}(\mA\boxtimes \mB)$$ is the canonical functor
$$ (\bD^1 \boxtimes \mA \coprod_{\partial \bD^1 \boxtimes \mA} \partial \bD^1) \boxtimes \mB
\simeq \bD^1 \boxtimes \mA \boxtimes \mB \coprod_{\partial \bD^1 \boxtimes \mA  \boxtimes \mB} \partial \bD^1 \boxtimes \mB \to \bD^1 \boxtimes \mA \boxtimes \mB \coprod_{\partial \bD^1 \boxtimes \mA \boxtimes \mB} \partial \bD^1$$ induced by the functor $\mB \to *.$
\end{enumerate}

\end{proposition}

\begin{proof}
We prove (1). The proof of (2) is similar.
The functor $S$ is the pushout in the category of antioriented functors $  \infty\Cat \to \infty\Cat$
of the antioriented functors $$\id \coprod \id \cong (-)\boxtimes \partial \bD^1 \to (-)\boxtimes \bD^1 , \ \id \coprod \id \to * \coprod *.$$
This gives the description of the structure map. We apply \cref{bien}.
\end{proof}

\begin{notation}
Let $\mC, \mD$ be $\infty$-categories and $\X,\Y \in \mC, \X',\Y' \in \mD.$
Let $$ \Fun_{\partial\bD^1/}^\oplax((\mC,\X,\Y), (\mD,\X',\Y'))$$ be the fiber over $(\X',\Y') \in \mD \times \mD$ of the following functor evaluating at $(\X,\Y) \in \mC \times \mC$:
$$ \Fun^\oplax(\mC,\mD) \to \mD \times \mD.$$ 

\end{notation}

\begin{remark}\label{enrstar}

By (the dual of) \cref{bien} there is a canonical equivalence $$ \Fun_{\partial\bD^1/}^\oplax((\mC,\X,\Y), (\mD,\X',\Y')) \simeq \L\Mor_{\infty\Cat_{\partial \bD^1/}}((\mC,\X,\Y), (\mD,\X',\Y')). $$ 

\end{remark}

\begin{proposition}\label{hom}

For every $\infty$-category $\mC$ and $\X,\Y \in \mC$ there is a canonical pullback square in $\infty\Cat$
$$\begin{xy}
\xymatrix{
\Mor_\mC(\X,\Y) \ar[d] \ar[r]
& \Fun^\oplax(\bD^1,\mC) \ar[d]
\\ 
\ast \ar[r]^{(\X,\Y)} & \mC \times \mC
}
\end{xy}$$
such that, for every $\infty$-category $\mB$, there is a canonical equivalence
$$ \Fun^\oplax(\mB, \Mor_\mC(\X,\Y)) \simeq \Fun_{\partial\bD^1/}^\oplax(S(\mB),\mC).$$

\end{proposition}

\begin{proof}

By \cref{thas2} for every $\infty$-category $\mB$ there is a canonical equivalence
$$ \Fun^\oplax(\mB, \ast \times_{(\mC \times \mC)} \Fun^\oplax(\bD^1,\mC)) \simeq 
\ast \times_{\Fun^\oplax(\mB,\mC) \times \Fun^\oplax(\mB,\mC)} \Fun^\oplax(\mB,\Fun^\oplax(\bD^1,\mC)) \simeq $$$$
* \times_{(\mC \times \mC)}(\mC \times \mC) \times_{\Fun^\oplax(\mB \boxtimes \partial \bD^1, \mC)} \Fun^\oplax(\mB \boxtimes \bD^1,\mC) \simeq * \times_{\Fun_{}^\oplax(\partial \bD^1,\mC)}\Fun^\oplax(\partial \bD^1 \coprod_{\mB \boxtimes \partial\bD^1} \mB \boxtimes \bD^1,\mC)$$
$$ \simeq \Fun_{\partial\bD^1/}^\oplax(\partial \bD^1 \coprod_{\mB \boxtimes \partial\bD^1} \mB \boxtimes \bD^1,\mC) \simeq \Fun_{\partial\bD^1}^\oplax(S(\mB),\mC).$$
Passing to maximal subspaces the result follows from the universal property of $S(\mC)$ of \cref{suspi}.	
\end{proof}

\begin{corollary}\label{Morenrfun}

The antioriented functor $S: \infty\fcat \to \infty\fcat_{\partial \bD^1 /} $ admits an antioriented right adjoint
\[
\Mor: \infty\fcat_{\partial \bD^1/}  \to \infty\fcat
\]
sending $(\mC,\X,\Y)$ to $\Mor_\mC(\X,\Y).$
\end{corollary}

\begin{proof}
By \cref{hom} the induced functor
$\Fun^\oplax(\mB, \Mor_\mC(\X,\Y)) \to \Fun_{\partial\bD^1/}^\oplax(S(\mB),\mC) $
is an equivalence.
In view of \cref{enrstar} this proves the result by \cite[Remark 2.72]{heine2024bienriched}.
\end{proof}

In the following we will refine \cref{alf}.

\begin{notation}
Let $\mC, \mD$ be bioriented categories. Let 	
$${\boxtimes\Fun\boxtimes}^\colim(\mC,\mD) \subset {{\boxtimes\Fun}\boxtimes}(\mC,\mD) $$ be the full subcategory of bioriented functors preserving small colimits.
		
\end{notation}
	
\begin{proposition}\label{hal}
	
Let $\mC$ be a bioriented category 
that admits small colimits. 
The functor
$${\boxtimes\Fun\boxtimes}^{\colim}(\infty\Cat,\mC)\to {\boxtimes\Fun\boxtimes}(\cube,\mC)$$
is fully faithful and the essential image consists precisely of those bioriented functors $\cube \to \mC$
whose underlying functor extends to a small colimit preserving functor $\infty\Cat \to \mC.$
\end{proposition}
	
\begin{proof}
		
By \cite[Proposition 2.105]{heine2024bienriched} the following induced functor is an equivalence.
$$ _{(\mP(\cube),\boxtimes) }\Fun_{(\mP(\cube),\boxtimes)}^{\mathrm{colim}}(\mP(\cube),\mC)\to {_{(\cube,\boxtimes)}\Fun_{(\cube,\boxtimes)}(\cube,\mC)}={\boxtimes\Fun\boxtimes}(\cube,\mC).$$
Here colim refers to bienriched functors whose underlying functor preserves small colimits.
The monoidal localization $\mP(\cube) \rightleftarrows \infty\Cat$ of \cref{locmon2} induces a localization
$$_{(\mP(\cube),\boxtimes)}\Fun_{(\mP(\cube),\boxtimes)}^{\mathrm{colim}}(\mP(\cube),\mC)\rightleftarrows {_{(\infty\Cat,\boxtimes)}\Fun_{(\infty\Cat,\boxtimes)}^{\mathrm{colim}}(\infty\Cat,\mC)}={\boxtimes\Fun\boxtimes}^{\mathrm{colim}}(\infty\Cat,\mC)$$
whose local objects are the enriched functors whose underlying functor preserves small colimits and local equivalences for the localization $\mP(\cube) \rightleftarrows \infty\Cat$, or equivalently descends to a small colimits preserving functor $\infty\Cat \to \mC.$
\end{proof}

\begin{theorem}\label{alfas2}
\begin{enumerate}[\normalfont(1)]\setlength{\itemsep}{-2pt}
\item The suspension functor $S: \infty\Cat \to \infty\Cat_{\partial\bD^1/}$ of $1$-categories refines to a bioriented functor $S:\infty\fcat^{\cop}\to\infty\fcat_{\partial\bD^1/}$. 
		
\item The antisuspension functor $\overline{S}: {^\cop}\infty\Cat \to \infty\Cat_{\partial\bD^1/}$ refines to a bioriented functor $\overline{S}: {^\cop}\infty\fcat \to \infty\fcat_{\partial\bD^1/}$
		
\item There is an equivalence of bioriented functors $S \simeq \overline{S} \circ (-)^{\co\op}: \infty\fcat^{\cop} \to \infty\fcat_{\partial\bD^1/}.$
		
\item There is an equivalence of bioriented functors $(-)^{\co\op} \circ S \simeq S \circ (-)^{\co\op}: {^{\cop}\infty\fcat} \to \infty\fcat_{\partial\bD^1/}.$
		
\end{enumerate}
	
\end{theorem}

\begin{proof}
	
(1): The lax monoidal embedding $ \infty\Cat^\strict \subset \infty\Cat $
under $(\cube,\boxtimes)$ is in particular a $(\cube, \boxtimes)$-bienriched functor. By \cref{bien} the latter gives rise to a $(\cube, \boxtimes)$-bienriched embedding $ \rho :\infty\Cat^\strict_{\partial\bD^1/} \to \infty\Cat_{\partial\bD^1/}.$
By \cref{alfa} the functor $S : \infty\Cat^\strict \to \infty\Cat^\strict_{\partial\bD^1/} $ 
refines to a strict bioriented functor
\[
S : \infty\fcat^\strict \to (\infty\fcat^\strict_{\partial\bD^1/})^\cop,
\]
which is in particular a $(\cube, \boxtimes)$-bienriched functor by restriction along the monoidal embedding
$(\cube, \boxtimes) \subset (\infty\Cat^\strict,\boxtimes).$
The composition of the $(\cube, \boxtimes)$-bienriched embedding $\cube \subset \infty\Cat^\strict $ followed by the $(\cube, \boxtimes)$-bienriched functor $S: \infty\Cat^\strict \to(\infty\Cat^\strict)^\cop_{\partial\bD^1/} $ and the $(\cube, \boxtimes)$-bienriched embedding $$ \rho:\infty\Cat^\strict_{\partial\bD^1/} \to \infty\Cat_{\partial\bD^1/}$$ is a $(\cube, \boxtimes)$-bienriched functor
\[
\phi: \cube \to \infty\Cat^\cop_{\partial\bD^1/}
\]
whose underlying functor is $\cube \subset \infty\Cat \xrightarrow{S} \infty\Cat_{\partial\bD^1/}.$ 
Since $S: \infty\Cat \to \infty\Cat_{\partial\bD^1/}$ preserves small colimits, by \cref{hal} the $(\cube, \boxtimes)$-bienriched functor $\phi$ uniquely extends to a small colimit preserving bioriented functor $\infty\fcat \to \infty\fcat^\cop_{\partial\bD^1/}$ whose underlying functor is $S: \infty\Cat \to \infty\Cat_{\partial\bD^1/}$.
The proof of (2) is similar.

(3): By \cref{alfa} (3) there is an isomorphism of strict bioriented functors $$S \cong \bar{S} \circ (-)^{\co\op}: \infty\fcat^\strict \to(\infty\fcat^\strict)^\cop_{\partial\bD^1/}.$$
By cubical density the restricted and prolonged equivalence 
$$S_{\mid \cube} \simeq \rho \circ S_{\mid \cube} \simeq \rho \circ \bar{S}_{\mid \cube}  \circ (-)_{\mid \cube}^{\co\op} \simeq \bar{S}_{\mid \cube}  \circ (-)_{\mid \cube}^{\co\op}: \cube \to \infty\Cat^\cop_{\partial\bD^1/} $$ 
of $(\cube, \boxtimes)$-bienriched functors uniquely extends to an equivalence of bioriented functors 
$$S \simeq \bar{S} \circ (-)^{\co\op}: \infty\fcat \to \infty\fcat^\cop_{\partial\bD^1/}.$$

(4): By \cref{alfa} (4) there is an isomorphism of strict bioriented functors $$(-)^{\co\op} \circ S \cong S \circ (-)^{\co\op}: {^{\cop}\infty\fcat^\strict} \to \infty\fcat^\strict_{\partial\bD^1/}.$$
By cubical density the restricted and prolonged equivalence $$(-)^{\co\op} \circ S_{\mid \cube} \simeq (-)^{\co\op} \circ \rho \circ S_{\mid \cube} \simeq \rho \circ (-)^{\co\op} \circ S_{\mid \cube} \cong \rho \circ S_{\mid \cube} \circ (-)^{\co\op}_{\mid\cube} \simeq S_{\mid \cube} \circ (-)^{\co\op}_{\mid \cube} : \cube \to \infty\Cat_{\partial\bD^1/} $$
of $(\cube, \boxtimes), (\cube, \boxtimes)$-bienriched functors uniquely extends to an equivalence of bioriented functors
$$(-)^{\co\op} \circ S \simeq S \circ (-)^{\co\op}: {^{\cop}\infty\fcat} \to \infty\fcat_{\partial\bD^1/}.$$ 
\end{proof}

\subsection{Oriented functoriality of join and slice}

\begin{proposition}

The join is a bioriented functor $\star: \langle\infty\fcat , \infty\fcat \rangle\to \infty\fcat$. For every $A,B,X,Y \in \infty\fcat$ the canonical morphism
$$ X \boxtimes (A \star B) \boxtimes Y \to (X \boxtimes A) \star (B \boxtimes Y), $$
is the cobase change of the projection
$$ X \boxtimes A \boxtimes Y \coprod X \boxtimes B \boxtimes Y \to X \boxtimes A \coprod B \boxtimes Y . $$
\end{proposition}

\begin{proof}

The inclusions $ \{ i\} \subset \bD^1$ for $i=0,1$
give rise to a morphism $(-) \boxtimes \{ i\} \boxtimes (-) \to (-) \boxtimes K \boxtimes (-)$ of
bioriented functors $ \langle\infty\fcat , \infty\fcat\rangle \to \infty\fcat.$
The morphism $ \id \to \bD^0 $ of (anti)oriented functors
$\infty\Cat \to \infty\Cat$ gives rise to bioriented functors $ \langle\infty\fcat , \infty\fcat \rangle\to \infty\fcat:$
$$ (-) \boxtimes \{ i\} \boxtimes (-) \to (-) \boxtimes \{i\} \boxtimes (-) \circ (\id \times \bD^0) \simeq \mathrm{pr}_{i+1}. $$
By the join formula the join underlies the following pushout of bioriented functors
$\langle\infty\fcat , \infty\fcat\rangle \to \infty\fcat : $
$$
\begin{xy}
\xymatrix{
(-)\boxtimes \{0\} \boxtimes (-) \coprod (-)\boxtimes \{1\} \boxtimes (-) \ar[d] \ar[r] 
& \mathrm{pr}_1 \coprod \mathrm{pr}_2 \ar[d]
\\ 
(-) \boxtimes \bD^1 \boxtimes (-) \ar[r] & (-)\star (-) 
}
\end{xy}
$$
\end{proof}

\begin{corollary}
Let $\mC$ be a small $\infty$-category.
\begin{enumerate}[\normalfont(1)]\setlength{\itemsep}{-2pt}
\item There is an antioriented adjunction
$ (-) \star \mC: \infty\fcat \rightleftarrows \infty\fcat_{\mC/ }: \mC_{/\!/^\lax (-)}$
\item There is an oriented adjunction
$ \mC \star (-): \infty\fcat \rightleftarrows \infty\fcat_{\mC/ }: \mC_{(-) /\!/^\oplax }.$
\item There is an oriented adjunction
$ ((-) \star \mC^\co)^\co: \infty\fcat \rightleftarrows \infty\fcat_{\mC/ }: (\mC^\co_{/\!/^\lax (-)^\co})^\co.$
\item There is an antioriented adjunction
$ (\mC^\co \star (-))^\co: \infty\fcat \rightleftarrows \infty\fcat_{\mC/ }: (\mC^\co_{(-)^\co/\!/^\oplax})^\co.$
\end{enumerate}
\end{corollary}

\subsection{Oriented pullbacks}

In this subsection we use the notion of weighted colimit, the enriched analogue of the notion of colimit \cite[Section 3]{heine2024higher} or \cite{hinich2021colimits}, to define the notions of oriented pushouts and pullbacks in oriented categories and the dual notion of antioriented pushouts and pullbacks in antioriented categories.





\begin{definition}
Let $\mC$ be an oriented category.
\begin{enumerate}[\normalfont(1)]\setlength{\itemsep}{-2pt}
\item 
The oriented pullback of the diagram $X \to Z\leftarrow Y$ in $\mC$ is a diagram
\[
\xymatrix{& W\ar[rd]\ar[ld] &\\
X\ar[rd] & \Longrightarrow & Y\ar[ld]\\
& Z &}
\]
in $\mC$ such that for all objects $T$ of $\mC$ the induced functor
\[
\R\Mor_{\mC}(T,W)\to \R\Mor_{\mC}(T,X) \underset{\R\Mor_{\mC}(T,Z)}{\times}
\Fun^\oplax(\bD^1,\R\Mor_{\mC}(T,Z)) \underset{\R\Mor_{\mC}(T,Z)}{\times} \R\Mor_{\mC}(T,Y)
\]
is an equivalence in $\infty\Cat$.
In this case, we will write $X\underset{Z}{\vec{\times}} Y$ or $ Y\underset{Z}{{\cev\times}} X$ for $W.$

\item 
The oriented pushout of the oriented diagram $X \leftarrow Z\to Y$ in $\mC$ is the oriented pullback of the corresponding diagram in the oriented category $\mC^\op,$ which we denote by $ X\underset{Z}{{\vec{+}}} Y $ or $Y\underset{Z}{\cev{+}} X$.

\end{enumerate}

\end{definition}

\begin{remark}
An oriented pullback square in an oriented category $\mC$ is a diagram $\cube^2\to\mC$ which satisfies the universal property above.
\end{remark}





\begin{definition}
Let $\mC$ be an antioriented category.
\begin{enumerate}[\normalfont(1)]\setlength{\itemsep}{-2pt}
\item 
The antioriented pullback of the diagram $X \to Z\leftarrow Y$ in $\mC$ is the oriented pushout of the corresponding diagram in the oriented category $\mC^\circ$, which we denote by $X\underset{Z}{\bar{\vec{\times}}} Y$ or $ Y\underset{Z}{{\bar{\cev\times}}} X$.

\item 
The antioriented pushout of the diagram $X \leftarrow Z\to Y$ in $\mC$
is the oriented pullback of the corresponding diagram in the oriented category $\mC^\circ$, which we denote by $ X\underset{Z}{{\bar{\vec{+}}}} Y $ or $Y\underset{Z}{\bar{\cev{+}}} X.$

\end{enumerate}

\end{definition}

\begin{definition}Let $\mC$ be a bioriented category.
\begin{enumerate}[\normalfont(1)]\setlength{\itemsep}{-2pt}

\item The (anti)oriented pullback of a diagram $A \to C\leftarrow B$ in $\mC$ is the (anti)oriented pullback in the underlying (anti)oriented category of $\mC$.

\item The (anti)oriented pushout of a diagram $A \leftarrow C \to B$ in $\mC$ is the (anti)oriented pushout in the underlying (anti)oriented category of $\mC$.

\end{enumerate}

\end{definition}


\begin{remark}
Note that the oriented pullback is not symmetric since it makes use of a specified ordering of the sources of the maps $A\to C$ and $B\to C$.
\end{remark}



\begin{lemma}\label{0desc}\label{adesc}
Let $\mC$ be an oriented category.
\begin{enumerate}[\normalfont(1)]\setlength{\itemsep}{-2pt}
\item Let $A \leftarrow C \to \B$ be morphisms in $\mC$. If $\mC$ admits pushouts and right tensors with $\bD^1$, 
there is a canonical equivalence $$\A\,\underset{C}{\vec{+}}\, \B \simeq \A\!\!\underset{C \otimes \{0\}}{+} \!\!(\C \ot \bD^1) \!\!\underset{C \otimes \{1\}}{+}\!\!\B.$$	

\item Let $A \to C\leftarrow B$ be morphisms in $\mC$. If $\mC$ admits pullbacks and right cotensors with $\bD^1$, 
there is a canonical equivalence $${\A \,\underset{\C}{\vec{\times}}}\, \B \simeq \A \!\!\underset{\C^{\{0\}}}{\times} \!\!{\C^{\bD^1}}\!\!\underset{{\C^{\{1\}}}}{\times}  \!\!\B.$$	
\end{enumerate}

\end{lemma}

\begin{proof}

We prove (1). The proof of (2) is similar.
For every $\X \in \mC$ there is a canonical equivalence
$$\R\Mor_{\mathcal{C}}(\A\,\underset{C}{\vec{+}}\, \B, \X) \simeq \R\Mor_{\mC}(A,X) \underset{\R\Mor_{\mC}(C,X)}{\times}
\Fun^\lax(\bD^1,\R\Mor_{\mC}(C,X)) \underset{\R\Mor_{\mC}(C,X)}{\times} \R\Mor_{\mC}(B,X) \simeq $$
$$ \R\Mor_{\mathcal{C}}(\A\!\!\underset{\{0\}\otimes\C}{+} \!\!(\C \ot \bD^1) \!\!\underset{\{1\}\otimes\C}{+}\!\!\B, \X). $$
\end{proof}

Dually, we obtain the following:

\begin{corollary}\label{bdesc}
Let $\mC$ be an antioriented category.
\begin{enumerate}[\normalfont(1)]\setlength{\itemsep}{-2pt}
\item 
Let $A \leftarrow C \to \B$ be morphisms in $\mC$. If $\mC$ admits pushouts and left tensors with $\bD^1$, 
there is a canonical equivalence $$\A\,\underset{C}{\bar{\vec{+}}}\, \B \simeq \A \underset{\{0\}\otimes\C}{+} ( \bD^1\ot \C) \underset{\{1\}\otimes\C}{+} \B.$$	

\item Let $A \to C\leftarrow B$ be morphisms in $\mC$. If $\mC$ admits pullbacks and left cotensors with $\bD^1$, 
there is a canonical equivalence $$\A\underset{C}{\bar{\vec{\times}}} \B \simeq \A \underset{^{\{0\}}\C}{\times}{^{\bD^1}\C} \underset{^{\{1\}}\C}{\times} \B.$$	
\end{enumerate}

\end{corollary}

\begin{corollary}\label{redfib}
\begin{enumerate}[\normalfont(1)]\setlength{\itemsep}{-2pt}
\item Let $\mC$ be a reduced oriented category that admits pushouts and reduced right tensors with $\bD^1$ and $\C \to \A$ a morphism in $\mC$.
There are canonical equivalences
\[
0\,\underset{C}{\vec{+}}\, \A  \simeq \C \wedge (\bD^1,0) \underset{\C \wedge (\partial\bD^1,0)}{+}\A \qquad\textrm{and} \qquad \A\,\underset{C}{\vec{+}}\, 0 \simeq \A \underset{\C \wedge(\partial\bD^1,1)}{+} \C \wedge (\bD^1,1).
\]
\item Let $\mC$ be a reduced oriented category that admits pullbacks and reduced right cotensors with $\bD^1$ and $\A \to \C$ a morphism in $\mC$.
There are canonical equivalences
\[
0\,\underset{C}{\vec{\times}}\,\A \simeq \C^{(\bD^1,0)}\underset{\C^{(\partial\bD^1,0)}}{\times} A\qquad\textrm{and}\qquad \A\,\underset{C}{\vec{\times}}\,0\simeq \A \underset{\C^{(\partial\bD^1,1)}}{\times} \C^{(\bD^1,1)}
\]
\item Let $\mC$ be a reduced antioriented category that admits pushouts and reduced left tensors with $\bD^1$ and $\C \to \A$ a morphism in $\mC$.
There are canonical equivalences
\[
0\,\underset{C}{\bar{\vec{+}}}\, \A  \simeq (\bD^1,0) \wedge \C \underset{(\partial\bD^1,0)\wedge \C}{+}\A \qquad\textrm{and} \qquad \A\,\underset{C}{\bar{\vec{+}}}\, 0 \simeq \A \underset{(\partial\bD^1,1)\wedge \C}{+} (\bD^1,1) \wedge \C.
\]

\item Let $\mC$ be a reduced antioriented category that admits pullbacks and reduced left cotensors with $\bD^1$ and $\A \to \C$ a morphism in $\mC$.
There are canonical equivalences
\[
0\,\underset{C}{\bar{\vec{\times}}}\,\A \simeq {^{(\bD^1,0)}\C}\underset{^{(\partial\bD^1,0)}\C}{\times} A\qquad\textrm{and}\qquad \A\,\underset{C}{\bar{\vec{\times}}}\,0\simeq \A \underset{^{(\partial\bD^1,1)}\C}{\times} ^{(\bD^1,1)}\C
\]

\end{enumerate}

\end{corollary}

\begin{proof}
Immediate from \cref{adesc}, \cref{bdesc}, and
the fact that for every $\X \in \infty\Cat_*$ the reduced left cotensor ${^\X\C}$ is the pullback $0 \times_{\C}{^\X\C}$ and the reduced right cotensor $\C^\X$ is the pullback ${\C^\X}\times_{\C}0.$
\end{proof}

\begin{corollary}Let $\mC$ be a reduced oriented category and $\X \in \mC$.
The right tensor $\X \wedge (\bD^1/\partial\bD^1)$ and the oriented pushout $0\,\underset{C}{\vec{+}}\, 0$ both satisfy the same universal property.
\end{corollary}

\begin{corollary}Let $\mC$ be a reduced antioriented category and $\X \in \mC$.
The left tensor $(\bD^1/\partial\bD^1)\wedge X$ and the antioriented pushout $0\,\underset{C}{\bar{\vec{+}}}\, 0$ both satisfy the same universal property.
\end{corollary}

Finally, we have the following pasting law:

\begin{lemma}\label{pasting}

Consider the following diagram in any oriented category $\mC$, where the left hand square is a commutative square:
\[
\begin{tikzcd}
\Q \ar{d} \ar{r} & \P \ar{r}{} \ar{d}[swap]{} & \B  \ar{d}{} \\
\E \ar{r} & \A \ar[double]{ur}{}  \ar{r}[swap]{} & \C
\end{tikzcd}
\]
If the right hand square is an oriented pullback square, the left hand square is a pullback square if and only if the outer square is an oriented pullback square.

\end{lemma}

\begin{proof}

Mapping out of any object of $\mC$ we can assume that $\mC=\infty\Cat$. So we can assume that
$\mC$ admits pullbacks and oriented pullbacks.	
The canonical morphism $$\P \to \A \,\underset{C}{\vec{\times}}\, \B \simeq \A \times_{\C^{\{0\}}} \C^{\bD^1} \times_{\C^{\{1\}}} \B$$ is an equivalence.
The canonical morphism $ \Q \to\E \,\underset{C}{\vec{\times}}\, \B \simeq \E \times_{\C^{\{0\}}} \C^{\bD^1} \times_{\C^{\{1\}}} \B$ factors as the composite
$$ \Q \to \E \times_\A \P \simeq \E \times_\A \A \times_{\C^{\{0\}}} \C^{\bD^1} \times_{\C^{\{1\}}} \B \simeq \E \times_{\C^{\{0\}}} \C^{\bD^1} \times_{\C^{\{1\}}} \B.$$	
\end{proof}

\begin{proposition}\label{homs} 
Let $\F: \mA \to \mC,\G: \mB \to \mC$ be functors and $\A,\A'\in \mA, \B,\B' \in \mB$ and $\sigma: \F(\A)\to \G(\B), \sigma': \F(\A')\to \G(\B')$ morphisms. There is a canonical equivalence $$\Mor_{{\mA \,\underset{\mC}{\vec{\times}}}\, \mB}((\A,\B, \sigma),(\A',\B', \sigma')) \simeq {\Mor_\mB(\B,\B') \,\underset{\Mor_\mC(\F(\A),\G(\B'))}{\vec{\times}}}\, \Mor_{\mA}(\A,\A').$$

\end{proposition}

	


\begin{proof}

In view of \cref{0desc} and since pullbacks commute with forming morphism objects, we can reduce to the case that $\F, \G$ are the identities, to simplify notation.
So we want to see that there is a canonical equivalence $$\Mor_{\Fun^\lax(\bD^1,\mC)}((\A,\B, \sigma),(\A',\B', \sigma')) \simeq $$$$ \Mor_\mC(B,B') \times_{\Mor_\mC(A,B')} \Fun^\oplax(\bD^1,\Mor_\mC(A,B')) \times_{\Mor_\mC(A,B')}  \Mor_\mC(A,A').$$

There is a canonical equivalence $$\Mor_{\Fun^\oplax(\bD^1,\mC)}((\A,\B, \sigma),(\A',\B', \sigma')) \simeq $$
$$\{ \sigma\}\times_{\Fun^\oplax(\bD^1,\mC)} \Fun^\oplax(\bD^1,\Fun^\oplax(\bD^1,\mC)) \times_{\Fun^\lax(\bD^1,\mC)} \{ \sigma'\} \simeq $$
$$\{ \sigma\}\times_{\Fun_{\partial\bD^1/}^\lax(\bD^1,\mC)} \Fun^\oplax(\bD^1,\Fun^\oplax(\bD^1,\mC)) \times_{\Fun_{\partial\bD^1/}^\oplax(\bD^1,\mC)} \{ \sigma'\} \simeq $$
$$\{ \sigma\}\times_{\Mor_\mC(A,B)} \Fun^\oplax(\bD^1,\Fun^\oplax(\bD^1,\mC)) \times_{\Mor_\mC(A',B')} \{ \sigma'\} \simeq $$
$$\{ \sigma\}\times_{\Mor_\mC(A,B)} \Fun^\oplax(\bD^1 \boxtimes \bD^1,\mC) \times_{\Mor_\mC(A',B')} \{ \sigma'\} \simeq $$
$$\{ \sigma\}\times_{\Mor_\mC(A,B)} \Mor_\mC(A,B) \times \Mor_\mC(B,B') \times_{\Mor_\mC(A,B')} \Fun_{\partial\bD^1/}^\oplax(\bD^2,\mC)$$$$ \times_{\Mor_\mC(A,B')}  \Mor_\mC(A,A') \times \Mor_\mC(A',B') \times_{\Mor_\mC(A',B')} \{ \sigma'\} \simeq $$
$$\Mor_\mC(B,B') \times_{\Mor_\mC(A,B')} \Fun_{\partial\bD^1/}^\oplax(\bD^2,\mC) \times_{\Mor_\mC(A,B')}  \Mor_\mC(A,A') \simeq $$
$$ \Mor_\mC(B,B') \times_{\Mor_\mC(A,B')} \Fun^\oplax(\bD^1,\Mor_\mC(A,B')) \times_{\Mor_\mC(A,B')}  \Mor_\mC(A,A').$$

\end{proof}

\begin{corollary}\label{homso}

Let $\F: \mA \to \mC,\G: \mB \to \mC$ be functors and $\A,\A'\in \mA, \B,\B' \in \mB$ and $\sigma: \F(\A)\to \G(\B), \sigma': \F(\A')\to \G(\B')$ morphisms. There is a canonical equivalence $$\Mor_{\mA\underset{\mC}{\bar{\vec{\times}}} \mB}((\A,\B, \sigma),(\A',\B', \sigma')) \simeq \Mor_\mA(\A,\A')\underset{\Mor_\mC(\F(\A),\G(\B'))}{\bar{\vec{\times}}} \Mor_\mB(\B,\B').$$

\end{corollary}

\begin{proof}
By \cref{homs} there is a canonical equivalence 
$$\Mor_{{\mA \,\underset{\mC}{\bar{\vec{\times}}}}\, \mB}((\A,\B, \sigma),(\A',\B', \sigma')) \simeq \Mor_{(\mA^\co\underset{\mC^\co}{\vec{\times}} \mB^\co)^\co}((\A,\B, \sigma),(\A',\B', \sigma')) $$
$$ \simeq \Mor_{\mA^\co\underset{\mC^\co}{\vec{\times}} \mB^\co}((\A,\B, \sigma),(\A',\B', \sigma'))^\op \simeq$$
$$ (\Mor_{\mB^\co}(\B,\B')\underset{\Mor_{\mC^\co}(\F(\A),\G(\B'))}{\vec{\times}} \Mor_{\mA^\co}(\A,\A'))^\op \simeq
$$$$ {\Mor_{\mA^\co}(\A,\A')^\op \,\underset{\Mor_{\mC^\co}(\F(\A),\G(\B'))^\op}{\bar{\vec{\times}}}}\, \Mor_{\mB^\co}(\B,\B')^\op $$$$ \simeq {\Mor_{\mA}(\A,\A')\,\underset{\Mor_\mC(\F(\A),\G(\B'))}{\bar{\vec{\times}}}}\, \Mor_\mB(\B,\B'). $$
\end{proof}

\begin{corollary}\label{homs2}\label{homso2} Let $\mB$ be an $\infty$-category and $\sigma: \A \to \B, \sigma': \A' \to \B'$ morphisms in $\mB$. There are canonical equivalences $$\Mor_{\Fun^\lax(\bD^1,\mB)}(\sigma, \sigma') \simeq \Mor_\mA(\A,\A')\underset{\Mor_\mC(\A,\B')}{\bar{\vec{\times}}} \Mor_\mB(\B,\B'), $$
$$\Mor_{\Fun^\oplax(\bD^1,\mB)}(\sigma, \sigma') \simeq {\Mor_\mB(\B,\B') \,\underset{\Mor_\mC(\A,\B')}{\vec{\times}}}\, \Mor_{\mA}(\A,\A'). $$\end{corollary}

\section{\mbox{$\infty$-Categories as oriented categories}}

\subsection{The cartesian product is a quotient of the Gray tensor product}
The main purpose of this subsection is to establish that the canonical map from the Gray tensor product to the cartesian product of $\infty$-categories is an epimorphism.
This result will be used in the next subsection to show that the category of $\infty$-categories embeds fully faithfully into the category of bioriented (or oriented, or antioriented) categories.

The fact that the product is a quotient of the tensor is dual to the fact that $\infty$-categories of functors and strict transformations embed into functors and (op)lax transformations.
We first prove this latter statement, beginning with a computation of mapping spaces in a category of enriched functors out of a suspension.

\begin{proposition}\label{homsuspo}

Let $\mV$ be a presentably monoidal category, $A \in \mV$ and $\mC$ a left $\mV$-enriched category.
Let $F,G: S(\A) \to \mC$ be left $\mV$-enriched functors.
\begin{enumerate}[\normalfont(1)]\setlength{\itemsep}{-2pt}
\item There is a canonical pullback square of spaces
$$\begin{xy}
\xymatrix{
\Map_{_\mV\Cat(S(A), \mC)}(F,G) \ar[r] \ar[d] & \Map_\mC(F(1),G(1)) \ar[d]
\\ 
\Map_\mC(F(0),G(0)) \ar[r] & \Map_\mV(A, \L\Mor_\mC(F(0),G(1))).
}
\end{xy}$$

\item Let $\tau: \mV^\rev \to \mV$ be a monoidal involution and $\mC$ a right $\mV$-enriched category.
There is a canonical pullback square of spaces
$$\begin{xy}
\xymatrix{
\Map_{_\mV\Cat(S(\tau(A)), \tau^*(\mC))}(F,G) \ar[r] \ar[d] & \Map_\mC(F(1),G(1)) \ar[d]
\\ 
\Map_\mC(F(0),G(0)) \ar[r] & \Map_\mV(A, \R\Mor_\mC(F(0),G(1))).
}
\end{xy}$$

\item Let $\tau: \mV^\rev \to \mV$ be a monoidal involution and $\mC$ a $\mV, \mV$-enriched category.
There is a canonical pullback square in $\mV:$
$$\begin{xy}
\xymatrix{
\R\Mor_{_\mV\Cat(S(A), \mC)}(F,G) \ar[r] \ar[d] & \tau(\L\Mor_\mC(F(1),G(1))) \ar[d]
\\ 
\tau(\L\Mor_\mC(F(0),G(0))) \ar[r] & \tau(\L\Mor_\mV(A, \L\Mor_\mC(F(0),G(1)))).
}
\end{xy}$$
\end{enumerate}

\end{proposition}

\begin{proof}

For (1) and (2) we can assume that $\mC$ is a presentably left $\mV$-tensored category
since we can embed $\mC$ into the presentably left $\mV$-tensored 
category $\mC'$ of $\mV$-enriched presheaves via the left $\mV$-enriched Yoneda embedding. The latter induces an embedding of categories
$ _\mV\Cat(S(A), \mC) \to {_\mV\Cat(S(A), \mC')}$,
which under the assumptions of (2) refines to a right $\mV$-enriched embedding. So the results for $\mC'$ will apply the results for $\mC.$

(1): Let $F': A \to \L\Mor_{\mC}(F(0),F(1))$ and $ G': A \to \L\Mor_{\mC}(G(0),G(1))$ be the corresponding morphisms in $\mV.$ 
An object of the space $ \Map_{_\mV\Cat(S(A), \mC)}(F,G) $ is a left $\mV$-enriched transformation $F \to G$ of left $\mV$-enriched functors $S(A) \to \mC$ and so classified by a left $\mV$-enriched functor $S(A) \to \mC^{\bD^1}$ such that
$S(A) \to \mC^{\bD^1} \to \mC^{\{0\}}$ is $F$ and $S(A) \to \mC^{\bD^1} \to \mC^{\{1\}}$ is $G.$
The latter corresponds to morphisms $\alpha_i: F(i) \to G(i)$ for $i=0, 1 $
and a left $\mV$-enriched functor $A \to \L\Mor_{\mC^{\bD^1}}(\alpha_0,\alpha_1) $ such that
$$A \to \L\Mor_{\mC^{\bD^1}}(\alpha_0,\alpha_1) \to \L\Mor_{\mC^{\{0\}}}(F(0),F(1))$$ is $F'$ and $$A \to \L\Mor_{\mC^{\bD^1}}(\alpha_0,\alpha_1) \to \L\Mor_{\mC^{\{1\}}}(G(0),G(1))$$ is $G'.$
There is a canonical equivalence
$$ \L\Mor_{\mC^{\bD^1}}(\alpha_0,\alpha_1) \simeq \L\Mor_{\mC}(F(0),F(1)) \times_{\L\Mor_{\mC}(F(0),G(1))} \L\Mor_{\mC}(G(0),G(1)).$$
Consequently, an object of $\Map_{_\mV\Cat(S(A), \mC)}(F,G)$ corresponds to morphisms $\alpha_i: F(i) \to G(i)$ for $i=0, 1 $ and an equivalence between 
$$A \xrightarrow{F'} \L\Mor_{\mC}(F(0),F(1)) \to \L\Mor_{\mC}(F(0),G(1))$$ and $$A \xrightarrow{G'} \L\Mor_{\mC}(G(0),G(1)) \to \L\Mor_{\mC}(F(0),G(1)).$$
In  other words there is a canonical bijection of equivalence classes between the spaces $ \Map_{_\mV\Cat(S(A), \mC)}(F,G) $
and $$\Map_\mC(F(0),G(0)) \times_{\Map_\mV(A, \L\Mor_\mC(F(0),G(1)))} \Map_\mC(F(1),G(1)). $$

We prove (1) by showing that for every small space $X$ there is a canonical bijection of equivalence classes between the spaces $\Map_\mS(X, \Map_{_\mV\Cat(S(A), \mC)}(F,G)) $
and $$\Map_\mS(X, \Map_\mC(F(0),G(0)) \times_{\Map_\mV(A, \L\Mor_\mC(F(0),G(1)))} \Map_\mC(F(1),G(1))) $$ and so conclude (1) by the Yoneda lemma.
Let $G^X$ be the cotensor of $G$ and $X$, i.e. the limit of the constant diagram $X \to _\mV\Cat(S(A), \mC)$ with value $G,$ which is formed object-wise.
There are canonical equivalences of spaces 
$$\Map_\mS(X, \Map_{_\mV\Cat(S(A), \mC)}(F,G)) \simeq \Map_{_\mV\Cat(S(A), \mC)}(F,G^X)$$
and 
$$\Map_\mS(X, \Map_\mC(F(0),G(0)) \times_{\Map_\mV(A, \L\Mor_\mC(F(0),G(1)))} \Map_\mC(F(1),G(1))) \simeq $$$$ \Map_\mC(F(0),G^X(0)) \times_{\Map_\mV(A, \L\Mor_\mC(F(0),G^X(1)))} \Map_\mC(F(1),G^X(1))).$$
So replacing $G$ by $G^X$ we obtain (1).
(2) follows immediately from (1).

(3): We prove (2) by showing that for every object $V \in \mV$ there is a canonical bijection of equivalence classes between the spaces $\Map_\mV(V, \R\Mor_{_\mV\Cat(S(A), \mC)}(F,G)) $
and $$\Map_\mV(V, \LMor_\mC(F(0),G(0)) \times_{\L\Mor_\mV(A, \L\Mor_\mC(F(0),G(1)))} \L\Mor_\mC(F(1),G(1))) ,$$
which implies (2) by application of the Yoneda lemma.
By \cite[Corollary 4.32]{heine2024bienriched} the right $\mV$-enriched category $ _\mV\Cat(S(A), \mC)$
admits right cotensors which are formed objectwise since $\mC$ admits right cotensors.
Let $G^V$ be the right cotensor of $G$ and $V$ in $_\mV\Cat(S(A), \mC).$

There is a canonical equivalence of spaces 
$$\Map_\mV(V, \R\Mor_\mC(F(0),G(0)) \times_{\R\Mor_\mV(A, \R\Mor_\mC(F(0),G(1)))} \R\Mor_\mC(F(1),G(1))) \simeq $$$$ \Map_\mC(F(0),G^V(0)) \times_{\Map_\mV(A, \R\Mor_\mC(F(0),G^V(1)))} \Map_\mC(F(1),G^V(1))$$
$$ \simeq \Map_{_\mV\Cat(S(\tau(A)), \tau^*(\mC))}(F,G^V) \simeq $$
$$ \Map_\mV(V, \R\Mor_{_\mV\Cat(S(\tau(A)), \tau^*(\mC))}(F,G)),$$
where the second last equivalence is by (2). 
The latter represents an equivalence
$$ \R\Mor_{_\mV\Cat(S(\tau(A)), \tau^*(\mC))}(F,G) \simeq \R\Mor_\mC(F(0),G(0)) \times_{\R\Mor_\mV(A, \R\Mor_\mC(F(0),G(1)))} \R\Mor_\mC(F(1),G(1)).$$
Replacing $\mC$ by $\tau^*(\mC) $ and $A$ by $\tau(A)$
gives an equivalence
$$ \R\Mor_{_\mV\Cat(S(A), \mC)}(F,G) \simeq \R\Mor_{\tau^*(\mC)}(F(0),G(0)) \times_{\R\Mor_\mV(\tau(A), \R\Mor_{\tau^*(\mC)}(F(0),G(1)))} \R\Mor_{\tau^*(\mC)}(F(1),G(1)) $$
$$ \simeq \tau(\L\Mor_{\mC}(F(0),G(0))) \times_{\R\Mor_\mV(\tau(A), \tau(\L\Mor_{\mC}(F(0),G(1))))} \tau(\L\Mor_{\mC}(F(1),G(1))) \simeq  $$ 
$$ \simeq \tau(\L\Mor_{\mC}(F(0),G(0))) \times_{\tau(\L\Mor_\mV(A, \L\Mor_{\mC}(F(0),G(1))))} \tau(\L\Mor_{\mC}(F(1),G(1))). $$ 
\end{proof}

\begin{corollary}\label{nat}

Let $\mA, \mC$ be $\infty$-categories and $F,G: S(\mA) \to \mC$ functors.
There is a canonical pullback square of $\infty$-categories:
$$\begin{xy}
\xymatrix{
\Mor_{\Fun(S(\mA), \mC)}(F,G) \ar[r] \ar[d] & \Mor_\mC(F(1),G(1)) \ar[d]
\\ 
\Mor_\mC(F(0),G(0)) \ar[r] & \Fun(\mA, \Mor_\mC(F(0),G(1))).
}
\end{xy}$$

\end{corollary}

\begin{corollary}\label{suspho}
Let $\phi: \mA \to \mB$ be a functor such that for every $\infty$-category $\mC$ the induced functor
$$\Fun(\mB,\mC) \to \Fun(\mA,\mC)$$ is fully faithful.
The induced functor
$\Fun(S(\mB),\mC) \to \Fun(S(\mA),\mC)$ is fully faithful.
   
\end{corollary}

\begin{proof}

For every functors $F,G: S(\mA) \to \mC$ the induced functor
$$ \Mor_{\Fun(S(\mB), \mC)}(F,G) \to \Mor_{\Fun(S(\mA), \mC)}(F \circ S(\phi),G \circ S(\phi))$$
identifies with the induced functor
$$ \Mor_\mC(F(0),G(0)) \times_{\Fun(\mA, \Mor_\mC(F(0),G(1)))} \Mor_\mC(F(1),G(1)) \to $$$$ \Mor_\mC(F(0),G(0)) \times_{\Fun(\mB, \Mor_\mC(F(0),G(1)))} \Mor_\mC(F(1),G(1)).$$
    
\end{proof}

\begin{corollary}
Let $\mC $ be an $\infty$-category and $n \geq 1.$
The induced functor
$$\Fun(\bD^{n-1},\mC) \to \Fun(\bD^n,\mC)$$ is fully faithful.
    
\end{corollary}

\begin{proof}
We proceed by induction on $n \geq 1.$
The induction step follows from \cref{suspho}.
We prove the induction start $n=1$. 
By \cref{nat} for every $A,B \in \mC$ the induced functor
$$\Mor_{\Fun(\bD^{0},\mC)}(A,B) \to \Mor_{\Fun(\bD^1,\mC)}(\id_A, \id_B) $$ identifies with the canonical equivalence
$$ \Mor_\mC(A,B) \simeq \Mor_\mC(A,B) \times_{\Mor_\mC(A,B)} \Mor_\mC(A,B).$$
   
\end{proof}

\begin{corollary}

Let $\mA$ be an $\infty$-category, $\mC$ a bioriented category and $F,G: S(\mA) \to \mC$ antioriented functors. There is a canonical pullback square of $\infty$-categories:
$$\begin{xy}
\xymatrix{
\R\Mor_{\boxtimes\Fun(S(\mA), \mC)}(F,G) \ar[r] \ar[d] & \L\Mor_\mC(F(1),G(1))^\op \ar[d]
\\ 
\L\Mor_\mC(F(0),G(0))^\op \ar[r] & \Fun^\oplax(A, \L\Mor_\mC(F(0),G(1)))^\op.
}
\end{xy}$$

\end{corollary}

\begin{proposition}\label{laxnat}

Let $\mA, \mC$ be $\infty$-categories and $F,G: S(\mA) \to \mC$ functors.
There is a canonical antioriented pullback square of $\infty$-categories:
$$\begin{xy}
\xymatrix{
\Mor_{\Fun^\lax(S(\mA), \mC)}(F,G) \ar[r] \ar[d] & \Mor_\mC(F(1),G(1)) \ar[d]
\\ 
\Mor_\mC(F(0),G(0)) \ar[r] & \Fun^\lax(\mA, \Mor_\mC(F(0),G(1))).
}
\end{xy}$$
    
\end{proposition}

\begin{proof}

The oriented pushout square
$$\begin{xy}
\xymatrix{
\mA \ar[r] \ar[d] & \bD^0 \ar[d]
\\ 
\bD^0 \ar[r] & S(\mA)
}
\end{xy}$$
gives rise to an antioriented pullback square
$$\begin{xy}
\xymatrix{
\Fun^\lax(S(\mA), \mC) \ar[r] \ar[d] & \mC \ar[d]
\\ 
\mC \ar[r] & \Fun^\lax(\mA,\mC).
}
\end{xy}$$
The latter gives rise to an antioriented pullback square
$$\begin{xy}
\xymatrix{
\Mor_{\Fun^\lax(S(\mA), \mC)}(F,G) \ar[r] \ar[d] & \Mor_\mC(F(1),G(1)) \ar[d]
\\ 
\Mor_\mC(F(0),G(0)) \ar[r] & \Mor_{\Fun^\lax(\mA,\mC)}(\underline{F(0)},\underline{G(1)}),
}
\end{xy}$$
where $\underline{(-)}$ indicates the constant functor.

The canonical pullback square

$$\begin{xy}
\xymatrix{
\Mor_\mC(F(0),G(1)) \ar[r] \ar[d] & \Fun^\oplax(\bD^1,\mC) \ar[d]
\\ 
\{(F(0),G(1))\} \ar[r] & \mC \times \mC
}
\end{xy}$$
gives rise to a pullback square
$$\begin{xy}
\xymatrix{
\Fun^\lax(\mA,\Mor_\mC(F(0),G(1))) \ar[r] \ar[d] & \Fun^\lax(\mA,\Fun^\oplax(\bD^1,\mC)) \simeq \Fun^\oplax(\bD^1,\Fun^\lax(\mA,\mC)) \ar[d]
\\ 
\{(F(0),G(1))\} \ar[r] & \Fun^\lax(\mA,\mC) \times \Fun^\lax(\mA,\mC)
}
\end{xy}$$
that provides an equivalence $$\Fun^\lax(\mA,\Mor_\mC(F(0),G(1))) \simeq \Mor_{\Fun^\lax(\mA,\mC)}(\underline{F(0)},\underline{G(1)}) . $$
\end{proof}

\begin{corollary}\label{laxnat2}

Let $\mA, \mC$ be $\infty$-categories and $F,G: S(\mA) \to \mC$ functors.
There is a canonical oriented pullback square of $\infty$-categories:
$$\begin{xy}
\xymatrix{
\Mor_{\Fun^\oplax(S(\mA), \mC)}(F,G) \ar[r] \ar[d] & \Mor_{\mC}(F(0),G(0)) \ar[d]
\\ 
\Mor_{\mC}(F(1),G(1)) \ar[r] & \Fun^\oplax(\mA, \Mor_{\mC}(F(0),G(1))).
}
\end{xy}$$
    
\end{corollary}

\begin{proof}
The functors $F,G: S(\mA) \to \mC$ give rise to functors
$F^\op,G^\op: S(\mA)^\op \simeq S(\mA^\co) \to \mC^\op.$
By \cref{laxnat} there is a canonical antioriented pullback square of $\infty$-categories:
\[
\xymatrix{
\Mor_{\Fun^{\lax}(S(\mA^\co), \mC^\op)}(G^{\op} ,F^\op) \ar[r] \ar[d] & \Mor_{\mC^\op}(G^{\op}(1),F^{\op}(1)) \ar[d]
\\ 
\Mor_{\mC^\op}(G^{\op}(0),F^{\op}(0)) \ar[r] & \Fun^\lax(\mA^\co, \Mor_{\mC^\op}(G^{\op}(0),F^{\op}(1))).
}
\]
The latter gives rise to an oriented pullback square of $\infty$-categories:
$$\begin{xy}
\xymatrix{
\Mor_{\Fun^\lax(S(\mA^\co), \mC^\op)}(G^\op,F^\op)^\co \ar[r] \ar[d] & \Mor_{\mC^\op}(G^\op(1),F^\op(1))^\co \ar[d]
\\ 
\Mor_{\mC^\op}(G^\op(0),F^\op(0))^\co \ar[r] & \Fun^\lax(\mA^\co, \Mor_{\mC^\op}(G^\op(0),F^\op(1)))^\co.
}
\end{xy}$$
The last square identifies with the following one:
$$\begin{xy}
\xymatrix{
\Mor_{\Fun^\lax(S(\mA^\co), \mC^\op)^\op}(F,G) \ar[r] \ar[d] & \Mor_{\mC}(F(0),G(0)) \ar[d]
\\ 
\Mor_{\mC}(F(1),G(1)) \ar[r] & \Fun^\oplax(\mA, \Mor_{\mC^\op}(G^\op(0),F^\op(1))^\co),
}
\end{xy}$$
which identifies with the following one:
$$\begin{xy}
\xymatrix{
\Mor_{\Fun^\oplax(S(\mA), \mC)}(F,G) \ar[r] \ar[d] & \Mor_{\mC}(F(0),G(0)) \ar[d]
\\ 
\Mor_{\mC}(F(1),G(1)) \ar[r] & \Fun^\oplax(\mA, \Mor_{\mC}(F(0),G(1))).
}
\end{xy}$$
\end{proof}

\begin{lemma}\label{diskcomp}

Let $n \geq 0.$
The $\infty$-category $\bD^n$ is compact in $\infty\Cat.$
    
\end{lemma}

\begin{proof}

We proceed by induction on $n \geq 0$.
For $n=0$ the functor $\Map_{\infty\Cat}(\bD^n,-): \infty\Cat \to \mS$ is the right adjoint of the embedding $\mS \subset \infty\Cat$
and preserves small filtered colimits.

We prove the induction step. Let $n \geq 1$. We assume the statement for $n-1$.
Let $\mJ$ be a small filtered $\infty$-category and 
$F: \mJ \to \infty\Cat$ a functor.
The induced map 
$$\xi: \colim(\Map_{\infty\Cat}(\bD^n,-) \circ F) \to \Map_{\infty\Cat}(\bD^n,\colim(F)) $$ is a map over
$$ \colim(\iota_0 \circ F)^{\times 2} \simeq \colim(\Map_{\infty\Cat}(\bD^0,-) \circ F)^{\times 2} \simeq \Map_{\infty\Cat}(\bD^0,\colim(F))^{\times 2} \simeq \iota_0(\colim(F))^{\times 2}. $$

Thus $\xi$ is an equivalence if it induces an equivalence on the fiber over every $A,B \in \iota_0(\colim(F)).$
Since $\mJ$ is filtered, there is a $j \in \mJ$ and $A',B' \in F(j) $ that are sent to $A,B \in \colim(F),$ respectively. 
Since the functor $\mJ_{j/} \to \mJ$ is cofinal, the functor
$\mJ_{j/} \to \mJ \xrightarrow{F} \infty\Cat $
lifts to a functor $F': \mJ_{j/} \to \infty\Cat_{\partial\bD^1/}$
along the forgetful functor $U: \infty\Cat_{\partial\bD^1/} \to \infty\Cat$ whose colimit is $(\colim(F),A,B).$

The map $\xi$ over $ \iota_0(\colim(F))^{\times 2} $
identifies with the canonical map
$$\xi': \colim(\Map_{\infty\Cat}(\bD^n,-) \circ U \circ F') \to \Map_{\infty\Cat}(\bD^n,\colim(U \circ F')) $$ over
$\iota_0(\colim(U \circ F'))^{\times 2}.$
The latter induces on the fiber over $A,B$ the canonical map
$$ \colim(\Map_{\infty\Cat_{\partial\bD^1 /}}(\bD^n,-) \circ F') \to \Map_{\infty\Cat_{\partial\bD^1 /}}(\bD^n,\colim(F')), $$
which identifies with the map
$$ \colim(\Map_{\infty\Cat}(\bD^{n-1},-) \circ \Mor \circ F') \to \Map_{\infty\Cat}(\bD^{n-1},\Mor_{\colim(F)}(A,B)). $$
The latter map is an equivalence by induction hypothesis and \cref{homfil}.
\end{proof}

\begin{theorem}\label{thm:oplaxinclusion}
Let $\mA$ and $ \mC$ be $\infty$-categories.
The canonical functors $$\Fun(\mA,\mC) \to \Fun^\lax(\mA,\mC),$$$$ \Fun(\mA,\mC) \to \Fun^\oplax(\mA,\mC)$$ are inclusions.
\end{theorem}

\begin{proof}
It suffice to prove the first claim. 
The functors $\Fun(-,\mC), \Fun^\lax(-,\mC): \infty\Cat^\op \to \infty\Cat $ send colimits to limits.
Since $\infty\Cat$ is generated by the disks, 
it suffices to see (1) for $\mA$ the $n$-disk for any $n \geq 0.$
We prove this by induction on $n \geq 0$.
For $n=0$ the functor of (1) identifies with the identity.
Next we prove (1) for $n=1.$
The canonical functor $\Fun(\bD^1,\mC) \to \Fun^\oplax(\bD^1,\mC)$ induces an equivalence on maximal subspaces.
So it suffices to see that it induces an inclusion on morphism $\infty$-categories.

By \cref{nat} for every morphisms $f: A \to B, g: X \to Y $ in $\mC$ 
there is a canonical equivalence
$$ \Mor_{\Fun(\bD^1,\mC)}(f,g) \simeq \Mor_\mC(B,Y) \times_{\Mor_\mC(A,Y)} \Mor_\mC(A,X).$$ 
This implies that the functor $\Fun(\bD^1,-)$ preserves small filtered colimits since $\bD^1$ is compact in $\infty\Cat$ by \cref{diskcomp} and small filtered colimits commute with pullbacks in $\infty\Cat$ since $\infty\Cat$ is compactly generated by construction.
The functor $\Fun^\lax(\bD^1,-): \infty\Cat \to \infty\Cat $ preserves small filtered colimits since oriented cubes are compact in $\infty\Cat$ (since oriented cubes are finite colimits of disks, which are compact by \cref{diskcomp}) and $\infty\Cat$ is generated under small colimits by the oriented cubes.
Since every $\infty$-category is the small filtered colimit 
of its underlying $m$-categories for $m \geq 0$, 
it suffices to see that the canonical functor $$\Fun(\bD^1,\mC) \to \Fun^\lax(\bD^1,\mC)$$ is an inclusion
for $\mC$ any $m$-category for $m \geq 0$.
We prove the latter statement by induction on $m 
\geq 0.$

For $m=0$ the functor $\Fun(\bD^1,\mC) \to \Fun^\lax(\bD^1,\mC)$ identifies with the identity of $\mC$ since the classifying spaces of the cubes are contractible.
We assume the statement holds for $m$ and let $\mC$ be an $m+1$-category.
The canonical functor $\Fun(\bD^{1},\mC) \to \Fun^\lax(\bD^{1},\mC)$ induces an equivalence on underlying spaces and induces on morphism $\infty$-categories for any morphisms $f: A \to B, g: X \to Y $ in $\mC$
the canonical functor 
$$\Mor_\mC(B,Y) \times \Mor_\mC(A,X) \times_{\Mor_\mC(A,Y)\times \Mor_\mC(A,Y)} \Mor_\mC(A,Y) \to $$$$ \Mor_\mC(B,Y) \times \Mor_\mC(A,X) \times_{\Mor_\mC(A,Y)\times \Mor_\mC(A,Y)} \Fun^\lax(\bD^1,\Mor_\mC(A,Y)).$$
The latter is an inclusion by induction hypothesis since
$\Mor_\mC(A,Y)$ is an $m$-category. 

So we have proven the case $n=1$ of the theorem.
We assume now the theorem holds for $\mA=\bD^n$ and prove the theorem for $\mA=\bD^{n+1}$.

The canonical functor $\Fun(\bD^{n+1},\mC) \to \Fun^\lax(\bD^{n+1},\mC)$
induces an equivalence on maximal subspaces.
Hence it suffices to see that for every functors $F,G:\bD^{n+1} \to \mC$
the induced functor
$$ \Mor_{\Fun(\bD^{n+1}, \mC)}(F,G) \to \Mor_{\Fun^\lax(\bD^{n+1}, \mC)}(F,G)$$
is an inclusion.
By \cref{nat} and \cref{laxnat} the latter functor identifies with the following functor
$$ \Mor_\mC(F(0),G(0)) \times \Mor_\mC(F(1),G(1)) \underset{\hspace{-1em}\Fun(\bD^n, \Mor_\mC(F(0),G(1)))\times \Fun(\bD^n, \Mor_\mC(F(0),G(1)))}{\times}\hspace{-1em}\Fun(\bD^n, \Mor_\mC(F(0),G(1))) $$$$ \to \Mor_\mC(F(0),G(0)) \times\Mor_\mC(F(1),G(1)) \underset{\hspace{-1em}\Fun^\lax(\bD^n, \Mor_\mC(F(0),G(1)))}{\times}\hspace{-1em}\Fun^\lax(\bD^1,\Fun^\lax(\bD^n, \Mor_\mC(F(0),G(1)))). $$
Since inclusions are stable under pullback, it suffices to see
that the canonical functor 
$$ \Fun(\bD^n, \Mor_\mC(F(0),G(1))) \to \Fun^\lax(\bD^1,\Fun^\lax(\bD^n, \Mor_\mC(F(0),G(1)))) $$
is an inclusion. Let $\mB:= \Mor_\mC(F(0),G(1)).$
The latter functor factors as the canonical functors 
$$ \Fun(\bD^n, \mB) \to \Fun^\lax(\bD^n, \mB) \to \Fun^\lax(\bD^1,\Fun^\lax(\bD^n, \mB)).$$
By induction hypothesis the canonical functor
$\Fun(\bD^{n},\mB) \to \Fun^\lax(\bD^{n},\mB)$
is an inclusion.
So setting $\mD:= \Fun^\lax(\bD^n, \mB)$
we need to see that the canonical functor
$\mD \to \Fun^\lax(\bD^{1},\mD)$
is an inclusion.
The latter factors as
\[
\mD \to \Fun(\bD^{1},\mD) \to \Fun^\lax(\bD^{1},\mD).
\]
The functor $\Fun(\bD^{1},\mD) \to \Fun^\lax(\bD^{1},\mD)$
is an inclusion by the case $\mA=\bD^1.$
The functor $\mD \to \Fun(\bD^{1},\mD) $ is an embedding
for any $\infty$-category $\mD$
since it induces on morphism $\infty$-categories between $X,Y \in \mD$
the canonical equivalence $\Mor_\mD(X,Y) \to \Mor_\mD(X,Y) \times_{\Mor_\mD(X,Y)} \Mor_\mD(X,Y).$
\end{proof}

We obtain the following:

\begin{theorem}\label{grayepi}

For every $n \geq 2$ and $\infty$-categories $\mA_1,...,\mA_n $
the canonical functor
\[
\mA_1 \boxtimes\cdots\boxtimes \mA_n \to \mA_1 \times\cdots\times \mA_n
\]
is an epimorphism.
\end{theorem}

\begin{proof}

We prove the statement by induction on $n \geq 2.$
Let $n \geq 2.$ We assume the statement for $n.$ 


If we have proven the induction start $n=2$, the composition
$$ (\mA_1\boxtimes\cdots\boxtimes\mA_n) \boxtimes\mA_{n+1} \to 
(\mA_1\boxtimes\cdots\boxtimes\mA_n) \times \mA_{n+1} \to (\mA_1\times\cdots\times\mA_n) \times \mA_{n+1}$$ is an epimorphism
by the induction hypothesis since epimorpisms are closed under composition and preserved by any functor preserving pushouts and so by the functor
$(-) \times \mA_{n+1}: \infty\Cat \to \infty\Cat.$

It remains to prove the induction start $n=2$.
The case $n=2$ is equivalent to the statement that the map 
$$\Map_{\infty\Cat}(\mA_1\times\mA_2,\mB)\simeq\Map_{\infty\Cat}(\mA_1,\Fun(\mA_2,\mB))\to \Map_{\infty\Cat}(\mA_1,\Fun^{\oplax}(\mA_2,\mB))\simeq\Map_{\infty\Cat}(\mA_1\boxtimes\mA_2,\mB)
$$ 
is a monomorphism. This follows from \cref{thm:oplaxinclusion}.
\end{proof}

By \cite[Theorem 3.21.]{heine2019restricted}
the identity functor refines to a lax monoidal functor
$(\infty\Cat,\times)\to(\infty\Cat,\boxtimes)$
from the cartesian monoidal structure to the Gray monoidal structure.
The latter induces a bioriented functor 
$i:\infty\scat$$\to\infty\fcat$.

\begin{corollary}\label{monomos}
The bioriented functor $i:\infty\scat$$\to\infty\fcat$ is a monomorphism of bioriented categories.
\end{corollary}

\begin{proof}
We need to see that $i$ induces an embedding on underlying spaces and monomorphisms on morphism objects.
The bioriented functor $i$ induces an equivalence on underlying spaces. It induces on morphism objects between
$X,Y \in \infty\Cat$ the following map
$$ \Map_{\infty\Cat}((-) \times X \times (-),Y) \to \Map_{\infty\Cat}((-) \boxtimes X \boxtimes (-),Y) $$
of presheaves on $\infty\Cat \times \infty\Cat.$
Since monomorphisms in the category of presheaves are formed object-wise, it suffices to see that for every $S,T \in \infty\Cat$ the following map is an embedding of spaces:
$$ \Map_{\infty\Cat}(S \times X \times T,Y) \to \Map_{\infty\Cat}(S \boxtimes X \boxtimes T,Y). $$
This holds by \cref{grayepi}.
\end{proof}
    
\subsection{$\infty$-categories as oriented categories}

\begin{lemma}
Let $K $ be a category and $\alpha: F \to G$ a map of functors $K \to \Cat$ that is objectwise fully faithful.
Assume there is an object $X \in K$ such that $\alpha_X: F(X) \to G(X)$
is an equivalence and for every $Y \in K$ there is a morphism $X \to Y$
in $K.$ The induced functor $\lim(\alpha)$ is an equivalence.

\end{lemma}

\begin{proof}

The induced functor $\lim(\alpha)$ is the functor
on cocartesian sections induced by the map $\beta: \mX \to \mY $ of cocartesian fibrations over $K$ classified by $\alpha.$
Since $\alpha$ is object-wise fully faithful, the functor
$\beta$ is fully faithful.
Hence also the functor on cocartesian sections induced by $\beta$ 
is fully faithful. So it remains to see that the functor on cocartesian sections induced by $\beta$ is essentially surjective.
In other words we have to see that any cocartesian section of
$\mY \to K$ lands in $\mX \subset \mY$. The induced section of $\mX \to K$ is then automatically cocartesian.
So we have to see that 
for every $T \in \lim(G)$ and every $Y \in K$ the image of
$T$ under the projection $\lim(G) \to G(Y) $ is in the image of
$\alpha_Y: F(Y) \to G(Y).$
By assumption there is an object $X \in K$ such that $\alpha_X: F((X) \to G(X)$
is an equivalence and there is a morphism $\phi: X \to Y$
in $K.$ 
Let $T'$ be the image of $T$ under the projection $\lim(G) \to G(X).$
The image of $T$ under the projection $\lim(G) \to G(Y) $ is 
the image of $T'$ under the map $ G(\phi): G(X) \to G(Y).$
So it suffices to see that $T'$ is in the image of
$\alpha_X: F(X) \to G(X). $
This holds since $\alpha_X: F(X) \to G(X)$ is an equivalence by assumption.
\end{proof}

\begin{corollary}\label{limi}

Let $\alpha: F \to G$ be a map of functors $\Delta \to \Cat$ that is objectwise fully faithful.
If $\alpha_{[0]}: F([0]) \to G([0])$ is an equivalence, the induced functor $\lim(\alpha)$ is an equivalence.

\end{corollary}

\begin{lemma}\label{fullyfaith}

Let $\bj: \mV \to \mW$ be a lax monoidal functor
that induces on underlying categories a fully faithful functor and
such that for every natural number $n > 1$ and $A_1, ..., A_n \in \mV$,
the canonical morphism $j(A_1) \ot ... \ot j(A_n) \to j(A_1 \ot ... \ot A_n)$ in $\mW$ is an epimorphism.
\begin{enumerate}[\normalfont(1)]\setlength{\itemsep}{-2pt}
\item The induced functor $\Alg^{\mathrm{nu}}(\mV) \to \Alg^{\mathrm{nu}}(\mW)$ is fully faithful.
The essential image precisely consists of the non-unital associative algebras $A$ in $\mW$ whose underlying object belong to $\mV$ and such that the multiplication morphism
$ \bj(A) \ot \bj(A) \to \bj(\A) $
in $\mW$ factors through the epimorphism $\bj(A) \ot \bj(A) \to \bj(A\ot A).$

\item If the canonical morphism $\tu_\mW \to \bj(\tu_\mV)$ is an epimorphism, the induced functor $\Alg(\mV) \to \Alg(\mW)$ is fully faithful.
The essential image precisely consists of the associative algebras $A$ in $\mW$ whose underlying object belong to $\mV$ and such that the multiplication morphism
$ \bj(A) \ot \bj(A) \to \bj(\A) $
in $\mW$ factors through the epimorphism $\bj(A) \ot \bj(A) \to \bj(A\ot A)$
and such that the unit $\tu_\mW \to \bj(A)$ factors as $\tu_\mW \to \bj(\tu_\mV) \to \bj(A). $
   
\end{enumerate}

\end{lemma}

\begin{remark}
Let $\bj: \mV \to \mW$ be a lax monoidal functor such that
for every $A,B \in \mV$ the canonical morphism $j(A) \ot j(B) \to j(A \ot B)$ in $\mW$ is an epimorphism, and for every $Z \in \mW$
the functor $(-)\ot j(Z): \mW \to \mW$ preserves pushouts.
For every natural $n \geq 2$ and $A_1, ..., A_n \in \mV$
the canonical morphism $\theta: j(A_1) \ot ... \ot j(A_n) \to j(A_1 \ot ... \ot A_n)$ in $\mW$ is an epimorphism.
This follows by induction, where the induction start $n=2$ is tautological.
The canonical morphism $\theta: j(A_1) \ot ... \ot j(A_{n+1}) \to j(A_1 \ot ... \ot A_{n+1})$ factors as
$$j(A_1) \ot ... \ot j(A_{n+1}) \xrightarrow{\theta \ot j(A_{n+1})} j(A_1 \ot ... \ot A_n) \ot j(A_{n+1}) \xrightarrow{\theta} j(A_1 \ot ... \ot A_n \ot A_{n+1}).$$
So the induction step follows from the fact that epimorphisms are closed under composition and preserved by any functor preserving pushouts,
and so by the functor $(-) \ot j(A_{n+1}): \mW \to \mW.$
    
\end{remark}

\begin{proof}

By \cite[Theorem 5.4.4.5]{lurie.higheralgebra} there is a canonical inclusion $\Alg(\mV) \subset \Alg^{\mathrm{nu}}(\mV)$
whose image is the subcategory of non-unital associative algebras $A$ in $\mV$ that admit a (necessarily unique) unit, i.e. there is a (necessarily unique) morphism $\eta: \tu_\mV \to A$ such that the compositions
$ \tu_\mV \ot A \to A \ot A \to A, A \ot \tu_\mV \to A \ot A \to A $ are the identities, and morphisms of non-unital associative algebras in $\mV$ preserving units.

The induced functor $ \Alg^{\mathrm{nu}}(\mV) \to \Alg^{\mathrm{nu}}(\mW)$
restricts to the functor $ \Alg(\mV) \to \Alg(\mW)$.
Moreover if the canonical morphism $\tu_\mW \to \bj(\tu_\mV)$ is an epimorphism,
a morphism of non-unital associative algebras in $\mV$ preserves the unit if and only if the image in $\mW$ preserves the unit.
Furthermore a non-unital associative algebra in $\mV$ 
admits a unit if and only if its image in $\mW$ admits a unit and this unit factors through the epimorphism $\tu_\mW \to \bj(\tu_\mV).$
Consequently (2) follows from (1).

We prove (1). Let $A,B \in \Alg^{\mathrm{nu}}(\mV).$ The induced map 
\begin{equation}\label{nk}
\map_{\Alg^{\mathrm{nu}}(\mV)}(A,B) \to \map_{\Alg^{\mathrm{nu}}(\mW)}(j(A),j(B))
\end{equation} 
identifies with the map
$$ \lim(\theta): \lim_{[n] \in \Delta} \map_\mV(A^{\ot n+1},B) \to \lim_{[n] \in \Delta} \map_\mW(\bj(A)^{\ot n+1},\bj(B))$$
on limits induced by the map
$$ \theta: \map_\mV(A^{\ot n+1},B) \to \map_\mW(\bj(A^{\ot n+1}),\bj(B)) \to \map_\mW(\bj(A)^{\ot n+1},\bj(B))$$
of functors $\Delta \to \mS.$
For every $[n] \in \Delta $ the map $\theta_{[n]}$ is induced by the morphism $ \bj(A)^{\ot n+1} \to \bj(A^{\ot n+1}) $
in $\mW,$ which is an epimorphism by assumption.
Since $\bj: \mV \to \mW$ induces on underlying categories a fully faithful functor, the map $\theta_{[n]}$ is an embedding of spaces
and $\theta_{[0]}$ is an equivalence.
So fully-faithfulness follows from \cref{limi} and the assumption that $\bj: \mV \to \mW$ induces on underlying categories a fully faithful functor so that $\theta_{[0]}$ is an equivalence.

We characterize the essential image. The assumptions imply that a non-unital associative algebra $A$ in $\mW$ belongs to $\mV$ if the underlying object of $A$ belongs to $\mV$ and for every
$n \geq 2$ the multiplication morphism $\bj(A)^{\ot n} \to \bj(A)$ 
factors through the epimorphism $\bj(A)^{\ot n} \to \bj(A^{\ot n})$.
By associativity we can moreover assume that $n=2$ in the latter condition.
\end{proof}

\begin{proposition}\label{enrinterchange}
Let $\mV,\mW$ be presentably monoidal categories and $\bj: \mV \to \mW$ a lax monoidal functor that preserves small colimits and induces on underlying categories a fully faithful functor.
Assume that for every $A,B \in \mV$ the canonical morphism $j(A) \ot j(B) \to j(A \ot B)$ in $\mW$ is an epimorphism and that the canonical morphism $\tu_\mW \to \bj(\tu_\mV)$ is an epimorphism.
\begin{enumerate}[\normalfont(1)]\setlength{\itemsep}{-2pt}

\item The induced functor $ {\mV\mathrm{-}\Cat} \to {\mW\mathrm{-}\Cat}$ of 2-categories is fully faithful.
The essential image precisely consists of the left $\mW$-enriched categories $\mC$
such that for every $X,Y \in \mC $ the morphism object $\L\Mor_\mC(X,Y)$ belongs to $\mV$ and for every $X, Y, Z \in \mC $ the composition morphism
$$ \bj(\L\Mor_\mC(Y,Z)) \ot \bj(\L\Mor_\mC(X,Y)) \to \bj(\L\Mor_\mC(X,Z)) $$
in $\mW$ factors through the epimorphism $$\bj(\L\Mor_\mC(Y,Z)) \ot \bj(\L\Mor_\mC(X,Y))\to \bj(\L\Mor_\mC(Y,Z) \ot \L\Mor_\mC(X,Y)), $$
and for every $X \in \mC$ the unit
$\tu \to \bj(\LMor_\mC(X,X))$ factors through the epimorphism $\tu_\mW \to \bj(\tu_\mV).$

\item Assume that $\mW$ is generated under small colimits by the essential image of $\mV.$
For every $\mC \in {\mV\mathrm{-}\Cat}$ the induced functor
$$ {_\mV\Fun}(\mC,\mV) \to {_\mW\Fun}(\bj_!(\mC),\bj_!(\mV)) \to {_\mW\Fun}(\bj_!(\mC),\mW) $$
is fully faithful.

\end{enumerate}
    
\end{proposition}

\begin{proof}

(1): We prove first that the induced functor ${\mV\mathrm{-}\Cat} \to {\mW\mathrm{-}\Cat}$ of 1-categories is fully faithful.
Let $\mC,\mD \in {\mV\mathrm{-}\Cat} $ and $X,Y$ the respective spaces of objects.
We would like to see that the induced map
\begin{equation}\label{enro}
\map_{{\mV\mathrm{-}\Cat}}(\mC,\mD) \to \map_{{\mW\mathrm{-}\Cat}}(\bj_!(\mC),\bj_!(\mD)) \end{equation}
is an equivalence.
By \cite[Remark 4.42]{HEINE2023108941} the forgetful functor ${\mV\mathrm{-}\Cat} \to \mS$ is a cartesian fibration whose fiber over $X$ we denote by ${\mV\mathrm{-}\Cat_X}.$
Map \ref{enro} is a map over $\map_\mS(Y,X)$ and induces on the fiber over any map $\phi: X \to Y$ the canonical map
$$ \map_{{\mV\mathrm{-}\Cat_X}}(\mC,\phi^*(\mD)) \to \map_{{\mW\mathrm{-}\Cat_X}}(\bj_!(\mC),\bj_!(\phi^*(\mD))).$$
By \cite[Corollary 4.28]{HEINE2023108941} there is a monoidal structure on the category $\Fun(X\times X,\mV)$ compatible with small colimits whose tensor unit is $\map_X \ot \tu$ and whose tensor product of $\mA, \mB \in \Fun(X\times X,\mV)$ sends $S,T \in X$ to $$ \colim_{R \in X} \L\Mor_\mA(R,T) \ot \L\Mor_\mB(S,R).$$
By \cite[Remark 4.44]{HEINE2023108941} there is a canonical equivalence ${\mV\mathrm{-}\Cat_X} \simeq \Alg(\Fun(X\times X,\mV))$
and for every left $\mV$-enriched category $\mC$ the multiplication $\mC \ot \mC \to \mC$ of the corresponding associative algebra in $\Fun(X\times X,\mV)$ is the 
morphism $ \colim_{R \in X} \L\Mor_\mC(R,T) \ot \L\Mor_\mC(S,R) \to \L\Mor_\mC(S,T)$
providing composition.

Moreover the lax monoidal functor $\bj: \mV \to \mW$ induces a
lax monoidal functor $ \bj_! :\Fun(X\times X,\mV) \to \Fun(X\times X,\mW). $
Since the lax monoidal functor $\bj: \mV \to \mW$ induces on underlying categories a fully faithful functor, $\bj_! : \Fun(X\times X,\mV) \to \Fun(X \times X,\mW) $
induces on underlying categories a fully faithful functor.

So the result follows from \cref{fullyfaith} if we prove that for 
every $\mA,\mB \in \Fun(X\times X,\mV)$ the canonical morphism $\bj_!(\mA) \ot \bj_!(\mB) \to \bj_!(\mA \ot \mB)$ in $\Fun(X\times X,\mW)$ is an epimorphism
and the canonical morphism $$\map_X \ot \tu_\mW \to \map_X \ot \bj(\tu_\mV) \simeq \bj_!(\map_X \ot \tu_\mV) $$ is an epimorphism.
Since epimorphisms are object-wise, this is equivalent to say that
the component at any $(S,T) \in X \times X $ of the canonical morphisms $$\bj_!(\mA) \ot \bj_!(\mB) \to \bj_!(\mA \ot \mB), \ \map_X \ot \tu_\mW \to \map_X \ot \bj(\tu_\mV) $$ are epimorphisms.
The component ar $(S,T) $ of the second natural transformation is an epimorphism
since epimorphisms are closed under small colimits in the category of arrows.
The component at $(S,T)$ of the first natural transformation identifies with the canonical morphism
$$\rho: \colim_{R \in X} \bj(\Mor_\mA(R,T)) \ot \bj(\Mor_\mB(S,R)) \to \colim_{R \in X} \bj(\Mor_\mA(R,T) \ot \Mor_\mB(S,R)) $$$$ \to
\bj(\colim_{R \in X} \Mor_\mA(R,T) \ot \Mor_\mB(S,R)) $$
in $\mW.$
By assumption $\bj$ preserves small colimits and the canonical morphism $$ \bj(\Mor_\mA(R,T)) \ot \bj(\Mor_\mB(S,R)) \to \bj(\Mor_\mA(R,T) \ot \Mor_\mB(S,R))$$
is an epimorphism.
Thus $\rho$ is an epimorphism since epimorphisms are closed under small colimits in the category of arrows.

We prove next that the induced functor $$ {\mV\mathrm{-}\Cat} \to {\mW\mathrm{-}\Cat}$$ of 2-categories is fully faithful.
We would like to see that the induced functor
\begin{equation}\label{enros}
\L\Mor_{{\mV\mathrm{-}\Cat}}(\mC,\mD) \to \L\Mor_{{\mW\mathrm{-}\Cat}}(\bj_!(\mC),\bj_!(\mD)) \end{equation}
is an equivalence.
By what we have shown the functor \ref{enros} induces an equivalence on underlying spaces and so is essentially surjective.
So it suffices to see that the functor \ref{enros} is fully faithful.
By \cite[Corollary 5.27]{heine2024bienriched} for every $F,G \in \L\Mor_{{\mV\mathrm{-}\Cat}}(\mC,\mD) $
the induced map
\begin{equation}\label{enrost}
\map_{\L\Mor_{{\mV\mathrm{-}\Cat}}(\mC,\mD)}(F,G) \to \map_{\L\Mor_{{\mW\mathrm{-}\Cat}}(\bj_!(\mC),\bj_!(\mD))}(\bj_!(F),\bj_!(G)) \end{equation}
identifies as an enriched end, which is computed as the totalization
$$ \lim(\theta): \lim_{[n] \in \Delta} \prod_{A_0, ..., A_n \in X} \map_\mV(\L\Mor_\mC(A_{n-1}, A_n) \ot ... \ot \L\Mor_\mC(A_0, A_1),  \L\Mor_\mD(F(A_0),G(A_n))) \to $$$$ \lim_{[n] \in \Delta} \prod_{A_0, ..., A_n \in X} \map_\mW(\bj(\L\Mor_\mC(A_{n-1}, A_n)) \ot ... \ot \bj(\L\Mor_\mC(A_0, A_1)),  \bj(\L\Mor_\mD(F(A_0)),G(A_n)))) $$
on limits induced by the map
$$ \theta: \prod_{A_0, ..., A_n \in X} \map_\mV(\L\Mor_\mC(A_{n-1}, A_n) \ot ... \ot \L\Mor_\mC(A_0, A_1),  \L\Mor_\mD(F(A_0),G(A_n))) \to $$$$ \prod_{A_0, ..., A_n \in X} \map_\mW(\bj(\L\Mor_\mC(A_{n-1}, A_n)) \ot ... \ot \bj(\L\Mor_\mC(A_0, A_1)),  \bj(\L\Mor_\mD(F(A_0),G(A_n)))) $$
of functors $\Delta \to \mS.$
For every $[n] \in \Delta $ the map $\theta_{[n]}$ is induced by the morphism $$\bj(\L\Mor_\mC(A_{n-1}, A_n)) \ot ... \ot \bj(\L\Mor_\mC(A_0, A_1)) \to \bj(\L\Mor_\mC(A_{n-1}, A_n) \ot ... \ot \L\Mor_\mC(A_0, A_1)) $$ in $\mW,$ which is an epimorphism by assumption.

Since $\bj: \mV \to \mW$ induces on underlying categories a fully faithful functor, the map $\theta_{[n]}$ is an embedding of spaces
and $\theta_{[0]}$ is an equivalence since $\bj$ preserves the tensor unit. Thus the map \ref{enrost} is an equivalence by \cref{limi}.
This proves (1).

We continue with (2).
We prove first that for every $A,B \in \mV$
the induced morphism $\bj(\L\Mor_\mV(A,B)) \to \L\Mor_\mW(\bj(A),\bj(B))$
in $\mW$ is a monomorphism.
Since $\mW$ is generated under small colimits by the essential image of
$\bj$ and monomorphisms are closed under small limits in the categort of arrows, it suffices to see that for every $Z \in \mV$
the induced map $$\map_\mW(\bj(Z), \bj(\L\Mor_\mV(A,B))) \to \map_\mW(\bj(Z), \L\Mor_\mW(\bj(A),\bj(B))) $$
is an embedding.
The latter map factors as 
$$\map_\mW(\bj(Z), \bj(\L\Mor_\mV(A,B)))\simeq \map_\mV(Z, \L\Mor_\mV(A,B)) \simeq \map_\mV(Z \ot A, B) \simeq \map_\mV(\bj(Z \ot A), \bj(B)) \to $$$$ \map_\mW(\bj(Z) \ot \bj(A),\bj(B)) \simeq  \map_\mW(\bj(Z), \L\Mor_\mW(\bj(A),\bj(B))) $$
induced by the epimorphism $\bj(Z) \ot \bj(A) \to \bj(Z \ot A)$
and so is an embedding.

Let $\mC \in {\mV\mathrm{-}\Cat}$ and $F,G: \mC \to \infty\Cat$ be functors.
As before, the induced map
$$ \map_{{_\mV\Fun}(\mC,\mV)}(F,G) \to \map_{{_\mW\Fun}(\bj_!(\mC),\mW)}(\bj \circ F,\bj \circ G) $$
identifies with the map $$ \lim(\theta): \lim_{[n] \in \Delta} \prod_{A_0, ..., A_n \in X} \map_\mV(\L\Mor_\mC(A_{n-1}, A_n) \ot ... \ot \L\Mor_\mC(A_0, A_1),  \L\Mor_\mV(F(A_0),G(A_n))) \to $$$$ \lim_{[n] \in \Delta} \prod_{A_0, ..., A_n \in X} \map_\mW(\bj(\L\Mor_\mC(A_{n-1}, A_n)) \ot ... \ot \bj(\L\Mor_\mC(A_0, A_1)), \L\Mor_\mW(j(F(A_0)),j(G(A_n)))) $$
on limits induced by the map
$$ \theta: \prod_{A_0, ..., A_n \in X} \map_\mV(\L\Mor_\mC(A_{n-1}, A_n) \ot ... \ot \L\Mor_\mC(A_0, A_1),  \L\Mor_\mV(F(A_0),G(A_n))) \to $$$$ \prod_{A_0, ..., A_n \in X} \map_\mW(\bj(\L\Mor_\mC(A_{n-1}, A_n)) \ot ... \ot \bj(\L\Mor_\mC(A_0, A_1)),  \L\Mor_\mW(j(F(A_0)), j(G(A_n)))) $$
of functors $\Delta \to \mS.$
For every $[n] \in \Delta $ the map $\theta_{[n]}$ is induced by the morphism $$\bj(\L\Mor_\mC(A_{n-1}, A_n)) \ot ... \ot \bj(\L\Mor_\mC(A_0, A_1)) \to \bj(\L\Mor_\mC(A_{n-1}, A_n) \ot ... \ot \L\Mor_\mC(A_0, A_1)) $$ in $\mW,$ which is an epimorphism by assumption, and the morphism
$$j(\L\Mor_\mV(F(A_0),G(A_n))) \to \L\Mor_\mW(j(F(A_0)), j(G(A_n))),$$
which is a monomorphism as we have proven.
Since $\bj: \mV \to \mW$ induces on underlying categories a fully faithful functor, the map $\theta_{[n]}$ is an embedding of spaces.
The map $\theta_{[0]}$ identifies with the map
$$ \prod_{A_0 \in X} \map_\mV(\tu_\mV, \L\Mor_\mV(F(A_0),G(A_0))) \to $$$$ \prod_{A_0 \in X} \map_\mW(\tu_\mW,  \L\Mor_\mW(j(F(A_0)), j(G(A_0)))) $$
induced by $j$ and the monomorphism $j(\L\Mor_\mV(F(A_0), G(A_0))) \to 
\L\Mor_\mW(\bj(F(A_0)), \bj(G(A_0))). $
The map $\theta_{[0]}$ identifies with the map
$$ \prod_{A_0 \in X} \map_\mV(F(A_0),G(A_0)) \to \prod_{A_0 \in X}
\map_\mW(j(F(A_0)), j(G(A_0))) $$
induced by $j$ and so is an equivalence. Thus the map \ref{enrost} is an equivalence by \cref{limi}.
This proves (2).
\end{proof}

By \cite[Theorem 3.21.]{heine2019restricted}
the identity functor refines to a lax monoidal functor
$(\infty\Cat,\times)\to(\infty\Cat,\boxtimes)$
from the cartesian monoidal structure to the Gray monoidal structure.
The latter induces right adjoint functors of 2-categories 
\begin{align}\label{functo1}
\infty\Cat \simeq {_{(\infty\Cat,\times)}\Cat} &\to {_{(\infty\Cat,\boxtimes)}\Cat}={\boxtimes\Cat}\\
\infty\Cat \simeq \Cat_{(\infty\Cat,\times)}&\to\Cat_{(\infty\Cat,\boxtimes)}={\Cat\boxtimes}\label{functo2}\\
{_{(\infty\Cat,\times)}\Cat_{(\infty\Cat,\times)}} &\to {_{_{(\infty\Cat,\boxtimes)}}\Cat_{(\infty\Cat,\boxtimes)}}={\boxtimes\Cat\boxtimes}
\end{align}
transfering the enrichment.

\begin{definition}
\begin{enumerate}
\item The oriented realization functor $||-||_\boxtimes: {\Cat\boxtimes}\to \infty\Cat $ is the left adjoint of the functor \ref{functo2}.

\item The antioriented realization functor $_\boxtimes||-||: {\boxtimes\Cat}\to \infty\Cat$ is the left adjoint of the functor \ref{functo1}.
\end{enumerate}
    
\end{definition}

\begin{theorem}\label{interchange}
\begin{enumerate}[\normalfont(1)]\setlength{\itemsep}{-2pt}
\item The induced functor $$ \infty\Cat\simeq\, _{\infty\Cat}\Cat \to {\boxtimes\Cat} $$ of 2-categories is fully faithful. The essential image precisely consists of the antioriented categories $\mC$ such that the antioriented interchange law holds strictly:
for every $X,Y,Z \in \mC$ and functors $\bD^n \to \L\Mor_\mC(Y,Z)$ and $\bD^m \to \L\Mor_\mC(X,Y)$
the functor $$\bD^n \boxtimes \bD^m \to \L\Mor_\mC(Y,Z) \boxtimes \L\Mor_\mC(X,Y) \to \L\Mor_\mC(X,Z)$$ factors through the functor
$\bD^n \boxtimes \bD^m \to \bD^n \times \bD^m.$ 
\item The induced functor $$\infty\Cat\simeq\Cat_{\infty\Cat} \to {\Cat\boxtimes} $$ of 2-categoriesis fully faithful. The essential image precisely consists of the oriented categories $\mC$ such that the oriented interchange law holds strictly:
for every $X,Y,Z \in \mC$ and functors $\bD^n \to \R\Mor_\mC(Y,Z)$ and $\bD^m \to \R\Mor_\mC(X,Y)$
the functor $$\bD^n \boxtimes \bD^m \to \R\Mor_\mC(Y,Z) \boxtimes \R\Mor_\mC(X,Y) \to \R\Mor_\mC(X,Z)$$ factors through the functor
$\bD^n \boxtimes \bD^m \to \bD^n \times \bD^m.$ 

\item The induced functor $$_{\infty\Cat}\Cat_{\infty\Cat} \to {\boxtimes \Cat\boxtimes} $$ of 2-categories is fully faithful.
\end{enumerate}   
\end{theorem}

\begin{proof}

In view of \cref{grayepi} we can apply \cref{enrinterchange} (1) to the lax monoidal functor 
$(\infty\Cat,\times)\to(\infty\Cat,\boxtimes)$ lifting the identity to obtain the result.
\end{proof}

\cref{enrinterchange} (2) specializes to the following:



\begin{theorem}\label{presheaves}
Let $\mD \in \infty\Cat$.
In the following diagram of 1-categories the functors are embeddings:
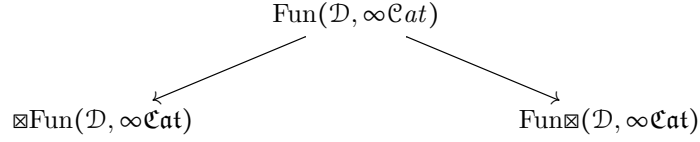
\[
\xymatrix{
& \Fun(\mD,\infty\scat) \ar[ld]\ar[rd] &\\
{\boxtimes\Fun}(\mD,\infty\fcat)  & & {\Fun\boxtimes}(\mD,\infty\fcat)
}
\]
 
\end{theorem}

We also obtain the following:
\begin{corollary}
There is a pullback square of 2-categories,
where all functors are embeddings:
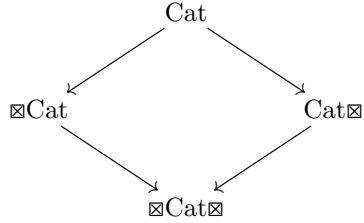
\[
\xymatrix{
& \Cat\ar[ld]\ar[rd] &\\
\boxtimes\Cat\ar[rd] & & \Cat\boxtimes\ar[ld]\\
& {\boxtimes\Cat\boxtimes} &
}
\]
\end{corollary}

In the following we use the methods of this section to compute morphism $\infty$-categories in the Gray tensor product of two suspensions.

\begin{notation}
Let $\mA,\mB$ be $\infty$-categories.
We set $$ \partial(S(\mA) \times S(\mB)):= (S(\mB) \vee S(\mA)) \coprod_{S(\emptyset)} (S(\mA) \vee S(\mB)).$$
    
\end{notation}

\begin{theorem}\label{susfor}
Let $\mA,\mB$ be $\infty$-categories.
There is a canonical pushout square of $\infty$-categories
$$\begin{xy}
\xymatrix{
S(\mA \boxtimes \partial\bD^1 \boxtimes \mB) \ar[d] \ar[r] & S(\mA \boxtimes \bD^1 \boxtimes \mB) \ar[d]
\\ 
\partial(S(\mA) \times S(\mB)) \ar[r] & S(\mA) \boxtimes S(\mB).
}
\end{xy}$$
\end{theorem}

\begin{proof}

We have to prove that there is a canonical pushout square:
$$\begin{xy}
\xymatrix{
S(\mA \boxtimes \partial\bD^1 \boxtimes \mB) \ar[d] \ar[r] & S(\mA \boxtimes \bD^1 \boxtimes \mB) \ar[d]
\\ 
(S(\mB) \vee S(\mA)) \coprod_{S(\emptyset)} (S(\mA) \vee S(\mB)) \ar[r] & S(\mA) \boxtimes S(\mB).
}
\end{xy}$$

We prove that for every $\infty$-category $\mC$ the induced map
$$\rho: \Map_{\infty\Cat}(S(\mA), \Fun^\oplax(S(\mB),\mC)) \simeq \Map_{\infty\Cat}(S(\mA) \boxtimes S(\mB),\mC) \to $$
$$ \Map_{\infty\Cat}(S(\mA \boxtimes \bD^1 \boxtimes \mB),\mC) \times_{\Map_{\infty\Cat}(S(\mA \boxtimes \partial\bD^1 \boxtimes \mB),\mC)} $$$$ (\Map_{\infty\Cat}(S(\mB),\mC) \times_{\mC}  \Map_{\infty\Cat}(S(\mA),\mC)) \times_{(\mC \times \mC)} (\Map_{\infty\Cat}(S(\mA),\mC) \times_{\mC} \Map_{\infty\Cat}(S(\mB),\mC))$$
is an equivalence. The map $\rho$ is a map over
$\Map_{\infty\Cat}(S(\mB),\mC) \times \Map_{\infty\Cat}(S(\mB),\mC), $ and so it suffices to see that $\rho$ induces an equivalence on the fiber over any pair $(\alpha,\beta) \in \Map_{\infty\Cat}(S(\mB),\mC) \times \Map_{\infty\Cat}(S(\mB),\mC).$
The map $\rho$ induces on the fiber over $(\alpha,\beta)$ the map
$$\rho': \Map_{\infty\Cat}(\mA, \Mor_{\Fun^\oplax(S(\mB),\mC)}(\alpha,\beta)) \simeq \Map_{\infty\Cat_{\partial\bD^1 /}}(S(\mA), \Fun^\oplax(S(\mB),\mC)) \to $$
$$ \Map_{\infty\Cat_{\partial\bD^1 /}}(S(\mA \boxtimes \bD^1 \boxtimes \mB),\mC) \times_{\Map_{\infty\Cat_{\partial\bD^1 /}}(S(\mA \boxtimes \partial\bD^1 \boxtimes \mB),\mC)}(\Map_{\infty\Cat_{\partial\bD^1 /}}(S(\mA),\mC) \times \Map_{\infty\Cat_{\partial\bD^1 /}}(S(\mA),\mC))$$
$$ \simeq R:= \Map_{\infty\Cat}(\mA \boxtimes \bD^1 \boxtimes \mB,\Mor_\mC(\alpha(0),\beta(1))) \times_{\Map_{\infty\Cat}(\mA \boxtimes \partial\bD^1 \boxtimes \mB,\Mor_\mC(\alpha(0),\beta(1)))}$$$$ (\Map_{\infty\Cat}(\mA,\Mor_\mC(\alpha(1),\beta(1))) \times \Map_{\infty\Cat}(\mA,\Mor_\mC(\alpha(0),\beta(0))))$$

Let $$T:= \Mor_\mC(\alpha(1),\beta(1)) \times \Mor_\mC(\alpha(0),\beta(0)) \times_{\Fun^\oplax(\partial\bD^1 \boxtimes \mB,\Mor_\mC(\alpha(0),\beta(1)))} \Fun^\oplax(\bD^1 \boxtimes \mB,\Mor_\mC(\alpha(0),\beta(1))).$$
By \cref{laxnat} there is a canonical equivalence
$ \Mor_{\Fun^\oplax(S(\mB),\mC)}(\alpha,\beta) \simeq T.$
The map $\rho'$ identifies with the map
$$ \Map_{\infty\Cat}(\mA, \Mor_{\Fun^\oplax(S(\mB),\mC)}(\alpha,\beta)) \simeq \Map_{\infty\Cat}(\mA, T) \simeq R.$$
\end{proof}

\begin{corollary}\label{corsusform}
Let $\mA,\mB$ be $\infty$-categories.
There is a canonical pushout square of $\infty$-categories
$$\begin{xy}
\xymatrix{
S(\mA \times \partial\bD^1 \times \mB) \ar[d] \ar[r] & S(\mA \times \partial\bD^1 \times \mB \coprod_{\mA \boxtimes \partial\bD^1 \boxtimes \mB} \mA \boxtimes \bD^1 \boxtimes \mB) \ar[d]
\\ 
\partial(S(\mA) \times S(\mB)) \ar[r] & S(\mA) \boxtimes S(\mB).
}
\end{xy}$$\end{corollary}

\begin{proof}

There is a canonical diagram of pushout squares of $\infty$-categories since the suspension $S: \infty\Cat \to \infty\Cat$ preserves small weakly contractible colimits and so pushouts:
$$\begin{xy}
\xymatrix{
S(\mA \boxtimes \partial\bD^1 \boxtimes \mB) \ar[d] \ar[r] & S(\mA \boxtimes \bD^1 \boxtimes \mB) \ar[d]
\\
S(\mA \times \partial\bD^1 \times \mB) \ar[r] & S(\mA \times \partial\bD^1 \times \mB \coprod_{\mA \boxtimes \partial\bD^1 \boxtimes \mB} \mA \boxtimes \bD^1 \boxtimes \mB).
}
\end{xy}$$
So the result follows from \cref{susfor}    
\end{proof}

\begin{corollary}\label{tensorsuspensionformula}
Let $\mA,\mB$ be $\infty$-categories.
There is a canonical pushout square of $\infty$-categories
$$\begin{xy}
\xymatrix{
S(\mA \times \mB) \coprod S(\mA \times \mB) \ar[d] \ar[r] & S(\mA \boxtimes \bD^1 \boxtimes \mB) \ar[d]
\\ 
(S(\mB) \vee S(\mA)) \coprod (S(\mA) \vee S(\mB)) \ar[r] & S(\mA) \boxtimes S(\mB),
}
\end{xy}$$
where both vertical functors are full embeddings.
In particular, the right vertical functor induces an equivalence
$$ \mA \times \partial\bD^1 \times \mB \coprod_{\mA \boxtimes \partial\bD^1 \boxtimes \mB} \mA \boxtimes \bD^1 \boxtimes \mB \simeq \Mor_{S(\mA) \boxtimes S(\mB)}((0,0),(1,1)). $$
    
\end{corollary}

\begin{proof}

There is a canonical pushout square of $\infty$-categories
$$\begin{xy}
\xymatrix{
S(\A \times \mB) \coprod S(\A \times \mB) \ar[d] \ar[r] & S(\A \times \mB) \coprod_{S(\emptyset)} S(\A \times \mB) \simeq S(\mA \times \partial\bD^1 \times \mB) \ar[d]
\\ 
(S(\mB) \vee S(\mA)) \coprod (S(\mA) \vee S(\mB)) \ar[r] & (S(\mA) \vee S(\mB)) \coprod_{S(\emptyset)} (S(\mB) \vee S(\mA))= \partial(S(\mA) \times S(\mB)).
}
\end{xy}$$
So the result follows from \cref{corsusform}.
\end{proof}

\begin{corollary}

Let $\mA,\mB$ be $\infty$-categories.
There is a canonical pushout square of $\infty$-categories
$$\begin{xy}
\xymatrix{
S(\mA \times \partial\bD^1 \times \mB \coprod_{\mA \boxtimes \partial\bD^1 \boxtimes \mB} \mA \boxtimes \bD^1 \boxtimes \mB) \ar[d] \ar[r] &  S(\mA \times \mB) \ar[d] 
\\ 
S(\mA) \boxtimes S(\mB) \ar[r] & S(\mA) \times S(\mB).
}
\end{xy}$$

\end{corollary}

\begin{proof}

There is a canonical pushout square of $\infty$-categories
$$\begin{xy}
\xymatrix{
S(\mA \times \mB) \coprod S(\mA \times \mB) \ar[d] \ar[r] & S(\mA \times \mB) \ar[d]
\\ 
(S(\mB) \vee S(\mA)) \coprod (S(\mA) \vee S(\mB)) \ar[r] & S(\mA) \times S(\mB).
}
\end{xy}$$
Thus the claim follows from the pasting law.
\end{proof}

\begin{corollary}

There is a canonical pushout square of $\infty$-categories
$$\begin{xy}
\xymatrix{
S(\bD^1 \times \bD^1) \coprod S(\bD^1 \times \bD^1) \ar[d] \ar[r] & S(\cube^3) \ar[d]
\\ 
(\bD^2 \vee \bD^2) \coprod (\bD^2 \vee \bD^2) \ar[r] & \bD^2 \boxtimes \bD^2.
}
\end{xy}$$
    
\end{corollary}

\begin{corollary}

There is a canonical pushout square of $\infty$-categories
$$\begin{xy}
\xymatrix{
S(\bD^1 \times \bD^1) \coprod S(\bD^1 \times \bD^1) \ar[d] \ar[r] & S(\bD^1 \times \bD^1 \times \bD^1) \ar[d]
\\ 
(\bD^2 \vee \bD^2) \coprod (\bD^2 \vee \bD^2) \ar[r] & \tau_2(\bD^2 \boxtimes \bD^2).
}
\end{xy}$$
\end{corollary}

\subsection{A universal property of the oriented category of oriented spaces}
In order to state our next results, recall \cref{Thetasgen}.
The antioriented realization functor $_\boxtimes||-||: {\boxtimes\Cat}\to \infty\Cat$ induces a functor
$\Theta(\cube,\boxtimes)\to\Theta(\cube,\times).$

Similarly, the oriented realization functor $||-||_\boxtimes: {\Cat\boxtimes}\to \infty\Cat$ induces a functor
$\Theta(\bar{\cube},\bar{\boxtimes})\to\Theta(\bar{\cube},\times).$

\begin{theorem}\label{geomorient}
\begin{enumerate}[\normalfont(1)]\setlength{\itemsep}{-2pt}
\item The full subcategory $$ \Theta(\cube, \boxtimes) \subset {\boxtimes\Cat} $$ is dense.
A presheaf on $\Theta(\cube, \boxtimes) $ is in the image of the $\Theta(\cube, \boxtimes)$-nerve functor if and only if it satisfies the Segal condition and all morphism presheaves on $\cube$ belong to the essential image of the oriented cubical nerve, i.e. are oriented cubical spaces satisfying the Segal condition.
Moreover a presheaf on $\Theta(\cube, \boxtimes) $ is in the image of the $\Theta(\cube, \boxtimes)$-nerve functor if and only if it it factors through the functor $\Theta(\cube,\boxtimes)\to\Theta(\cube,\times).$

\item The full subcategory $$ \Theta(\bar{\cube}, \bar{\boxtimes}) \subset {\Cat\boxtimes} $$ is dense.
A presheaf on $\Theta(\bar{\cube}, \bar{\boxtimes}) $ is in the image of the $\Theta(\bar{\cube}, \bar{\boxtimes})$-nerve functor if and only if it satisfies the Segal condition and all morphism presheaves on $\bar{\cube}$ belong to the essential image of the antioriented cubical nerve, i.e. are antioriented cubical spaces satisfying the Segal condition.
A presheaf on $\Theta(\bar{\cube}, \bar{\boxtimes}) $ is in the image of the $\Theta(\bar{\cube}, \bar{\boxtimes})$-nerve functor if and only if it it factors through the functor $\Theta(\bar{\cube}, \bar{\boxtimes})\to\Theta(\bar{\cube},\times).$
    
\end{enumerate}
\end{theorem}

\begin{proof}
We prove (2). The proof of (1) is similar.
The density claim follows from \cref{denseinherited} and the density of the oriented cubes \cite[Theorem 2.8.12.]{GepnerHeine2026}.
It remains to see that a presheaf on $\Theta(\cube, \boxtimes) $ is in the image of the $\Theta(\cube, \boxtimes)$-nerve functor if and only if it it factors through the functor $\Theta(\cube,\boxtimes)\to\Theta(\cube,\times).$

Since by definition the oriented realization functor $||-||_\boxtimes: {\Cat\boxtimes}\to \infty\Cat$ induces a functor
$\theta: \Theta(\bar{\cube},\bar{\boxtimes}) \to \Theta(\bar{\cube},\times),$
there is a commutative square of left adjoints
$$\begin{xy}
\xymatrix{
\mP(\Theta(\bar{\cube},\bar{\boxtimes})) \ar[r]^{\theta_!} \ar[d] & \mP(\Theta(\bar{\cube},\times)) \ar[d]
\\ 
{\Cat\boxtimes} \ar[r]^{||-||_\boxtimes} & \infty\Cat.
}
\end{xy}$$
The latter induces a commutative square between the corresponding right adjoints:
$$\begin{xy}\label{folz}
\xymatrix{
\infty\Cat \ar[r] \ar[d] & {\Cat\boxtimes}\ar[d]
\\ 
\mP(\Theta(\bar{\cube},\times)) \ar[r]^{\theta^*}  & \mP(\Theta(\bar{\cube},\bar{\boxtimes})).
}
\end{xy}$$

By \cref{theta} the full subcategory $\Theta \subset \infty\Cat$ is dense and so also the full subcategory $ \Theta \subset \Theta(\bar{\cube}, \times)$ is dense in $\infty\Cat$.
Hence the vertical functors in the commutative square \ref{folz} are fully faithful. Since the functor
$\theta: \Theta(\bar{\cube},\bar{\boxtimes}) \to \Theta(\bar{\cube},\times)$ is essentially surjective, we find that the commutative square \ref{folz} is a pullback square.
This implies the result.
\end{proof}

\begin{remark}
By \cref{dua} there exists an equivalence $ \bar{\cube}^n \simeq \cube^n $ for every $\n \geq 0$ and therefore $ \bar{\cube} \simeq \cube$, but this is not true in general, and fails for the orientals.
Moreover we remark that the involution
$(-)^\co: \infty\Cat \to \infty\Cat$ restricts to an involution
$ \cube \simeq \bar{\cube}$, which gives rise to an equivalence
$ \Theta(\cube,\boxtimes) \simeq \Theta(\bar{\cube},\bar{\boxtimes}). $
Thus these test categories are canonically equivalent, so any result for oriented categories implies a result for antioriented categories, and vice-versa, but the results are not the same --- one must carefully keep track of all the involutions.
\end{remark}

\begin{corollary}\label{restind}
Restriction and left Kan extension along the functor of $1$-categories $\Theta(\bar{\cube},\bar{\boxtimes})\to\Theta(\bar{\cube},\times)$ descends to a localization
\[
{\Cat\boxtimes}\leftrightarrows\infty\Cat.
\]
Moreover, the left adjoint agrees with the oriented realization functor $||-||_\boxtimes.$
\end{corollary}

Given a presheaf $X$ on $\Theta(\bar{\cube}, \bar{\boxtimes})$ and objects $S$ and $T$ in $X(\ast)$, we write $\RMor_X(S,T)$ for the presheaf on $\cube$ defined by the rule
\[
\xymatrix{
\Map_{\infty\Cat}(\cube^n,\RMor_X(S,T))\ar[d]\ar[r] & X(S(\cube^n))\ar[d]\\
(S,T)\ar[r] & X(\ast)\times X(\ast)
}
\]
When $X$ is the $\Theta(\bar{\cube}, \bar{\boxtimes})$-nerve of an oriented category $\mC$ with objects $S$ and $T$, this computes $\RMor_\C(S,T)$.
In particular, if $X$ is a presheaf on $\Theta(\bar{\cube}, \bar{\boxtimes})$ satisfying the Segal condition, then $\RMor_X(S,T)$ is a presheaf on $\cube$ satisfying the Segal condition, and therefore may be regarded as an $\infty$-category.

\begin{definition}
Let $X$ be a presheaf on $\Theta(\bar{\cube}, \bar{\boxtimes})$ satisfying the Segal condition.
\begin{enumerate}[\normalfont(1)]\setlength{\itemsep}{-2pt}
    \item
    Let $A \in \infty\Cat$ and  $B \in X(\ast).$
A {\em tensor} of $A$ and $B$ is an object $B \ot A \in X(\ast)$ such that there is an object $ A \to \R\Mor_X(B,B \ot A) \in X({\bD^1})$ and for every $C \in X(\ast)$ the induced functor
$$\R\Mor_X(B \ot A,C) \to \Fun^\lax(A, \R\Mor_X(B,C)) $$
is an equivalence.
    \item
The presheaf $X$ {\em admits tensors} if for every $A \in \infty\Cat, B \in X(\ast)$ there is a tensor $B \ot A \in X(\ast)$.
    \item
If $Y$ is another presheaf on $\Theta(\bar{\cube}, \bar{\boxtimes})$ satisfying the Segal condition, and $X$ and $Y$ admits tensors, then a map $X\to Y$ {\em preserves tensors} if the canonical morphism comparing the tensors is an equivalence.
\end{enumerate}
\end{definition}

Similarly, we define when a presheaf on $\Theta(\cube, \boxtimes)$ and $\Theta(\cube, \times)$ admit tensors.

We characterize the oriented category of oriented spaces,
the antioriented category of antioriented spaces and the $\infty$-category of $\infty$-categories by the following universal properties:

\begin{theorem}\label{charas}
\begin{enumerate}[\normalfont(1)]\setlength{\itemsep}{-2pt}
\item Let $X$ be a presheaf on $\Theta(\bar{\cube}, \bar{\boxtimes})$ satisfying the Segal condition and that is tensored.
The map $$ \Map'_{\mP(\Theta(\bar{\cube}, \bar{\boxtimes}))}(\N_{\Theta(\bar{\cube}, \bar{\boxtimes})}(\infty\fcat_{\mid \boxtimes}),X) \to \Map_{\mP(\Theta(\bar{\cube}, \bar{\boxtimes}))}(\N_{\Theta(\bar{\cube}, \bar{\boxtimes})}(\ast),X) \simeq X(\ast) $$ induced by the oriented functor $\ast \to \infty\fcat_{\mid \boxtimes} $ taking the final $\infty$-category, is an equivalence,
where the left hand side is the subspace of maps of presheaves $ \N_{\Theta(\bar{\cube}, \bar{\boxtimes})}(\infty\fcat_{\mid \boxtimes}) \to X $ on $\Theta(\bar{\cube}, \bar{\boxtimes})$ preserving tensors.

\item Let $X$ be a presheaf on $\Theta(\cube, \boxtimes)$ satisfying the Segal condition and that is tensored.
The map $$ \Map'_{\mP(\Theta(\cube, \boxtimes))}(\N_{\Theta(\cube, \boxtimes)}(_{\boxtimes \mid}\infty\fcat),X) \to \Map_{\mP(\Theta(\cube, \boxtimes))}(\N_{\Theta(\cube, \boxtimes)}(\ast),X) \simeq X(\ast) $$ induced by the antioriented functor $\ast \to {_{\boxtimes \mid}\infty\fcat} $ taking the final $\infty$-category, is an equivalence,
where the left hand side is the subspace of maps of presheaves $ \N_{\Theta(\cube, \boxtimes)}(_{\boxtimes \mid}\infty\fcat) \to X $ on $\Theta(\cube, \boxtimes)$ preserving tensors.

\item Let $X$ be a presheaf on $\Theta(\cube, \times)$ satisfying the Segal condition and that is tensored.
The map $$ \Map'_{\mP(\Theta(\cube, \times))}(\N_{\Theta(\cube, \times)}(\infty\Cat),X) \to \Map_{\mP(\Theta(\cube, \times))}(\N_{\Theta(\cube, \times)}(\ast),X) \simeq X(\ast) $$ induced by the functor $\ast \to \infty\Cat $ taking the final $\infty$-category, is an equivalence,
where the left hand side is the subspace of maps of presheaves $ \N_{\Theta(\cube, \times)}(\infty\Cat) \to X $ on $\Theta(\cube, \times)$ preserving tensors.

\end{enumerate}

\end{theorem}

We finish this section by the following proposition comparing bienrichment in the tensor product and Gray tensor product:

\begin{proposition}
There is a cartesian square
\[
\xymatrix{
& _{(\infty\Cat,\times)}\Cat_{(\infty\Cat,\times)}\ar[ld]\ar[rd] &\\
_{(\infty\Cat,\boxtimes)}\Cat_{(\infty\Cat,\times)}\ar[rd] & & _{(\infty\Cat,\times)}\Cat_{(\infty\Cat,\boxtimes)}\ar[ld]\\
& _{(\infty\Cat,\boxtimes)}\Cat_{(\infty\Cat,\boxtimes)} &.
}
\]
\end{proposition}

\begin{proof}

It suffices to prove that the commutative square of monoidal categories and lax monoidal functors
\[
\xymatrix{
& (\infty\Cat, \times) \ot (\infty\Cat, \times) \ar[ld]\ar[rd] &\\
(\infty\Cat, \boxtimes) \ot (\infty\Cat, \times)  \ar[rd] & & (\infty\Cat, \times) \ot (\infty\Cat, \boxtimes) \ar[ld]\\
& (\infty\Cat, \boxtimes) \ot (\infty\Cat, \boxtimes) &
}
\]
is a pullback square.
The latter commutative square induces on underlying categories a constant diagram and is the restriction along lax monoidal embeddings of the following commutative square of monoidal categories and lax monoidal functors
\[
\xymatrix{
& \mP((\infty\Cat, \times) \times (\infty\Cat, \times))\ar[ld]\ar[rd] &\\
\mP((\infty\Cat, \boxtimes) \times (\infty\Cat, \times))  \ar[rd] & & \mP((\infty\Cat, \times) \times (\infty\Cat, \boxtimes)) \ar[ld]\\
& \mP((\infty\Cat, \boxtimes) \times (\infty\Cat, \boxtimes)) &
}
\]
The latter commutative square induces on underlying categories a constant diagram. So it suffices that the latter square of monoidal categories and lax monoidal functors induces a pullback square on multi-morphism spaces. To check this, we can reduce to representables since all tensor products preserve component-wise small colimits and representable presheaves are complete compact. Consequently, it is enough to see that the commutative square of monoidal categories and lax monoidal functors
\[
\xymatrix{
& (\infty\Cat, \times) \times (\infty\Cat, \times) \ar[ld]\ar[rd] &\\
(\infty\Cat, \boxtimes) \times (\infty\Cat, \times)  \ar[rd] & & (\infty\Cat, \times) \times (\infty\Cat, \boxtimes) \ar[ld]\\
& (\infty\Cat, \boxtimes) \times (\infty\Cat, \boxtimes) &
}
\]
is a pullback square. This is evident since the commutative square in the first component and the commutative square in the second component of the product, are both pullback squares.    
\end{proof}

\vspace{.25cm}
\bibliographystyle{plain}
\bibliography{mainbib}

\end{document}